\documentclass[12pt]{amsart}
\usepackage{amsmath,latexsym,amsfonts,amssymb,amsthm}
\usepackage{geometry}
\usepackage{hyperref}
\usepackage{mathrsfs}
\usepackage{graphicx,color}
\usepackage[alphabetic]{amsrefs}
\usepackage[all,cmtip]{xy}
\usepackage{bbm}
\usepackage{comment}
\newtheorem{prop}{Proposition}[section]
\newtheorem{thm}[prop]{Theorem}
\newtheorem{cor}[prop]{Corollary}

\newtheorem{lem}[prop]{Lemma}

\theoremstyle{definition}

\newtheorem{defn}[prop]{Definition}
\newtheorem{expl}[prop]{Example}
\newtheorem{rem}[prop]{\it Remark}

\newtheorem*{claim*}{Claim}
\newtheorem{lemdefn}[prop]{Lemma-Definition}

\newcommand{\bP}{\mathbb{P}}
\newcommand{\bC}{\mathbb{C}}
\newcommand{\bR}{\mathbb{R}}
\newcommand{\bA}{\mathbb{A}}
\newcommand{\bQ}{\mathbb{Q}}
\newcommand{\bZ}{\mathbb{Z}}
\newcommand{\bN}{\mathbb{N}}

\newcommand{\bG}{\mathbb{G}}

\newcommand{\tS}{\Tilde{S}}

\newcommand{\cO}{\mathcal{O}}

\newcommand{\cF}{\mathcal{F}}
\newcommand{\cG}{\mathcal{G}}
\newcommand{\cJ}{\mathcal{J}}

\newcommand{\cB}{\mathcal{B}}

\newcommand{\cW}{\mathcal{W}}

\newcommand{\fa}{\mathfrak{a}}
\newcommand{\fb}{\mathfrak{b}}

\newcommand{\fm}{\mathfrak{m}}

\newcommand{\sP}{\mathscr{P}}
\newcommand{\sX}{\mathscr{X}}

\newcommand{\sE}{\mathscr{E}}
\newcommand{\sD}{\mathscr{D}}
\newcommand{\sU}{\mathscr{U}}
\newcommand{\sH}{\mathscr{H}}

\DeclareMathOperator{\Jac}{Jac}
\DeclareMathOperator{\Fl}{Fl}
\DeclareMathOperator{\trdeg}{tr. deg}
\DeclareMathOperator{\prindiv}{div}

\DeclareMathOperator{\Sym}{Sym}
\DeclareMathOperator{\id}{id}
\DeclareMathOperator{\lcm}{lcm}
\DeclareMathOperator{\Spec}{Spec}

\DeclareMathOperator{\Pic}{Pic}

\DeclareMathOperator{\Supp}{Supp}

\DeclareMathOperator{\lct}{lct}

\DeclareMathOperator{\Aut}{Aut}
\DeclareMathOperator{\vol}{vol}
\DeclareMathOperator{\ord}{ord}
\DeclareMathOperator{\Val}{Val}

\DeclareMathOperator{\GL}{GL}

\DeclareMathOperator{\QM}{QM}

\DeclareMathOperator{\Bl}{Bl}

\newcommand{\scrProj}{\mathscr{P}roj}
\newcommand{\scrBl}{\mathscr{B}l}

\numberwithin{equation}{section}
\counterwithout{equation}{section}

\title[Asymptotics of stability thresholds]{Asymptotics of stability thresholds}

\author{Junyao Peng}
\address{Department of Mathematics, Princeton University, Princeton, NJ 08544, USA}
\email     {junyaop@princeton.edu}

\date{\today}

\begin{document}

\begin{abstract}
We study asymptotic behavior of the stability thresholds of a big line bundle, and prove explicit bounds on the error terms. This answers Jin-Rubinstein-Tian's questions affirmatively. A key step in our proof is to show that the stability thresholds of a big line bundle can always be computed by quasi-monomial valuations. This generalizes Blum--Jonsson's result on the stability thresholds of an ample line bundle.
\end{abstract}

\maketitle
\tableofcontents

\section{Introduction}
Stability thresholds are fundamental invariants in the study of Fano varieties. The $\alpha$-invariant, also known as the global log canonical threshold, was introduced by Tian in the complex analytic setting to study the K\"ahler-Einstein problem \cite{Tia87}. A closely related invariant is the $\delta$-invariant, which characterizes K-stability of Fano varieties via the Fujita-Li valuative criterion \cite{Fuj19}, \cite{Li17}. 

Recently, the behavior of these invariants in more general settings has attracted increasing attention. It is known that for a big line bundle $L$ on a projective variety $X$ with klt singularities, $\alpha(L)$ (resp. $\delta(L)$) can be approximated by finite-level data $\alpha_m(L)$ (resp. $\delta_m(L)$), defined as the log canonical threshold of $X$ with respect to the linear system $|mL|$ (resp. $m$-basis type divisors). In a series of works \cite{JR25a}, \cite{JR25b}, \cite{JRT25a}, and \cite{JRT25b}, Jin, Rubinstein and Tian investigate the approximating process of these invariants by their finite level counterparts. They provide a novel unified approach to the asymptotics of the $\alpha$-invariant and the $\delta$-invariant, and prove the following bounds on the error terms in \cite{JRT25a}, assuming that $\alpha(L)$ (resp. $\delta(L)$) is computed by a \textit{divisorial} valuation:
\[
|\alpha_m(L) - \alpha(L)| =O\left(\frac{1}{m^{1/n}}\right) \quad\text{and}\quad |\delta_m(L) - \delta(L)| =O\left(\frac{1}{m}\right) .
\]
When $L$ is ample, \cite{BJ20}*{Corollary 5.2} establishes that for the $\alpha$-invariant, a better bound $O(1/m)$ can be obtained unconditionally. In fact, when $X$ is a Fano manifold and $L = -K_X$, Tian conjectured that $\alpha_m(L) = \alpha(L)$ for some positive integer $m$ \cite{Tia90}. Tian's conjecture in the $\alpha(-K_X)\leq 1$ case is solved in \cite{Bir21}*{Theorem 1.7} and in the toric Fano case, \cite{JR25b}*{Theorem 1.4} verifies that $\alpha_m(L) = \alpha(L)$ for every $m$. For the $\delta$-invariant, \cite{JR25a}*{Corollary 2.12} shows that the $O(1/m)$ bound is optimal when $X$ is a toric variety and $L$ is an ample toric line bundle. Moreover, Blum--Liu showed that one could uniformly approximate the $\delta$-invariants of ample line bundles in a family \cite{BL22}*{Theorem 5.2}, which was a key input to prove the openness of uniform K-stability in a family of Fano varieties \cite{BL22}*{Theorem A}. 

The assumptions ``$\alpha(L)$ (resp. $\delta(L)$) is computed by a divisorial valuation" are known to hold in several important cases. For instance, for the $\delta$-invariant, this assumption holds when $X$ is an $n$-dimensional Fano variety with klt singularities, $L$ is a multiple of $-K_X$, and $\delta(-K_X) < \frac{n+1}{n}$ by \cite{LXZ22}*{Theorem 1.2}. For a projective klt variety $X$ with $-K_X$ big, \cite{Xu23} implies that when $X$ is K-semistable, then $X$ has an anti-canonical log Fano model $Z$, and $\delta(-K_X)$ is computed by a divisorial valuation if the same holds for $\delta(-K_Z)$. However, these assumptions may not hold in general, at least for log pairs. For example, when $X$ is a smooth cubic surface, $C$ is the union of a line and a conic, and $\epsilon > 0$ is sufficiently small,
the proof of \cite{AZ22}*{Corollary A.7} shows that the $\delta$-invariant of the log anti-canonical divisor of $(X,(1-\epsilon)C)$ is an irrational number and cannot be computed by any divisorial valuation.

Our main result is the following theorem, which establishes that bounds on the error terms in \cite{JRT25a} hold unconditionally and that a better bound $O(1/m)$ can be obtained for the $\alpha$-invariant. In particular, it provides a complete answer to Problem 1.1 and 1.3 in \cite{JRT25a}. 
\begin{thm}\label{mainthm:JRT25a_without_assumptions}
Let $(X, D)$ be an $n$-dimensional projective klt pair and let $L$ be a big line bundle on $X$. Then there exists a constant $C>0$ independent of $m$ such that 
\[
|\alpha_m(L) - \alpha(L)| \leq \frac{C}{m} \quad \text{and} \quad |\delta_m(L) - \delta(L)| \leq \frac{C}{m}.
\]
\end{thm}
Furthermore, \cite{JRT25a} introduces a one-parameter family of invariants called $\delta^\tau$-invariants (denoted as $\boldsymbol{\delta}_\tau$ in \cite{JRT25a}). These $\delta^\tau$-invariants interpolate between the $\alpha$-invariant and the $\delta$-invariant, and they also provide new sufficient conditions for K-stability in the $-K_X$ big case, as shown in \cite{JRT25b}. Analogous asymptotic bounds for $\delta^\tau$-invariants are proved in \cite{JRT25a}*{Theorem 1.7}. Building on their work, we prove Theorem~\ref{thm:JRT_Theorem_1.7_without_assumptions}, which is the unconditional version of \cite{JRT25a}*{Theorem 1.7} and generalizes Theorem~\ref{mainthm:JRT25a_without_assumptions} for $\delta^\tau$-invariants.

There are two new ingredients in the proof of Theorem \ref{mainthm:JRT25a_without_assumptions}. The first ingredient establishes that the $\alpha$-invariant and the $\delta$-invariant of a big line bundle are computed by quasi-monomial valuations. Quasi-monomial valuations are higher-rank generalizations of divisorial valuations, and they arise naturally in many geometric contexts. For example, \cite{JM12} conjectures that any valuation computing the log canonical threshold of a graded ideal sequence must be quasi-monomial, and the existence part of this conjecture is resolved by \cite{Xu20}*{Theorem 1.1}. In the K-stability setting, \cite{BLX22}*{Theorem 1.5} shows that the $\delta$-invariant of a K-unstable Fano variety $X$ is always computed by a quasi-monomial valuation, whose associated graded anti-canonical section ring is finitely generated by \cite{LXZ22}*{Theorem 1.1} and induces an optimal destabilizing degeneration of $X$ by \cite{BLZ22}*{Theorem 1.1}. A local analog of this statement for the normalized volume is conjectured in \cite{Li18}*{Conjecture 7.1} and proved in \cite{XZ25}*{Theorem 1.1 and 1.2}. In all these studies, working with quasi-monomial valuations is natural, and sometimes even indispensable, as illustrated in \cite{JM12}*{Example 8.5}. One of our key observations is that, instead of only considering divisorial minimizers as in \cite{JRT25a}, one should also consider quasi-monomial minimizers, whose existence is confirmed by Theorem~\ref{mainthm:delta_is_computed_by_QM_vals} below.

When $L$ is ample, the combination of \cite{BJ20}*{Theorem E} and \cite{Xu20}*{Theorem 1.1} implies that $\alpha(L)$ and $\delta(L)$ are always computed by quasi-monomial valuations. Our next theorem shows that the same conclusion holds under the weaker assumption that $L$ is big.
\begin{thm}[cf. Theorem~\ref{thm:qm_val_computing_delta_tau}]\label{mainthm:delta_is_computed_by_QM_vals}
Let $(X,D)$ be an $n$-dimensional projective klt pair and let $L$ be a big line bundle on $X$. Then both $\alpha(L)$ and $\delta(L)$ are computed by quasi-monomial valuations on $X$.
\end{thm}
Again, we prove in Theorem~\ref{thm:qm_val_computing_delta_tau} that all the $\delta^\tau$-invariants of a big line bundle can be computed by quasi-monomial valuations.

The second ingredient in our proof of Theorem~\ref{mainthm:JRT25a_without_assumptions} is a construction of Newton--Okounkov bodies of a big line bundle $L$ with respect to a quasi-monomial valuation $v$. This construction is inspired by \cite{LM09}, \cite{KK12}, \cite{Bou14}, and the proof of \cite{LX18}*{Theorem 3.13} (which is essentially the local analog of our construction). The following theorem summarizes some properties of these Newton--Okounkov bodies.
\begin{thm}\label{mainthm:okounkov_bodies_wrt_qm_val}
    Let $X$ be an $n$-dimensional normal projective variety and $L$ be a big line bundle on $X$. Let $v$ be a quasi-monomial valuation on $X$ with rational rank $r$. Let $\mu:(Y,E = E_1+\cdots+E_r)$ be a log resolution of $X$ such that $v$ lies in $\QM_Z(Y,E)$ and has weight $\alpha\in\bR_{>0}^r$ (see Definition~\ref{defn:qm_val}), for some irreducible component $Z$ of $\bigcap_{i=1}^rE_i$. Then there exists a decreasing family of convex bodies $\{\Delta^t\subseteq \bR_{\geq 0}^n: t\in [0, T(v))\}$ (see Definition~\ref{defn:mult_filt}) with the following properties.
    \begin{itemize}
        \item[(a)] For every $t\in [0, T(v))$, we have \[
        \Delta^t = \Delta^0 \cap \left\{(x_1,\ldots,x_n)\in \bR^n: \sum_{i=1}^r \alpha_ix_i \geq t\right\}.
        \]
        \item[(b)] For every effective $\bQ$-divisor $D$ with $D\sim_\bQ L$, there exists a unique point $\beta\in \bQ^r_{\geq 0}$ such that $\alpha\cdot\beta:= \sum_{i=1}^r \alpha_i\beta_i = v(D)$. Moreover, $\beta\in p(\Delta^t)$ for every $t\leq v(D)$, where $p: \bR^n\to\bR^r$ is the projection to the first $r$ coordinates.
        \item [(c)] For every $t\in [0, T(v))$, the Euclidean volume of $\Delta^t$ is equal to $\frac{1}{n!} \vol(V_\bullet^t(v))$, where $V_\bullet^t(v)$ is the graded sublinear series of $L$ induced by $v$ (see Definition~\ref{defn:graded_linear_series_induced_by_filtrations}).
    \end{itemize}
\end{thm}

A key feature of these Newton--Okounkov bodies $\Delta^t$ in Theorem~\ref{mainthm:okounkov_bodies_wrt_qm_val} is that one can directly read off the $S$-invariant (see Lemma-Definition~\ref{lemdefn:S_tau_is_a_limit} and Remark~\ref{rem:S_tau_invariant_interpolates_S_and_T}) and the $T$-invariant of $v$, and approximate them using the integral points of $\Delta^t$. Once such a family of Newton--Okounkov bodies is constructed, techniques in \cite{JRT25a} yield asymptotic bounds for $S(v)$ and $T(v)$.\\

\noindent\textbf{Strategy of proof.} 
The proof of Theorem \ref{mainthm:delta_is_computed_by_QM_vals} is motivated by Blum--Jonsson's proof in the ample case \cite{BJ20}*{Theorem E}, which consists of four steps:
\begin{itemize}
\item[(1)] Suppose $\{v_i:i\geq 1\}$ is a sequence of valuations with $\delta(v_i) \searrow \delta(L)$. Each $v_i$ can be parametrized as a filtration $\mathcal{F}_i$ on the section ring of $L$, which corresponds to a point on the (product of) flag varieties.
\item[(2)] Using a generic limit argument, one obtains a limiting filtration $\mathcal{F}$ of all $\mathcal{F}_i$.
\item[(3)] The limiting valuation $v$ of $v_i$ can be constructed as any valuation computing the log canonical threshold of the base ideal sequence associated to $\mathcal{F}$. $v$ can be chosen to be quasi-monomial by \cite{Xu20}.
\item[(4)] Finally, one shows that $\delta(v) = \delta(L)$.
\end{itemize}
The key input in step (4) is a uniform Fujita type approximation of the $S$-invariant. \cite{BJ20}*{Section 5.1} defines the $\Tilde{S}_m$-invariants for any filtration, which are modifications of the $S_m$-invariants but exhibit smoother asymptotic behavior. Then \cite{BJ20}*{Theorem 5.3} shows that $\Tilde{S}_m(v)$ uniformly approximates $S(v)$ for any valuation $v$ with finite log discrepancy.

Unfortunately, the proof of the uniform Fujita type approximation result in \cite{BJ20} does not directly generalize to the case where $L$ is big.
Because of this, our proof of Theorem~\ref{mainthm:delta_is_computed_by_QM_vals} differs in several ways. Firstly, we parametrize not only filtrations $\mathcal{F}$ on the section ring of $L$, but also compatible filtrations $\mathcal{G}$ on auxiliary linear series $|mL+B|$ for some fixed ample divisor $B$ (see Section~\ref{subsec:parametrize_filtrations}). Secondly, we modify the definition of the $\Tilde{S}_m$-invariants in terms of the auxiliary filtration $\mathcal{G}$ (see Definition~\ref{defn:S_tilde_m_invariants}), and show that $\Tilde{S}_m(\cG)$ converges to $S(\mathcal{F})$ in certain cases (see Theorem~\ref{thm:uniform_fujita_approx_BJ_thm_5.3} and Lemma~\ref{lem:limit_of_S_tau_m_is_S}). Thirdly, in the generic limit argument, we need to control additional invariants to ensure that the limiting filtration would satisfy extra properties (see step 2 of Section~\ref{subsec:proof_of_theorem_6.1}). Finally, to prove the general version of Theorem~\ref{mainthm:JRT25a_without_assumptions} for the $\delta^\tau$-invariants, we need to replace the $S$-invariants by the $S^\tau$-invariants and adapt the arguments above to this setting. 

Once Theorem~\ref{mainthm:delta_is_computed_by_QM_vals} is established and a minimizing quasi-monomial valuation $v$ is found, we construct Newton--Okounkov bodies that capture asymptotic properties of the graded linear series induced by $v$ as in Theorem~\ref{mainthm:okounkov_bodies_wrt_qm_val}. When $v$ is the divisorial valuation induced by the divisor $E$, Theorem~\ref{mainthm:okounkov_bodies_wrt_qm_val} is known in \cite{LM09} and \cite{BC11}, by taking an admissible flag in which the first subvariety is $E$, and constructing the associated Newton-Okounkov bodies. In the non-divisorial case, we combine this strategy with ideas in \cite{KK12}, \cite{Bou14} and use a $\bZ^n$-valued valuation $\nu_\bullet$ to construct these bodies. For any rational function $f$, if we expand $f$ as a power series in the local coordinates given by $E_1,\ldots,E_r$, then the first $r$ coordinates of $\nu_\bullet(f)$ will be the indices of the lowest-order term in the expansion of $f$ with respect to the weight vector $\alpha$. The last $n-r$ coordinates of $\nu_\bullet(f)$ are determined by a suitable restriction of $f$ on $Z$ and an admissible flag on $Z$, as in \cite{LM09}. 

Finally, to prove Theorem~\ref{mainthm:JRT25a_without_assumptions}, we can apply Theorem~\ref{mainthm:okounkov_bodies_wrt_qm_val} to construct Newton--Okounkov bodies with respect to the minimizing quasi-monomial valuations given by Theorem~\ref{mainthm:delta_is_computed_by_QM_vals}. Then the uniform estimates of lattice point counts of these Newton--Okounkov bodies yield bounds on the the $\delta_m$-invariants, as argued in \cite{JRT25a}. The bounds on the $\alpha_m$-invariants can be deduced from the uniform Fujita type approximation of the $T$-invariant (Theorem~\ref{thm:uniform_fujita_approx_T}). Thus, we obtain Theorem~\ref{mainthm:JRT25a_without_assumptions}.

We would like to note that our proof strategy of Theorem~\ref{mainthm:JRT25a_without_assumptions} takes a roundabout route. The first ingredient of the proof of Theorem~\ref{mainthm:delta_is_computed_by_QM_vals} is based on \cite{Xu20}, which further relies on deep machinery of minimal model programs such as the boundedness of complements \cite{Bir19}. A seemingly more direct approach would be showing that there is a \textit{uniform} bound on $|S_m(v) - S(v)|$ for all (divisorial) valuations $v$. However, we have not succeeded in making this direct approach work.\\

\noindent\textbf{Outline of the paper.} In Section~\ref{sec:preliminary}, we recall standard results on quasi-monomial valuations, graded linear series, and stability invariants. In Section~\ref{sec:okounkov_bodies}, we construct Newton--Okounkov bodies of big line bundles with respect to quasi-monomial valuations. We then prove Theorem~\ref{mainthm:okounkov_bodies_wrt_qm_val} at the end of this section. In Section~\ref{sec:qm_vals_computing_stability_invariants}, we prove Theorem~\ref{thm:qm_val_computing_delta_tau}, from which Theorem~\ref{mainthm:delta_is_computed_by_QM_vals} follows as a special case. Finally, in Section~\ref{sec:asymptotics_of_stability_thresholds}, we prove Theorem \ref{mainthm:JRT25a_without_assumptions} based on the arguments in \cite{JRT25a}.

In Appendix~\ref{sec:variation_okounkov_bodies}, we study variations of Newton--Okounkov bodies constructed in Section~\ref{sec:okounkov_bodies}. Unlike \cite{LM09}*{Section 4} in which line bundles vary, we fix the line bundle and instead vary the weight vectors of quasi-monomial valuations on some fixed strata. We also give a nontrivial example of variations of Newton--Okounkov bodies in Section~\ref{subsec:explicit_variations_of_okunkov_bodies_on_P2}. In Appendix~\ref{sec:weighted_proj_stacks}, we recall some fundamental results about weighted projective stacks and weighted blow-ups, which can be used to study quasi-monomial valuations. 
\\

\noindent\textbf{Acknowledgement.} The author would like to thank his advisor Chenyang Xu for the suggestion of this problem, many detailed discussions, and providing insights into related topics. The author thanks Ziquan Zhuang for pointing out an example in \cite{AZ22} where the $\delta$-invariant is not computed by any divisorial valuation. The author is also grateful to Harold Blum, S{\'e}bastien Boucksom, Zhiyuan Chen, Chenzi Jin, and Mattias Jonsson for helpful discussions and comments. The author appreciates the Simons Collaboration on Moduli of Varieties for organizing many inspiring events and fostering a supportive community. The author is supported by NSF Grant DMS-2201349.

\section{Preliminary results}\label{sec:preliminary}
Throughout this paper, we work over the field of complex numbers $\bC$. We follow the standard terminology from \cite{KM98}, \cite{Kol13}, and \cite{Xu25}. We denote $\bN$ as the set of nonnegative integers, $\bZ_{>0}$ the set of positive integers, and $\bR_{>0}$ the set of positive real numbers.

Let $X$ be a normal projective variety. \textit{A graded ideal sequence $\fa_\bullet$ on $X$} is a sequence of ideals $\{\fa_p\subseteq \cO_X: p\in \bN\}$ such that $\fa_0 = \cO_X$ and $\fa_p \cdot \fa_q \subseteq \fa_{p+q}$ for all $p,q\in\bN$. Let $L$ be a line bundle on $X$. We denote $|L|$ as the complete linear system of $L$. For a sublinear system $V\subseteq H^0(L)$, we denote $\fb(|V|)$ as the base ideal of $V$.

Let $\Gamma$ be a totally ordered abelian group. We say that a function $v: K(X) \to \Gamma \cup \{\infty\}$ is a \textit{$\Gamma$-valuation on $X$} if $v$ is trivial on $\bC^\times \subseteq K(X)^\times$, $v(f) = \infty$ if and only if $f = 0$, and for any $f_1,f_2\in K(X)^\times$, we have 
        \[
        v(f_1f_2) = v(f_1) + v(f_2) \quad \text{and} \quad  v(f_1 + f_2) \geq \min \{v(f_1), v(f_2)\}.
        \]
\textit{The value group of $v$} is the image of $K(X)^\times$ in $\Gamma$. An $\bR$-valuation on $X$ is also called \textit{a valuation on $X$}. For a valuation $v$ on $X$ and $\lambda \geq 0$, we set
    \[
    \fa_\lambda (v) := \{f\in \cO_X: v(f)\geq \lambda\}.
    \]
Then $\fa_\bullet(v)$ is a graded ideal sequence and it is called \textit{the valuation ideal sequence associated to $v$}.

Let $D$ be an effective $\bQ$-divisor on $X$. We say that $(X,D)$ is a \textit{log pair} if $X$ is normal and $K_X+D$ is $\bQ$-Cartier. We say that $(X,D)$ is \textit{a klt pair} if it is a log pair with klt singularities. For a log pair $(X,D)$, we denote $A_{X,D}(v)$ as the log discrepancy of $v$ with respect to $(X,D)$ (see \cite{JM12}*{Section 5} for the definition). We denote $\Val_X^*, \Val_X^{\prindiv}$, and $\Val_{X,D}^{<\infty}$ as the set of non-trivial valuations, the set of divisorial valuations, and the set of valuations with finite log discrepancy with respect to $(X,D)$, respectively. Furthermore, if $(X,D)$ is a klt pair and $\fb$ is an ideal on $X$, we denote $\lct(X,D; \fb)$ as the log canonical threshold of $\fb$ with respect to $(X,D)$. When the pair $(X,D)$ is clear from the context, we sometimes also denote $A(v) = A_{X,D}(v)$ and $\lct(\fb) = \lct(X,D;\fb)$.

Throughout this article, we will define several invariants on the set of filtrations (see Definition~\ref{defn:mult_filt} and Definition~\ref{defn:compatible_filtration}). If $I$ is such an invariant and $\cF_v$ is the filtration induced by a valuation $v$ (see Definition~\ref{defn:filtration_induced_by_a_valuation} and Example~\ref{example:compatible_filtration_induced_by_valuation}), we will also denote $I(v) = I(\cF_v)$. 

\subsection{Quasi-monomial valuations}
Let $X$ be a smooth variety and let $D$ be a simple normal crossing divisor on $X$. Suppose $E$ has $r$ irreducible components $D_1,\ldots, D_r$ and $Z\subseteq \bigcap_{i=1}^r D_i$ is a (non-empty) irreducible component. Let $\eta$ denote the generic point of $Z$. Suppose around $\eta$, each $D_i$ is cut out by a function $x_i\in \cO_{X,\eta}$. By the Cohen structure theorem, there is an isomorphism $\hat{\cO}_{X,\eta}\cong K(Z)[[x_1,\ldots,x_r]]$. Thus, we can write any $f\in \hat{\cO}_{X,\eta}$ as a power series
\begin{equation}\label{eqn:power_series_expansion}
f = \sum_{\beta\in \bN^r} c_\beta x^\beta, \text{ where }x^\beta = \prod_{i=1}^r x_i^{\beta_i} \text{ and } c_\beta \in K(Z).
\end{equation}
\begin{defn}[\cite{JM12}*{Section 3.1}]\label{defn:qm_val}
    For any $\alpha = (\alpha_1,\ldots,\alpha_r)\in \bR_{\geq 0}^r$, we define a valuation $v_\alpha$ on $\cO_{X,\eta}$ as follows. For any $f\in \cO_{X,\eta}$ whose image in $\hat{\cO}_{X,\eta}$ is written in the power series form \eqref{eqn:power_series_expansion}, we set
    \[
    v_\alpha(f) = \min\left\{\alpha\cdot \beta = \sum_{i=1}^r \alpha_i\beta_i: \beta\in \bN^r, c_\beta \neq 0\right\}.
    \]
    We extend $v_\alpha$ to a valuation on $X$. We denote
    \[
    \QM_\eta(X,D) = \QM_Z(X,D) = \{v_\alpha: \alpha\in \bR_{\geq 0}^r\}
    \]
    and we call $v_\alpha$ the \emph{quasi-monomial valuation in $\QM_\eta(X,D)$ with weight vector $\alpha$}. Furthermore, we denote
    \[
    \QM(X,D) = \bigcup_{W} \QM_W(X,D),
    \]
    where $W$ runs through every irreducible stratum of $D$. We say that $v$ is a \emph{quasi-monomial valuation on $X$} if there is a log resolution $(Y,E)\to X$ such that $v\in \QM(Y,E)$.

    The \emph{rational rank} (\emph{$\bQ$-rank}) \emph{of a quasi-monomial valuation $v$} is the rank of the value group of $v$. In particular, if $v\in \QM_\eta(X,D)$ has weight vector $\alpha\in \bR^r_{\geq 0}$, then the $\bQ$-rank of $v$ is the dimension of the $\bQ$-vector space spanned by $\{\alpha_i:1\leq i\leq r\}$.
\end{defn}

\begin{rem}\label{remark:qm_val_wrt_nc_divisors}
In fact, $v_\alpha$ and $\QM_\eta(X,D)$ in Definition~\ref{defn:qm_val} still makes sense if one relax the condition ``$(X,D)$ is a simple normal crossing pair" to the following weaker condition: $(X,D)$ is a normal crossing pair in a neighborhood of $\eta$, where $\eta$ is a log canonical center of the pair $(X,D)$. (In fact, one only needs that local equations of $D_i$ in $\cO_{X,\eta}$ form a regular sequence.) This is because around $\eta$, every $f\in \cO_{X,\eta}$ still admits a power series expansion \eqref{eqn:power_series_expansion}.
\end{rem}

\begin{rem}
There is a natural $\bR_{>0}$-action on $\QM(X,D)$ which rescales any valuation by a positive constant. After putting a suitable topology on $\QM(X,D)\setminus \{0\}$, this $\bR_{>0}$-action will be free and continuous, and the quotient space $(\QM(X,D)\setminus \{0\})/\bR_{>0}$ can be identified with the dual complex of $D$.
\end{rem}

\begin{rem}
By \cite{JM12}*{Proposition 3.7}, a valuation $v$ on $X$ is quasi-monomial if and only if it is \emph{an Abhyankar valuation}, i.e., it satisfies the equality
\[
\trdeg(v) + \bQ\text{-rank}(v) = \dim X.
\]
\end{rem}

\subsection{Graded linear series and filtrations}
Let $L$ be a big line bundle on an $n$-dimensional normal projective variety $X$. For each $m\in \bN$, we denote $R_m = H^0(X,mL)$ and $R = \bigoplus_{m\in \bN} R_m$.
\begin{defn}\label{defn:graded_linear_series_and_contains_ample_series}
    \emph{A graded linear series of $L$} is a graded subalgebra
    \[
    V_\bullet = \bigoplus_{m\in \bN} V_m\subseteq \bigoplus_{m\in \bN} R_m = R.
    \]
    \emph{The volume of $V_\bullet$} is the limit
    \[
    \vol(V_\bullet):= \frac{1}{n!}\lim_{m\to\infty} \frac{\dim V_m}{m^n},
    \]
    which exists by \cite{Bou14}*{Theorem 3.7}.
    We say that \emph{$V_\bullet$ contains an ample series} if $V_m\neq 0$ for all sufficiently large $m$, and there exists an ample $\bQ$-Cartier $\bQ$-divisor $A\leq L$ such that
    \[
    H^0(X, mA) \subseteq V_m\subseteq R_m
    \]
    for all sufficiently divisible $m$.
\end{defn}

\begin{defn}\label{defn:mult_filt}
    \emph{A filtration $\cF$ on $R$} consists of a family 
    \[
    \cF^\lambda R_m\subseteq R_m
    \]
    of vector subspaces for all $m\in \bN$ and $\lambda\in \bR$ such that the following holds:
    \begin{itemize}
        \item (Decreasing) $\cF^\lambda R_m\subseteq \cF^{\lambda'}R_m$ for all $\lambda > \lambda'$;
        \item (Left-continuous) $\cF^\lambda R_m = \bigcap_{\lambda'<\lambda}\cF^{\lambda'}R_m$;
        \item (Bounded) $\cF^0 R_m  = R_m$ and $\cF^\lambda R_m = 0$ for $\lambda \gg 0$;
        \item (Multiplicative) $\cF^\lambda R_m \cdot \cF^{\lambda'} R_{m'}\subseteq \cF^{\lambda+\lambda'} R_{m+m'} $ for all $\lambda,\lambda'\in \bR$ and $m,m'\in \bN$.
    \end{itemize}
    For each $m\in \bN$, set
    \[
    T_m(\cF) := \frac{1}{m}\sup\{\lambda \geq 0: \cF^\lambda R_m \neq 0\} \quad\text{and}\quad T(\cF) := \sup\left\{T_m(\cF) : m\in\bN\right\}.
    \]
    We say that $\cF$ is \emph{linearly bounded} if $T(\cF)$ is finite.
\end{defn}
By Feteke's lemma, we also have 
    \[
    T(\cF) = \lim_{m\to\infty} T_m(\cF)
    \]
    for any filtration $\cF$ on $R$.

\begin{defn}\label{defn:graded_linear_series_induced_by_filtrations}
    Let $\cF$ be a filtration on $R$. For any $t\in \bR_{\geq 0}$, we define $V_\bullet^t(\cF)$ as the graded linear series given by
    \[
    V_\bullet^t(\cF) = \bigoplus_{m\in \bN} V_m^t(\cF) := \cF^{mt}R_m.
    \]
\end{defn}

\begin{lem}[\cite{BC11}*{Lemma 1.6}, \cite{BJ20}*{Section 2.4}]\label{lem:V_bullet^t_contains_ample_linear_series}
For any $t < T(\cF)$, $V_\bullet^t(\cF)$ contains an ample linear series. As a result, we have
\[
T(\cF) = \sup\{t\geq 0: \vol(V_\bullet^t(\cF)) > 0\}.
\]
\end{lem}
We recall the following log-concave property of the volume function.
\begin{lem}\label{lem:log_concavity_vol}
    The function $t\mapsto \vol(V_\bullet^t(\cF))^{1/n}$ is log-concave and continuous on $[0, T(\cF))$. 
\end{lem}
\begin{proof}
    Let $\Delta^t$ be the Newton--Okounkov body associated to the graded linear series $V_\bullet^t(\cF)$ (with respect to a fixed admissible flag on $X$). Then $\vol(V_\bullet^t(\cF)) = n! \vol(\Delta^t)$ and the log-concavity property follows from the Brunn-Minkowski inequality for convex bodies (details can be found in \cite{BKMS15}*{Lemma 2.22} or \cite{Xu25}*{Proposition 3.19}). Continuity follows from log-concavity and monotonicity.  
\end{proof}

\begin{defn}\label{defn:filtration_induced_by_a_valuation}
    Let $v$ be a valuation on $X$. Let $\cF_v$ be the filtration on on $R$ given by
    \[
    \cF_v^\lambda R_m := \{s\in R_m: v(s)\geq \lambda\}
    \]
    for all $m\in \bN$ and $\lambda\in \bR$. $\cF_v$ is called \emph{the filtration induced by $v$}.
\end{defn}

We say that a valuation $v$ \emph{has linear growth with respect to a big line bundle $L$} if the induced filtration $\cF_v$ on the section ring of $L$ is linearly bounded. By \cite{BKMS15}*{Lemma 2.8}, this notion does not depend on the choice of big line bundles. Thus, it makes no ambiguity to say that a valuation $v$ \emph{has linear growth} without reference to any big line bundle. 

A large class of valuations have linear growth. For example, any quasi-monomial valuation has linear growth by \cite{BKMS15}*{Proposition 2.12}. If $(X,D)$ has klt singularities and $A_{X,D}(v) < \infty$, then $v$ has linear growth by \cite{BJ20}*{Lemma 3.1}. If $v$ is a valuation centered at a closed point $x\in X$, then \cite{BKMS15}*{Theorem 2.16} shows that $v$ has linear growth if and only if
\[
\vol(v) := \lim_{m\to\infty} \frac{n!}{m^n} \text{length}(\cO_{X,x}/\fa_m(v)) > 0.
\]

\subsection{A family of stability invariants} 
Let $X$ be a $n$-dimensional normal projective variety and $L$ a big line bundle on $X$. Let $\cF$ be a \textit{linearly bounded} filtration on $R = \bigoplus_{m\in \bN} R_m = \bigoplus_{m\in \bN} H^0(X,mL)$. Denote $N_m = \dim R_m$. 

\begin{defn}
    The \emph{jumping numbers of $\cF$} are
    \[
    j_{m,k}(\cF) :=  \sup \{\lambda\geq 0: \dim \cF^\lambda R_m\geq k\}
    \]
    for $m\in\bN$ and $k\in \{1,2,\ldots, N_m\}$.
\end{defn}
By the left-continuous property of $\cF$ (Definition~\ref{defn:mult_filt}), one can replace ``sup" by ``max" in the definition of $j_{m,k}(\cF)$. The jumping numbers of $\cF$ satisfy
\[
mT_m(\cF) = j_{m,1}(\cF) \geq j_{m,2}(\cF) \geq \cdots \geq j_{m,N_m}(\cF) \geq 0.
\]

\begin{defn}[\cite{JRT25a}*{Section 2.3}]\label{defn:S_m,k_invariant}
    For $m\in \bN$ and $k\in \{1,2,\ldots, N_m\}$, define 
    \[
    S_{m,k}(\cF) := \frac{1}{mk}\sum_{\ell=1}^{k} j_{m,\ell}(\cF)
    \]
    Equivalently (by \cite{JRT25a}*{Lemma 2.13}), we have
    \[
    S_{m,k}(\cF) = \frac{1}{mk}\max\left\{\sum_{\ell=1}^k \lambda(s_\ell): s_1,\ldots, s_{k}\in R_m \text{ are linearly independent}\right\},
    \]
    where $\lambda(s) = \sup\{\lambda: s\in \cF^\lambda R_m\}$.
\end{defn}

\begin{lemdefn}[\cite{JRT25a}*{Proposition 4.12}]\label{lemdefn:S_tau_is_a_limit}
    Let $\{k_m\in \bZ: 1\leq k_m\leq N_m, m\in \bN\}$ be a sequence of positive integers and assume $\tau := \lim_{m\to\infty} \frac{k_m}{N_m}$ exists. Then the limit
    \[
    S^\tau(\cF) := \lim_{m\to\infty} S_{m, k_m}(\cF)
    \]
    exists and does not depend on the choice of $\{k_m: m\in \bN\}$.
\end{lemdefn}
\begin{proof}
    When $\cF$ is the filtration induced by a valuation $v$ of linear growth, this is proved in \cite{JRT25a}*{Proposition 4.12}. In fact, the same proof works for any linearly bounded filtration $\cF$, since the additive property of valuations is not used in that proof.
\end{proof}

\begin{rem}\label{rem:S_tau_invariant_interpolates_S_and_T}
    We have $S^0(\cF) = T(\cF)$ and $S^1$ is the usual $S$-invariant (also called the \emph{expected vanishing order}) in K-stability.
\end{rem}

Similar to the $S$-invariant, $S^\tau(\cF)$ can be expressed as an integral of $\vol(V_\bullet^t(\cF))$.
\begin{lem}\label{lem:integration_formula_S_tau}
When $\tau\in (0,1]$,
\[
S^\tau({\cF}) = \frac{1}{\tau\vol(L)}\int_{Q_\tau(\cF)}^\infty \vol(V_\bullet^t({\cF})) dt = \frac{1}{\tau\vol(L)}\int_{Q_\tau(\cF)}^{T(\cF)}\vol(V_\bullet^t({\cF})) dt,
\]
where
\[Q_\tau(\cF) = \inf \{t \geq 0: \vol(V_\bullet^t(\cF))\leq \tau \vol(L)\}\in [0, T(\cF)].
\]
\end{lem}
\begin{proof}
    When $\cF$ is the filtration induced by a valuation $v$ of linear growth, this follows from \cite{JRT25a}*{Remark 4.17}. The same proof works for any linearly bounded filtration $\cF$.
\end{proof}

We recall the following relations between the $S^\tau$-invariants and the $T$-invariant.
\begin{lem}\label{lem:ineq_S_and_T_invariants}
    The sequence $\{S^\tau(\cF): \tau\in [0,1]\}$ is non-increasing with respect to $\tau\in [0,1]$. Furthermore, we have
    \[
    \frac{1}{n+1}T(\cF)\leq S^\tau(\cF) \leq T(\cF).
    \]
\end{lem}
\begin{proof}
    The second inequality follows from the definition of $S^\tau$. The first inequality follows from Lemma~\ref{lem:integration_formula_S_tau} and Lemma~\ref{lem:log_concavity_vol} (details can be found in \cite{BJ20}*{Lemma 2.6}).
\end{proof}

Let $D$ be an effective $\bQ$-divisor on $X$ such that $(X,D)$ is a log pair with klt singularities. One can define a family of delta invariants indexed by $\tau\in [0,1]$.
\begin{defn}[\cite{JRT25a}*{Lemma 4.1}]\label{defn:delta_tau}
    Let $v$ be a valuation on $X$ such that $A_{X,D}(v) < \infty$. For $m\in \bN$ and $k\in \{1,2,\ldots, N_m\}$, we define
    \[
    \delta_{m,k}(v) := \frac{A_{X,D}(v)}{S_{m,k}(v)} \quad\text{and}\quad \delta_{m,k}(L) := \inf\{\delta_{m,k}(v):v\in \Val_X^{\prindiv}\}.
    \]
    Equivalently (by \cite{JRT25a}*{Lemma 4.1}), 
    \[
    \delta_{m,k}(L) = \inf \left\{mk\lct\left(X,D;\sum_{\ell=1}^k \prindiv(s_\ell)\right): s_1,\ldots,s_k\in R_m \text{ are linearly independent}\right\}.
    \]
    Finally, for $\tau\in [0,1]$, we define
    \[
    \delta^\tau(v) := \frac{A_{X,D}(v)}{S^\tau(v)} \quad \text{and}\quad \delta^\tau(L) := \inf\{\delta^\tau(v): v\in \Val_X^{\prindiv}\}.
    \]
\end{defn}

We have the following corollary of Lemma~\ref{lemdefn:S_tau_is_a_limit}:
\begin{cor}\cite{JRT25a}*{Corollary 4.26}
    Let $\tau\in [0,1]$ and $\{k_m\in \bN: 1\leq k_m\leq N_m, m\in \bN\}$ be a sequence of positive integers such that $\lim_{m\to\infty}\frac{k_m}{N_m} = \tau$. Then
    \[
    \delta^\tau(L) = \lim_{m\to\infty} \delta_{m,k_m}(L). 
    \]
\end{cor}

\begin{rem}\label{rem:delta_tau_interpolates_alpha_and_delta}
$\delta^0$ is the $\alpha$-invariant introduced by Tian in \cite{Tia87} and $\delta^1$ is the usual $\delta$-invariant in K-stability. Thus, $\{\delta^\tau:\tau\in [0,1]\}$ interpolates between the $\alpha$-invariant and the $\delta$-invariant. Furthermore, by \cite{JRT25b}*{Theorem 1.11}, modified versions of $\delta^\tau$-invariants give new valuative criteria for K-(semi)stability of $(X,L)$.
\end{rem}

\subsection{$\bN$-filtrations}
Let $L$ be a big line bundle on an $n$-dimensional normal projective variety $X$. For each $m\in \bN$, we denote $R_m = H^0(X,mL)$ and $R = \bigoplus_{m\in \bN} R_m$. Let $\cF$ be a filtration on $R$.

\begin{defn}\label{defn:N_filtration}
    We say that $\cF$ is an \emph{$\bN$-filtration} if all the jumping numbers of $\cF$ are integers. Equivalently, this means
    \[
    \cF^\lambda R_m = \cF^{\lceil \lambda\rceil} R_m
    \]
    for every $m\in \bN$ and $\lambda\in \bR$. Any filtration $\cF$ \emph{induces an $\bN$-filtration $\cF_\bN$ on $R$} given by
    \[
    \cF_\bN^\lambda R_m := \cF^{\lceil \lambda\rceil} R_m.
    \]
\end{defn}

\begin{lem}\cite{BJ20}*{Proposition 2.11}\label{lem:BJ_prop_2.11}
    Let $\cF$ be a linearly bounded filtration on $R$. Then $\cF_\bN$ is also linearly bounded. Furthermore, for any $m\in \bN$, $k\in \{1,2,\ldots, N_m\}$, and $\tau\in [0,1]$, we have
    \[
    T_m(\cF_\bN) = \frac{1}{m}\lfloor mT_m(\cF)\rfloor,
    \]
    \[
    S_{m,k}(\cF) -\frac{1}{m} \leq S_{m,k}(\cF_\bN) \leq S_{m,k}(\cF),
    \]
    and
    \[
    S^\tau(\cF_\bN) = S^\tau(\cF).
    \]
\end{lem}
\begin{proof}
    All statements follow from the fact that the jumping numbers of $\cF_\bN$ and $\cF$ satisfy the relation 
    \[j_{m,k}(\cF_\bN) = \lfloor j_{m,k}(\cF)\rfloor.\]
\end{proof}

\section{Newton--Okounkov bodies with respect to quasi-monomial valuations}\label{sec:okounkov_bodies}
In \cite{Oko96} and \cite{Oko03}, Okounkov studies the log-concavity property of certain multiplicity and degree functions in representation theory and algebraic geometry. His key innovation is to construct associated convex bodies so that techniques in convex geometry, such as the Brunn-Minkowski inequality, can be applied. These convex bodies are generalizations of Newton polytopes and are now called Newton--Okounkov bodies. Since the works of Okounkov, there has been many more systematic constructions of Newton--Okounkov bodies, for example by Lazarsfeld-Musta{\c t}{\u a} \cite{LM09}, Kaveh-Khovanskii \cite{KK12}, and Boucksom \cite{Bou14}.

Newton--Okounkov bodies are convenient tools to study approximation problems. Fujita's approximation theorem states that the volume of a big divisor can be approximated arbitrarily closely by the self-intersection number of an ample divisor on a birational model. There are several proofs of Fujita's approximation theorem, including the original proof by T. Fujita \cite{Fuj94}, the proof by the subadditivity of multiplier ideals \cite{DEL00}, the proof by higher jets \cite{Nak03}, and the proof by Newton--Okounkov bodies \cite{LM09}. The Newton--Okounkov body approach has the advantage that it can prove a stronger version of Fujita's approximation theorem \cite{LM09}*{Theorem D}. In the local setting, a similar approximation theorem for the Hilbert-Samuel multiplicity of primary ideals is proved in \cite{BL21}*{Lemma 13}.

In this section, we shall construct a Newton--Okounkov body of a line bundle $L$ with respect to a quasi-monomial valuation $v$. We will use this Newton--Okounkov body to obtain approximation results for $\delta$-invariants in Section~\ref{sec:asymptotics_of_stability_thresholds}. As in the usual constructions (for example, \cite{LM09}, \cite{KK12}, and \cite{Bou14}), the Newton--Okounkov body $\Delta$ will be a slice of the convex cone generated by a set of integral points $\Gamma\subseteq \bN^{n+1}$. $\Gamma$ will be defined as
\[
\Gamma = \{(\nu_\bullet(s), m): s\in H^0(X, mL)\}
\]
where $\nu_\bullet = (\nu_1,\ldots,\nu_n): K(X)\to \bZ^n$ is a $\bZ^n$-valuation on $X$. Such $\nu_\bullet$ will come by two parts. Suppose the rational rank of $v$ is $r$. Then the first $r$ coordinates $(\nu_1,\ldots, \nu_r)$ will be determined by $v$. The last $n-r$ coordinates will be defined as the $\bZ^{n-r}$-valuation associated to an admissible flag $Z_\bullet$ on the center of $v$. 

When $r = 1$, this construction coincides with the construction in \cite{LM09}. When $r = n$, it coincides with \cite{Bou14}*{Example 2.16}. We also notice that there is a similar construction in the local setting by Li-Xu in their proof of \cite{LX18}*{Theorem 3.13}.

\subsection{Construction of Newton--Okounkov bodies}\label{subsec:construction_of_okounkov_body}
Let $X$ be an $n$-dimensional smooth projective variety. Let $D = D_1 + \ldots + D_r$ be a simple normal crossing divisor on $X$. Suppose $Z$ is a (non-empty) irreducible component of $\bigcap_{i=1}^r D_i$. Let $\eta$ be the generic point of $Z$ and $v = v_\alpha\in \text{QM}_\eta(X,D)$ be the quasi-monomial valuation with weight vector $\alpha = (\alpha_1,\ldots,\alpha_r)\in \bR_{>0}^r$ (see Definition~\ref{defn:qm_val}). We assume that $\alpha_1,\ldots,\alpha_r$ are $\bQ$-linearly independent, so that the $\bQ$-rank of $v_\alpha$ equals $r$. Let
\[
Z_\bullet: Z = Z_0 \supset Z_1  \supset \cdots \supset Z_{n-r} = \{z\}
\]
be a flag of irreducible subvarieties of $Z$, such that $Z_i$ has codimension $i$ in $Z$ and is smooth at $z$ for every $1\leq i\leq n-r$. $Z_\bullet$ is usually called \emph{an admissible flag} \cite{LM09}*{Section 1}. 

We recall the definition of the valuation $\nu_{Z_\bullet}$ associated to the admissible flag $Z_{\bullet}$ in \cite{LM09}.
\begin{defn}[cf. \cite{LM09}*{Section 1.1}]\label{defn:val_associated_to_flag}
Let $g\in K(Z)$ be a nonzero rational function. After replacing $Z$ by an open neighborhood of $z\in Z$, we may assume that $Z_{i+1}$ is a Cartier divisor on $Z_i$ for all $i$. Set $G_0 = \prindiv(g)$. For $1\leq i\leq n-r$, we inductively set
\[
k_i = \ord_{Z_i}(G_{i-1}) \quad\text{and}\quad G_i = (G_{i-1} - k_iZ_i)|_{Z_i}.
\]
Then we define $\nu_{Z_\bullet}: K(Z)\to\bZ^{n-r}\cup\{\infty\}$ by setting
\[
\nu_{Z_\bullet}(g) = (k_1,\ldots,k_{n-r})
\]
and $\nu_{Z_\bullet}(0) = \infty$. $\nu_{Z_\bullet}$ is called \emph{the $(\bZ^{n-r}$-$)$valuation associated to the admissible flag $Z_\bullet$}.
\end{defn}

Let $f\in \cO_{X,\eta}$. Recall equation \eqref{eqn:power_series_expansion}, where we write $f$ as a power series expansion
\[
f = \sum_{\beta\in \bN^r} c_\beta x^\beta , \text{ where } x^\beta = \prod_{i=1}^r x_i^{\beta_i} \text{ and }c_\beta\in K(Z).
\]
Then $v(f) = \min\{\alpha\cdot \beta: c_\beta\neq 0\}$. Since $\alpha_1,\ldots,\alpha_r$ are $\bQ$-linearly independent, there exists a unique index $\beta$ such that $c_\beta\neq 0$ and $\alpha\cdot \beta = v(f)$. We call this index $\beta_v(f)$ and we say that $c_{\beta_v(f)}x^{\beta_v(f)}$ is \emph{the lowest-order term of $f$ with respect to $v$}. For $v_{\alpha'}\in \QM_\eta(X,D)$ with a general weight vector $\alpha'$, we say that $c_\beta x^\beta$ is \textit{a lowest-order term of $f$ with respect to $v_{\alpha'}$} if $c_\beta\neq 0$ and $\alpha'\cdot \beta = v_{\alpha'}(f)$. For a general $v_{\alpha'}\in \QM_\eta(X,D)$, there might be multiple lowest-order terms of $f$ with respect to it.

\begin{defn}\label{defn:nu_bullet}
Fix a quasi-monomial valuation $v\in \QM_\eta(X,D)$. Define $\nu_\bullet = (\nu_1,\ldots,\nu_n): \cO_{X,\eta}\to \bZ^n \cup \{\infty\}$ to be the following function. For nonzero $f\in \cO_{X,\eta}$ whose lowest-order term with respect to $v$ is $c_{\beta_v(f)}x^{\beta_v(f)}$, we set
\[
\nu_\bullet(f) := (\beta_v(f), \nu_{Z_\bullet}(c_{\beta_v(f)})),
\]
where $\nu_{Z_\bullet}$ is the valuation associated to the admissible flag $Z_{\bullet}$ in Definition~\ref{defn:val_associated_to_flag}. When $f = 0$, we set $\nu_\bullet(0) = \infty$.
\end{defn}

We shall prove that $\nu_\bullet$ in Definition~\ref{defn:nu_bullet} is a $\bZ^n$-valuation, i.e., it satisfies the following properties:
\begin{itemize}
\item $\nu_\bullet(f) = \infty$ if and only if $f = 0$.
\item For any nonzero $f_1,f_2\in \cO_{X,\eta}$, $\nu_\bullet(f_1 + f_2) \geq \min \{\nu_\bullet(f_1), \nu_\bullet(f_2)\}$, where we order $\bN^n$ in a certain way capturing the valuation $v$ (see Lemma~\ref{lem:nu_bullet_super_triangle_ineq} for details).
\item For any nonzero $f_1,f_2\in \cO_{X,\eta}$, $\nu_\bullet(f_1f_2) = \nu_\bullet(f_1)+\nu_\bullet(f_2)$.
\end{itemize}
The first property follows from the definition of $\nu_\bullet$. The last two properties are proved in the following lemmas.

\begin{lem}\label{lem:nu_bullet_super_triangle_ineq}
For any nonzero $f_1,f_2\in \cO_{X,\eta}$, we have \[\nu_\bullet(f_1 + f_2) \geq \min \{\nu_\bullet(f_1), \nu_\bullet(f_2)\},\] where we order $\bZ^n = \bZ^r\times \bZ^{n-r}$ in the following way. We say that $(a_1,\ldots,a_n)\geq (b_1,\ldots,b_n)$ if
\[
\left(\sum_{i=1}^r \alpha_ia_i, a_{r+1},\ldots, a_n\right) \geq \left(\sum_{i=1}^r \alpha_ib_i, b_{r+1},\ldots, b_n\right)
\]
as elements in $\bR\times \bZ^{n-r}$, in which the order is taken lexicographically.
\end{lem}
\begin{proof}
We may assume that $\nu_\bullet(f_1)\geq \nu_\bullet(f_2)$. Then $v(f_1)\geq v(f_2)$ by the definition of the ordering on $\bZ^n$.

If $v(f_1) > v(f_2)$, then the lowest-order term, with respect to $v$, of $(f_1 + f_2)$ is equal to that of $f_2$. Then
\[
\nu_{\bullet}(f_1 + f_2) = \nu_\bullet(f_2) = \min \{\nu_\bullet(f_1), \nu_\bullet(f_2)\}.
\]

If $v(f_1) = v(f_2)$, then the lowest-order term of $f_i$ with respect to $v$ has the form $c_ix^\beta$, where $c_i\in K(Z)$ for $i=1,2$ and $\beta\in \bN^r$. If $c_1 + c_2 = 0$, then
\[
v(f_1 + f_2) > \alpha\cdot \beta = v(f_1),
\]
which implies the desired statement. Otherwise, the lowest-order term of $f_1 + f_2$ is $(c_1 + c_2)x^\beta$. It suffices to show that
\[
\nu_{Z_\bullet}(c_1 + c_2)  \geq \min \{\nu_{Z_\bullet}(c_1), \nu_{Z_\bullet}(c_2)\},
\]
and this follows from the fact that $\nu_{Z_\bullet}$ is a $\bZ^{n-r}$-valuation.
\end{proof}

\begin{lem}\label{lem:nu_bullet_is_additive}
For any nonzero $f_1,f_2\in \cO_{X,\eta}$, $\nu_\bullet(f_1f_2) = \nu_\bullet(f_1)+\nu_\bullet(f_2)$.
\end{lem}
\begin{proof}
Suppose the lowest-order term of $f_i$ with respect to $v$ is $c_ix^{\beta_i}$. Then the lowest-order term of $f_1f_2$ is $c_1c_2 x^{\beta_1 + \beta_2}$. It suffices to show that
\[
\nu_{Z_\bullet}(c_1 c_2)  = \nu_{Z_\bullet}(c_1) + \nu_{Z_\bullet}(c_2),
\]
which follows from the fact that $\nu_{Z_\bullet}$ is a $\bZ^{n-r}$-valuation.
\end{proof}

We have the following observation.
\begin{lem}\label{lem:c_beta_has_no_poles}
Suppose $f\in \cO_{X,z}$ and $c_\beta x^\beta$ is the lowest-order term of $f$ with respect to $v$. Then $c_{\beta}\in \cO_{Z,z}$.
\end{lem}

The following definition will be used in the proof of Lemma~\ref{lem:c_beta_has_no_poles}.
\begin{defn}\label{defn:integral_approx_of_qm_val}
Fix $f\in \cO_{X,\eta}$. Suppose that $v'\in \text{QM}_\eta(X,D)$ has an integral weight vector $(\alpha_1',\ldots,\alpha_r')\in \bZ_{>0}^r$. We say that $v'$ is \emph{an integral approximation of $v$ for $f$} if the following holds: for any $\beta\in \bN^r$ with $c_\beta \neq 0$, we have \[\alpha'\cdot \beta_v(f) \leq \alpha'\cdot \beta,\] and the equality holds if and only if $\beta = \beta_v(f)$. In other words, the lowest-order term of $f$ with respect to $v$ is the unique lowest-order term of $f$ with respect to $v'$.
\end{defn} 
An integral approximation of $v$ for $f$ always exists. In fact, there exists $\epsilon > 0$ such that the following statement holds: for every $\gamma\in \bQ^r$ and positive integer $N$ satisfying $|\gamma-\alpha|<\epsilon$ and $\alpha':= N\gamma\in \bZ^r$, $v_{\alpha'}$ is an integral approximation of $v$ for $f$.

\begin{proof}[Proof of Lemma~\ref{lem:c_beta_has_no_poles}]
After shrinking $X$ around $z$ we may assume that $X$ is affine and $D_i$ is cut out by a single function $x_i$ for all $1\leq i\leq r$. Let $\alpha'\in \bZ_{>0}^r$ be an integral vector such that $v' = v_{\alpha'}\in \QM_\eta(X,D)$ is an integral approximation of $v$ for $f$ (see Definition~\ref{defn:integral_approx_of_qm_val}). Let $\rho: \sX'\to X$ be the stack-theoretic weighted blow-up of $X$ along $Z$ with respect to $\{D_i:1\leq i\leq r\}$ and weights $\{\alpha_i':1\leq i\leq r\}$ (see Definition~\ref{defn:weighted_blow_up_along_one_irred_component}). Let ${\sE}$ be the exceptional divisor of $\rho$. By Lemma~\ref{lem:exc_div_is_weighted_proj_stack}, $\rho_\sE = \rho|_\sE: \sE\to Z$ is a weighted projective stack. From Section~\ref{subsec:weighted_blow_up_along_snc_div}, $\sX'$ has the following local charts
\[
\sU_i := \sD_+(x_it^{\alpha_i'}) = [\Spec R[(x_it^{\alpha_i'})^{-1}] / \bG_m]
\]
where $R = \cO_X[x_1t^{\alpha_1'},\ldots, x_rt^{\alpha_r'}]$ and $1\leq i\leq r$. The map $\rho_i:{\sU}_i\to X$ is induced by the inclusion $\cO_X\subseteq R[(x_it^{\alpha_i'})^{-1}]$. Set $g = f - c_\beta x^\beta\in K(X)$. Then we have $v(g) > v(f)$ and $v'(g) > v'(f)$. 

With a slight abuse of notation, we will still write $\rho_i^* h = h$ for any $h\in K(X)$. We can compute
\[
t^{v'(f)}\rho_i^*f = t^{v'(f)}(c_\beta x^\beta + g) = c_\beta \prod_{i=1}^r (x_it^{\alpha_i'})^{\beta_i} + t^{v'(f)-v'(g)} (gt^{v'(g)}).
\]
On ${\sU}_i$, the exceptional divisor ${\sE}$ is cut out by $t^{-1} = 0$, where $t^{-1} = \frac{x_it^{\alpha_i'-1}}{x_it^{\alpha_i'}}\in R[(x_it^{\alpha_i'})^{-1}]$. Note that $x_it^{\alpha_i'}$ is nonzero on ${\sE}$ since it does not divide the factor $t^{-1}$. On the other hand, $t^{v'(f)-v'(g)}$ divides the factor $t^{-1}$ and $gt^{v'(g)}\in R[(x_it^{\alpha_i'})^{-1}]$. Therefore, we have
\[
(t^{v'(f)}\rho_i^*f)|_{\sE} = \left(c_\beta \prod_{i=1}^r (x_it^{\alpha_i'})^{\beta_i}\right)\Bigg|_{{\sE}}.
\]
Passing to the corresponding divisors, we have
\begin{equation}\label{eqn:pullback_of_power_series_restricting_to_exc_div}
(\rho^* \prindiv(f) - v'(f) {\sE})|_{{\sE}} = \prindiv \left(c_\beta \prod_{i=1}^r (x_it^{\alpha_i'})^{\beta_i}\right)\Bigg|_{{\sE}} = \prindiv(\rho_\sE^*c_\beta) + \sum_{i=1}^r \beta_i{\sH}_i,
\end{equation}
where ${\sH}_i\subseteq {\sE}$ is the divisor cut out by $x_it^{\alpha_i'} = 0$. In fact, $\sH_i$ is the $i$th coordinate hyperplane bundle over $Z$. In particular, it is a horizontal divisor on ${\sE}$ over $Z$. On the other hand, $\prindiv(\rho_\sE^*c_\beta) = \rho_\sE^*\prindiv(c_\beta)$ is a vertical divisor on ${\sE}$ over $Z$. Since the left hand side of $\eqref{eqn:pullback_of_power_series_restricting_to_exc_div}$ is an effective divisor, the right hand side is also effective. This forces $\prindiv(c_\beta)$ to be effective, i.e., $c_\beta$ is a regular function on $Z$.
\end{proof}

The definition of $\nu_\bullet$ naturally extends to sections of line bundles on $X$. Let $L$ be a line bundle on $X$ and $s\in H^0(X, L)$. Suppose locally near $z$ where $L$ is trivialized, $s$ is defined by a function $f\in \cO_{X,z}$.
Then we set
\[
\nu_\bullet(s) := \nu_\bullet(f).
\]
This definition does not depend on the choice of $f$. Indeed, suppose $f'$ is another function defining $s$ near $z$. Then $f/f'$ is a unit in $\cO_{X,z}$. In particular, $\nu_\bullet(f/f') = 0\in \bZ^n$. Hence,
\[
\nu_{\bullet}(f') = \nu_{\bullet}(f) + \nu_{\bullet}\left(\frac{f}{f'}\right) = \nu_{\bullet}(f).
\]
For an effective Cartier divisor $B$ on $X$ which is defined by a function $f\in \cO_{X,z}$ near $z$, we can similarly set
\[
\nu_\bullet(B) := \nu_\bullet(f).
\]
These definitions are compatible, i.e., for any $s\in H^0(X,L)$, we have $\nu_\bullet(s) = \nu_\bullet (\prindiv s)$. By Lemma~\ref{lem:nu_bullet_super_triangle_ineq} and Lemma~\ref{lem:nu_bullet_is_additive}, $\nu_\bullet$ satisfies the following properties:
\begin{itemize}
\item For any $s_1,s_2\in H^0(X, L)$, we have $\nu_\bullet(s_1 + s_2) \geq \min \{\nu_\bullet(s_1), \nu_\bullet(s_2)\}$, where we order $\bZ^n$ as in Lemma~\ref{lem:nu_bullet_super_triangle_ineq}.
\item For any $s_1\in H^0(X, L_1)$ and $s_2\in H^0(X, L_2)$, we have $\nu_\bullet(s_1\otimes s_2) = \nu_\bullet(s_1)+\nu_\bullet(s_2)$.
\end{itemize} 

We have the following quick corollary of Lemma~\ref{lem:c_beta_has_no_poles}.
\begin{cor}\label{cor:nu_bullet_is_nonnegative}
    Let $L$ be a line bundle on $X$ and $s\in H^0(X,L)$. Then $\nu_\bullet(s)\in \bN^n$. 
\end{cor}

We are now ready to define the Newton--Okounkov body of $L$ with respect to $\nu_\bullet$.
\begin{defn}\label{defn:okounkov_bodies_wrt_nu_bullet}
For $m\in \bN$, define 
\[\Gamma_m(\nu_\bullet, L) = \{\nu_\bullet(s): s\in H^0(X,mL)\setminus\{0\}\},
\]
which is a subset of $\bN^n$ by Corollary~\ref{cor:nu_bullet_is_nonnegative}, and
\[
\Gamma(\nu_\bullet, L) = \{(\nu_\bullet(s), m): s\in H^0(X,mL)\setminus\{0\}, m\in \bN\}\subseteq \bN^n \times \bN.\]
Let $\Sigma(\nu_\bullet, L)\subseteq \bR^{n+1}$ be the closed convex cone generated by $\Gamma(\nu_\bullet, L)$.
\emph{The Newton--Okounkov body of $L$ with respect to $\nu_\bullet$} is the convex set
\[
\Delta(\nu_\bullet, L) = \Sigma(\nu_\bullet, L)\cap (\bR^n \times \{1\}).
\]
Alternatively, $\Delta(\nu_\bullet, L)$ is the closed convex hull of $\bigcup_{m=1}^\infty \frac{1}{m}\Gamma_m(\nu_\bullet, L)$ in $\bR^n$.
\end{defn}

\subsection{Properties of Newton--Okounkov bodies}
We follow the set-up at the beginning of Section~\ref{subsec:construction_of_okounkov_body}. Recall that $v$ is a quasi-monomial valuation in $\QM_\eta(X,D)$ with a weight vector $\alpha$ and $Z_\bullet$ is an admissible flag on $Z$, from which we define a $\bZ^n$-valuation $\nu_\bullet : K(X)\to \bZ^n$.

Throughout this subsection, we assume that $L$ is a \emph{big} line bundle on $X$. We shall prove that the integral points and convex bodies defined in Definition~\ref{defn:okounkov_bodies_wrt_nu_bullet} satisfy the following properties.

\begin{thm}[cf. \cite{LM09}*{Lemma 2.2}]\label{thm:conditions_ABC_for_okounkov_bodies}
Suppose $L$ is a big line bundle on $X$.  Following  the notations in Definition~\ref{defn:okounkov_bodies_wrt_nu_bullet}, we have 
\begin{itemize}
\item[(a)] $\Gamma_0(\nu_\bullet, L) = \{0\} \subseteq \bN^n$;
\item[(b)] $\#\Gamma_m(\nu_\bullet, L) = \dim H^0(X, mL)$;
\item[(c)] $\Delta(\nu_\bullet, L)$ is compact;
\item[(d)] there exist finitely many vectors $(v_i, 1)$ spanning a submonoid $M\subseteq \bN^{n+1}$ such that $\Gamma(\nu_\bullet, L)\subseteq M$;
\item[(e)] $\Gamma(\nu_\bullet, L)$ generates $\bZ^{n+1}$ as a group.

\end{itemize}
\end{thm}

Note that these properties are precisely Lemma 1.3 and conditions (2.3)-(2.5) in \cite{LM09}. By \cite{LM09}*{Theorem 2.13} or \cite{Bou14}*{Theorem 0.2}, these conditions guarantee that the Euclidean volume of the convex body $\Delta(\nu_\bullet, L)$ is proportional to the volume of $L$:
\begin{cor}\label{cor:volume_equality_of_okounkov_bodies}
We have
\[
\frac{1}{n!}\vol(L) = \vol_{\bR^n}(\Delta(\nu_\bullet, L)).
\]
\end{cor}

\begin{proof}[Proof of Theorem~\ref{thm:conditions_ABC_for_okounkov_bodies}]
(a) This is clear from Definition~\ref{defn:okounkov_bodies_wrt_nu_bullet}. \\

\noindent(b) It suffices to prove the following statement. Let $s_1,s_2\in H^0(mL)$ be two nonzero sections such that $\nu_\bullet(s_1) = \nu_\bullet(s_2)$. Then there exists a non-trivial $\bC$-linear combination $s = \lambda_1s_1 + \lambda_2s_2$ of $s_1,s_2$ such that $\nu_\bullet(s) > \nu_\bullet(s_1)$, where we order $\bZ^{n}$ as in Lemma~\ref{lem:nu_bullet_super_triangle_ineq}.

Since $v(s_1) = v(s_2)$, the lowest-order term of $s_i$ has the form $c_ix^\beta$ for some nonzero $c_i\in \cO_{Z,z}$ (by Lemma~\ref{lem:c_beta_has_no_poles}) and $\beta\in \bN^r$. After possibly shrinking $X$, we may assume that $c_1,c_2\in \cO_Z(Z)$ and that each $Z_i\subseteq Z_{i-1}$ is a principal Cartier divisor cut out by a function $z_i$. Suppose there exist nonzero $\lambda_1,\lambda_2\in \bC$ such that $\nu_{\bullet\geq r+1}(\lambda_1c_1 + \lambda_2c_2) > \nu_{\bullet\geq r+1}(c_1)$. Then $s = \lambda_1c_1 + \lambda_2c_2$ would give $\nu_\bullet(s) > \nu_\bullet(s_1)$, as desired. Therefore, we may assume that $\nu_{\bullet\geq r+1}(\lambda_1c_1 + \lambda_2c_2) = \nu_{\bullet\geq r+1}(c_1)$ for any nonzer $\lambda_1,\lambda_2\in\bC$.

For any $g\in \cO_Z(Z)$ and $0\leq i\leq n-r$, we define $g^{(i)}\in \cO_{Z_i}(Z_i)$ inductively as follows. When $i = 0$, set $g^{(0)} = g$. For $1\leq i\leq n-r$, define $g^{(i)}\in \cO_{Z_i}(Z_i)$ as the image of $g^{(i-1)}/z_i^{e_i}$ under the restriction map $\cO_{Z_{i-1}}(Z_{i-1}) \to \cO_{Z_{i}}(Z_i)$, where $e_i = \ord_{z_i}(g^{(i-1)})$.

Let $c = \lambda_1c_1 + \lambda_2c_2$. By the assumption that $\nu_{\bullet\geq r+1}(c) = \nu_{\bullet\geq r+1}(c_1)$, we have
\[
c^{(i)} = \lambda_1 c_1^{(i)} + \lambda_2 c_2^{(i)} \neq 0 \in \cO_{Z_i}(Z_i)
\]
for all $0\leq i\leq n-r$. However, since $Z_{n-r} = \{z\}$ is a point, $c_1^{(n-r)}, c_2^{(n-r)}\in \bC$. Therefore, we can find nonzero $\lambda_1,\lambda_2\in \bC$ such that
\[
c^{(n-r)} = \lambda_1 c_1^{(n-r)} + \lambda_2 c_2^{(n-r)} = 0,
\]
which gives a contradiction. This concludes the proof of (b). \\

\noindent(c) By part (b), we have
\[
\lim_{m\to\infty} \frac{\Gamma_m(\nu_\bullet, L)}{m^n} = \lim_{m\to\infty} \frac{\dim H^0(X,mL)}{m^n} = n!\vol(L) <\infty.
\]
Then the compactness of $\Delta(\nu_\bullet, L)$ follows from \cite{Bou14}*{Corollary 1.14}. \\

\noindent (d) The proof of \cite{LM09}*{Lemma 2.2} works in our setting as well, because the compactness of $\Delta(\nu_\bullet, L)$ is established in part (b). \\

\noindent (e) We mimic the proof of \cite{LM09}*{Lemma 2.2}. We write $L\sim A-B$ as the difference of two very ample divisors $A,B$ on $X$. By possibly adding a further very ample divisor to both $A$ and $B$, we may assume that there exist sections $s_0\in H^0(X, A)$ and $t_i\in H^0(X,B)$ for $0\leq i\leq d$, such that $\nu_\bullet(s_0) = \nu_\bullet(t_0) =0$ and $\nu_\bullet(t_i) = e_i$ (the $i$th standard basis vector) for $1\leq i\leq d$. In fact, for $1\leq i\leq r$, we can take $t_i$ such that $\prindiv(t_i) = D_i + G_i$ for some $G_i\sim B- D_i$ not containing $z$; for $i\geq r+1$, we can take $t_i$ which does not vanish on $Z_{i-1-r}$ and vanishes simply on $Z_{i-r}$ in a neighborhood of $z$.

Since $L$ is big, for a sufficiently large integer $m$, $mL-B$ is linearly equivalent to some effective divisor $F_m$. Since $\prindiv(t_i) + F_m \sim B + F_m\sim mL$, by Lemma~\ref{lem:nu_bullet_is_additive}, we have
\[
(\nu_\bullet(F_m), m), (\nu_\bullet(F_m) + e_i, m) \in \Gamma(\nu_\bullet, L)
\]
for every $1\leq i\leq n$. On the other hand, $\prindiv(s_0) + F_m \sim A+F_m \sim (m+1)L$, so
\[
(\nu_\bullet(F_m), m+1)\in \Gamma(\nu_\bullet,L).
\]
Then it is clear that $\Gamma(\nu_\bullet, L)$ generates $\bZ^{n+1}$ as a group.\\
\end{proof}

\subsection{Variants of Newton--Okounkov bodies}
In this section, we describe some variants of the construction in Definition~\ref{defn:okounkov_bodies_wrt_nu_bullet}. 

\subsubsection{Normal crossing divisors with zero dimensional strata}\label{subsubsec:okunkov_body_wrt_nc_divisors_with_zero_dim_strata}
We note that this construction is also studied in \cite{Bou14}*{Section 2.3}.

Let $X$ be a smooth variety and $D$ a normal crossing divisor on $X$. Assume that $D$ has a zero dimensional stratum, i.e., around a point $p\in X$, $D$ is analytically locally given by $x_1x_2\cdots x_n = 0$, where $x_i\in \hat{\cO}_{X,p}.$ 

Let $v\in \text{QM}_p(X, D)$ with a weight vector $\alpha\in \bR_{>0}^r$ (see Remark~\ref{remark:qm_val_wrt_nc_divisors})  such that $\alpha_1,\ldots,\alpha_n$ are $\bQ$-linearly independent. Then for any $f\in \cO_{X,p}$, one can define
\[
\nu_\bullet(f) = \text{index of the lowest-order term of $f$ with respect to $v$}.
\]
Then, one can define the Newton--Okounkov body $\Delta(\nu_\bullet, L)$ of a line bundle $L$ with respect to $\nu_\bullet$ as in Definition~\ref{defn:okounkov_bodies_wrt_nu_bullet}. When $L$ is big, one can prove Theorem~\ref{thm:conditions_ABC_for_okounkov_bodies} for $\Delta(\nu_\bullet, L)$ using the same approach. (In fact, the proof is easier because the boundedness of $\Delta(\nu_\bullet, L)$ simply follows from the fact that $v$ has linear growth.)

\subsubsection{The singular case}
Let $X$ be a (possibly singular) quasi-projective variety and $L$ be a line bundle on $X$. Let $v$ be a quasi-monomial valuation on $X$. We can construct a Newton--Okounkov body of $L$ with respect to $v$ as follows.
\begin{defn}\label{defn:okounkov_body_built_on_log_resolution}
Let $\mu: (Y,E_1+E_2+\ldots+E_r)\to X$ be a partial resolution of $X$ such that the following holds:
\begin{itemize}
\item there exists an irreducible component $Z$ of $\bigcap_{i=1}^r E_i$ such that the pair $(Y, E_1+E_2+\ldots+E_r)$ is log smooth in a neighborhood of $Z$;
\item $v\in {\QM}_{Z}(Y,E_1+\ldots + E_r)$ has weight vector $\alpha\in \bR_{>0}^r$, such that $\alpha_1,\ldots,\alpha_r$ are $\bQ$-linearly independent.
\end{itemize}
Let $Z = Z_0 \supset Z_1 \supset \cdots\supset Z_{n-r} = \{z\}$ be an admissible flag on $Z$. Let $\nu_\bullet$ be the $\bZ^n$-valuation defined using $v$ and $Z_\bullet$ (see Definition~\ref{defn:nu_bullet}).
Then we say that $\Delta(\nu_\bullet, \mu^*L)$ is \emph{a Newton--Okounkov body of $L$ with respect to $v$}.
\end{defn}

Of course, there might be many different Newton--Okounkov bodies of $L$ with respect to $v$, depending on the partial resolution $Y\to X$ and the admissible flag $Z_\bullet$. However, certain choices of partial resolutions and admissible flags correspond to the same Newton--Okounkov body up to an integral linear transformation.

\begin{lem}\label{lem:lin_transform_okounkov_body_under_further_blow_ups}
Let $(Y, E = E_1 + \ldots + E_r)$ be a log smooth pair and $Z$ be an irreducible component of $\bigcap_{i=1}^r E_i$. Let $\phi: Y'\to Y$ be a composition of blow-ups of $Y$ along strata of $E$ (and their total transforms). Suppose $v\in \text{QM}_Z(Y, E)$ has weight vector $\alpha\in \bR_{>0}^r$ such that $\alpha_1,\ldots,\alpha_r$ are $\bQ$-linearly independent. Then the following statements hold.
\begin{itemize}
\item[(a)] There exist prime divisors $E_1',\ldots, E_r'$ which are irreducible components of $\phi^*E$ and an irreducible component $Z'\subseteq \bigcap_{i=1}^r E_i'$, such that \[v\in \text{QM}_{Z'}(Y', E' = E_1' + \cdots + E_r')\]
and $\phi$ maps $Z'$ isomorphically to $Z$.
\item[(b)] Suppose for $1\leq i\leq r$, $\ord_{E_i'}\in {\QM}_Z(Y,E)$ has weight vector $w^{(i)}\in \bN^r$. Then $v\in {\QM}_{Z'}(Y', E')$ has weight vector $W^{-1}\alpha$, where $W$ is the $r\times r$ matrix whose $i$th column is $w^{(i)}$.
\item[(c)] Let $Z_i'$ be the preimage of $Z_i$ under the isomorphism $Z'\to Z$. Then for any $f\in \cO_{Y,\eta}$, we have
\[
\nu_{Z_\bullet'}(\phi^*f) = \nu_{Z_\bullet}(f) \quad\text{and}\quad \beta_v(\phi^*f) = W^T \beta_v(f).
\]
In particular, we have $\nu_\bullet'(\phi^*f) = M \nu_\bullet(f)$, where $M$ is the block-diagonal $n\times n$ matrix
\[
M = \begin{pmatrix}
W^T & 0 \\
0 & \id_{n-r}
\end{pmatrix}\in \GL_n(\bZ)
\]
and $\nu_\bullet'$ is the $\bZ^n$-valuation on $Y'$ defined using $v$ and $Z_\bullet'$.
\item[(d)] For any line bundle $L$ on $Y$, we have 
\[
\Delta(\nu_\bullet', \phi^*L) = M \Delta(\nu_\bullet, L).
\]
\end{itemize}
\end{lem}
\begin{proof}
\noindent (a)  We may assume that $\phi:Y'\to Y$ is the blow-up along a single stratum $V$ of $E$, which is an irreducible component of $\bigcap_{i=1}^s E_i$ for some $s\leq r$. Let $E_0'$ be the exceptional divisor of $\phi$. Then $\phi_0= \phi|_{E_0'}: E_0'\to V$ is a projective bundle. Let $E_i'$ be the strict transform of $E_i$. When $1\leq i\leq s$, $E_i'\cap E_0'\subseteq E_0'$ is the $i$th coordinate hyperplane bundle over $V$. When $s+1\leq i\leq r$, $E_i'\cap E_0'$ is the preimage under $\phi_0$ of the unique irreducible component of $E_i\cap V$ containing $Z$. In particular, for every $j\in \{1,2,\ldots, s\}$, we have
\[
\bigcap_{0\leq i\leq s, i\neq j} E_i' = \text{the $j$th coordinate section over $V$},
\]
and 
\[
\bigcap_{i=s+1}^r E_i' = \phi_0^{-1}(Z).
\]
For $1\leq j\leq s$, we denote
\[
Z_j' := \bigcap_{0\leq i\leq r, i\neq j} E_i' = \text{the $j$th coordinate section over $Z$}.
\]
Then
\[
\QM_Z(Y,E) = \bigcup_{j=1}^s \QM_{Z_j'} \left(Y', \sum_{0\leq i\leq r, i\neq j} E_i'\right).
\]
(In fact, the right-hand side is the stellar subdivision of $\QM_\eta(Y,E)\cong \bR_{\geq 0}^r$ centered at the point $(1^s, 0^{r-s})\in \bR^r$.) Thus, there is an index $j\in \{1,2,\ldots,s\}$ such that 
\[
v\in  \QM_{Z_j'} \left(Y', \sum_{0\leq i\leq r, i\neq j} E_i'\right),
\]
as desired.\\

\noindent (b) Let $\alpha'$ be the weight of $v\in \QM_{Z'}(Y',E')$. For any $1\leq j\leq r$, we have
\[
\phi^*E_j = \sum_{i=1}^r w^{(i)}_j E_i' + \text{(other components)}.
\]
As a result,
\[
\alpha_j = v(E_j) = v\left(\sum_{i=1}^r w^{(i)}_j E_i'\right) = \sum_{i=1}^r w_j^{(i)} \alpha_i',
\]
which implies that $\alpha = W\alpha'$. \\

\noindent (c) Suppose for $1\leq j\leq r$, $E_j$ is cut out by $y_j\in \cO_{Y,\eta}$ and $E_j'$ is cut out by $y_j'\in \cO_{Y', \eta'}$. Then we have
\[
\phi^*y_j = \prod_{i=1}^r (y_i')^{w^{(i)}_j},
\]
and hence $\phi^* y^\beta = (y')^{W^T\beta}$ for every $\beta\in \bN^r$. In particular, if $cy^\beta$ is the lowest-order term of $f$ with respect to $v$, then the lowest-order term of $\phi^*f$ is $c(y')^{W^T\beta}$ (after identifying $K(Z') = K(Z)$). This proves (c).\\

\noindent(d) This follows directly from (c).
\end{proof}

\subsubsection{Graded linear series}\label{subsubsec:okounkov_body_for_graded_linear_series}
Let $W_\bullet$ be a graded linear series associated to a big line bundle $L$. One can take the same definition of $\nu_\bullet$ as in Definition~\ref{defn:nu_bullet} and define a Newton--Okounkov body $\Delta(\nu_\bullet, W_\bullet)$ after replacing each $|mL|$ in Definition~\ref{defn:okounkov_bodies_wrt_nu_bullet} by the sublinear series $W_m$. By \cite{Bou14}*{Theorem 0.2}, we have the following statement.
\begin{prop}\label{prop:volume_equality_for_graded_linear_series}
    Assume $W_\bullet$ contains an ample series (see Definition~\ref{defn:graded_linear_series_and_contains_ample_series}). Then Theorem~\ref{thm:conditions_ABC_for_okounkov_bodies} holds for $W_\bullet$ and we have a volume equality
    \[
    \vol_{\bR^n}(\Delta(\nu_\bullet, W_\bullet)) = \lim_{m\to\infty} \frac{\dim W_m}{m^n} = \frac{1}{n!} \vol(W_\bullet).
    \]
\end{prop}

\subsubsection{Multiple line bundles}
Suppose that we are in the same setting as Section~\ref{subsec:construction_of_okounkov_body}, but instead of having a single big line bundle $L$, we have $s\geq 1$ big line bundles $L_1,L_2,\ldots,L_s$. For each $\vec{m} = (m_1,\ldots,m_s)\in \bN^s$, we can similarly define a function $\nu_\bullet: H^0(X, \Vec{m}L) \to \bN^n\cup \{\infty\}$, where $\vec{m}L := m_1L_1 + \ldots + m_sL_s$.

\begin{defn}\label{defn:multiple_line_bundles}
For $\vec{m}\in \bN^s$, let $\Gamma_{\vec{m}}(\nu_\bullet, \{L_i:1\leq i\leq s\})\subseteq \bN^n$ be the image of $\nu_\bullet: |\vec{m}L| \to \bN^n$ and
\[
\Gamma(\nu_\bullet, \{L_i:1\leq i\leq s\}) = \{(p,\vec{m}): p\in \Gamma_m(\nu_\bullet, \{L_i:1\leq i\leq s\}), \vec{m}\in \bN^s\}\subseteq \bN^n\times \bN^s.
\]
The \emph{Newton--Okounkov body of $\{L_i:1\leq i\leq s\}$ with respect to $\nu_\bullet$} is the closed convex cone generated by 
$\Gamma(\nu_\bullet, \{L_i:1\leq i\leq s\})$, denoted as $\Delta(\nu_\bullet, \{L_i:1\leq i\leq s\})\subseteq \bR^{n+s}$. 
\end{defn}

Similar to \cite{LM09}*{Theorem 4.19}, we have the following proposition.
\begin{prop}\label{prop:slices_of_global_okounkov_bodies_mult_line_bundles}
Let $p: \Delta(\nu_\bullet, \{L_i:1\leq i\leq s\})\subseteq \bR^{n+s}\to \bR^{s}$ denote the projection to the last $s$ coordinates. Then for any $\vec{a}\in \bQ^s$, we have
\[
\Delta(\nu_\bullet, \vec{a}L) = {p}^{-1}(\vec{a}).
\]
Here, $\Delta(\nu_\bullet, \vec{a}L) = \frac{1}{N}\Delta(\nu_\bullet, N\vec{a}L)$ for a sufficiently divisible $N\in \bN$.
\end{prop}

\begin{rem}
    Similar to \cite{LM09}*{Section 4.3}, one can generalize Definition~\ref{defn:multiple_line_bundles} and Proposition~\ref{prop:slices_of_global_okounkov_bodies_mult_line_bundles} for multi-graded linear series that contains an ample series (see Definition 4.15 and 4.17 in \cite{LM09}). 
\end{rem}

\subsection{Proof of Theorem~\ref{mainthm:okounkov_bodies_wrt_qm_val}}
Following Definition~\ref{defn:okounkov_body_built_on_log_resolution}, we take a log resolution $\mu: (Y,E_1+\cdots +E_r)\to X$ such that $v\in \QM_Z(Y,E_1+\cdots+E_r)$ is a quasi-monomial valuation with a weight vector $\alpha\in\bR_{>0}^r$. Let $Z_\bullet$ be an admissible flag on $Z$ and $\nu_\bullet$ be the $\bZ^n$-valuation defined using $v$ and $Z_\bullet$ (see Definition~\ref{defn:nu_bullet}). For each $t\in [0, T(v))$, let $\Delta^t = \Delta(\nu_\bullet, V_\bullet^t(v))$ be the Newton--Okounkov body of the graded sublinear series $V_\bullet^t(v)$ in Section~\ref{subsubsec:okounkov_body_for_graded_linear_series}.

\begin{proof}[Proof of Theorem~\ref{mainthm:okounkov_bodies_wrt_qm_val}]
(a) This follows from the fact that
\[
\Gamma_m(V_\bullet^t(v)) = \Gamma_m(L) \cap \left\{x\in \bR^n: \sum_{i=1}^r\alpha_ix_i \geq mt\right\}
\]
for every $m\in \bN$. \\

(b) Suppose $m$ is a positive integer such that $mD\in |mL|$. Suppose $E_i$ is locally defined by a function $y_i$ for $1\leq i\leq r$. Suppose $\mu^*(mD)$ is locally defined by a function $f$, whose lowest-order term with respect to $v$ is given by $cy^\gamma$ with $c\in K(Z)$ and $\gamma\in\bN^r$. Let $\beta = \frac{\gamma}{m}\in\bQ_{\geq 0}^r$. Then $v(D) = \alpha\cdot \beta$. The uniqueness of $\beta$ follows from the fact that $\alpha_1,\ldots,\alpha_r$ are $\bQ$-linearly independent. The rest of the statement follows from our construction of $\Delta^t$. \\

(c) By Lemma~\ref{lem:V_bullet^t_contains_ample_linear_series}, $V_\bullet^t(v)$ contains an ample linear series. Then by Proposition~\ref{prop:volume_equality_for_graded_linear_series}, we have
\[
\vol_{\bR^n}(\Delta^t) = \frac{1}{n!} \vol(V_\bullet^t(v)).
\]
\end{proof}

\section{Valuations computing stability thresholds}\label{sec:qm_vals_computing_stability_invariants}
Recall the $\delta^\tau$-invariants from Definition~\ref{defn:delta_tau}. The main goal of this section is to prove the following theorem.
\begin{thm}\label{thm:qm_val_computing_delta_tau}
Let $(X,D)$ be a projective klt pair. Let $L$ be a big line bundle on $X$ and $\tau\in [0,1]$. Then there exists a quasi-monomial valuation computing $\delta^\tau(L)$.
\end{thm}

\subsection{Multiplier ideals}
We first recall definitions and properties of multiplier ideals from \cite{Laz04b}. Let $(X,D)$ be a log pair and $\fa$ be a nonzero ideal on $X$. Let $\mu:Y\to X$ be a log resolution of $(X,D,\fa)$ and write $\fa \cdot \cO_Y = \cO_Y(-F)$. For $c>0$, \emph{the multiplier ideal $\cJ(X,D, c\cdot \fa)$} is
\[
\cJ(X,D, c\cdot \fa):= \mu_*\cO_Y(\lceil K_{Y} -\mu^*(K_X+D) - cF\rceil) \subseteq \cO_X.
\]

Let $\fa_\bullet$ be a graded ideal sequence on $X$ and $c>0$. Note that
\[
\cJ\left(X,D, \frac{c}{p}\cdot \fa_p\right) \subseteq \cJ\left(X,D, \frac{c}{q}\cdot \fa_q\right)
\]
for all positive integers $p,q$ such that $p$ divides $q$. This and the Noetherianity of $X$ imply that 
\[
\left\{\cJ\left(X,D, \frac{c}{p}\cdot \fa_p\right): p\in \bZ_{>0}\right\}
\]
has a maximal element, which we denote as $\cJ(X,D, c\cdot \fa_\bullet)$.

Let $L$ be a line bundle on $X$ and $V_\bullet$ be a graded linear series of $L$. Let $\fb_\bullet$ be the graded sequence of ideals on $X$ given by $\fb_m = \fb(|V_m|)$, where we use $\fb(|V|)$ to denote the base ideal of a linear system $V$. Then for any $c>0$, we set
\[
\cJ(X, D, c\cdot ||V_\bullet||) := \cJ(X, D, c\cdot \fb_\bullet).
\]

We recall the following properties of multiplier ideals.
\begin{lem}[cf. \cite{BJ20}*{Lemma 5.8}]\label{lem:BJ_Lemma_5.8}
    Suppose $(X,D)$ is a projective klt pair. Let $L$ be a line bundle on $X$ and $V_\bullet$ be a graded linear series of $L$. 
    \begin{itemize}
        \item [(a)] For every $m\in \bN$, we have \[\fb(|V_m|)\subseteq \cJ(X,D, m\cdot ||V_\bullet||).\]
        \item [(b)] There exists an ideal $J_{X,D}\subseteq \cO_X$ depending only on $X$ and $D$ such that for every $m\in \bN$ and $c>0$, we have
        \[
        (J_{X,D})^{m-1} \cdot \cJ(X,D, cm\cdot ||V_\bullet||) \subseteq \cJ(X,D, c\cdot ||V_\bullet||)^m.
        \]
    \end{itemize}
\end{lem}
\begin{proof}
(a) When $D = 0$, it follows from \cite{BJ20}*{Lemma 5.8(ii)}. The same proof works for the general case. \\

\noindent (b) Let $g: X'\to X$ be a small $\bQ$-factorial modification as in \cite{Kol13}*{Corollary 1.37}. Let $\fa' = \fa\cdot \cO_{X'}$ and $\fb' = \fb\cdot \cO_{X'}$. We show that 
\begin{align}\label{eqn:subadditivity_multiplier_ideal_on_Q_fact_model}
\Jac_{X'} \cdot \cO_{X'}(-\lceil D'\rceil) \cdot \cJ(X',D', c\cdot (\fa'\cdot\fb')) \subseteq \cJ(X',D', c\cdot \fa')\cdot \cJ(X',D', c\cdot \fb'),
\end{align}
where $\Jac_{X'}$ is the Jacobian ideal of $X'$ and $D'$ is the strict transform of $D$ on $X'$.

Indeed, since $X'$ is $\bQ$-Gorenstein, by \cite{Tak06}*{Theorem 2.9}, we have
\[
\Jac_{X'} \cdot \cJ(X',2D', c\cdot (\fa'\cdot\fb')) \subseteq \cJ(X',D', c\cdot \fa')\cdot \cJ(X',D', c\cdot \fb').
\]
From the definition of multiplier ideals, it is clear that
\[
\cO_{X'}(-\lceil D'\rceil) \cdot \cJ(X',D', c\cdot (\fa'\cdot\fb')) \subseteq \cJ(X',2D', c\cdot (\fa'\cdot\fb')).
\]
Combining these two inclusions together, we obtain \eqref{eqn:subadditivity_multiplier_ideal_on_Q_fact_model}. 

Note that $K_{X'} + D' = g^*(K_X+D)$. Applying $g_*$ to \eqref{eqn:subadditivity_multiplier_ideal_on_Q_fact_model}, we have
\begin{align}\label{eqn:subadditivity_multiplier_ideals}
J_{X,D} \cdot \cJ(X,D, c\cdot (\fa\cdot\fb)) \subseteq \cJ(X,D, c\cdot \fa)\cdot \cJ(X,D, c\cdot \fb),
\end{align}
where $J_{X,D} = g_*(\Jac_{X'} \cdot \cO_{X'}(-\lceil D'\rceil)) \subseteq \cO_X$ depends only on $X$ and $D$. Finally, the statement follows from repeatedly applying \eqref{eqn:subadditivity_multiplier_ideals}, with $\fa$ and $\fb$ being suitable powers of the base ideal of $V_k$ for sufficiently divisible $k$.
\end{proof}

\begin{comment}
\begin{lem}\cite{BJ20}*{Proposition 5.11}\label{lem:BJ_prop_5.11}
Let $X$ be a projective variety with klt singularities and $L$ be a line bundle on $X$. Let $v$ be a valuation on $X$ such that $A(v) <\infty$. Let $m\in \bN$ and $t\in\bQ_{>0}$ such that $mt \geq A(v)$. Then
\[
\cJ(X, m\cdot ||V_\bullet^t(v)|| ) \subseteq \fa_{mt-A_X(v)}(v).
\]
Here, $V_\bullet^t(v) = V_\bullet^t(\cF_v)$ is the graded linear series in Definition~\ref{defn:graded_linear_series_induced_by_filtrations} and $\fa_\lambda(v) = \{f\in \cO_X: v(f) \geq \lambda\}$.
\end{lem}
\end{comment}

\begin{lem}[cf. \cite{BJ20}*{Corollary 5.10}]\label{lem:BJ_cor_5.10}
Let $(X,D)$ be a projective klt pair and let $L$ be a big line bundle on $X$. Then there exists a very ample line bundle $B$ such that for any graded linear series $V_\bullet$ of $L$ and any $m\in \bN$,
\[
(mL +B) \otimes \mathcal{J}(X, D, m\cdot ||V_\bullet||)
\]
is globally generated. Furthermore, we may choose $B$ such that $H^0(X, B\otimes J_{X,D})\neq 0$, where $J_{X,D}$ is the ideal in Lemma~\ref{lem:BJ_Lemma_5.8}(b).
\end{lem}
\begin{proof}
    We take a very ample line bundle $H$ on $X$ such that $H-(K_X+D)$ is ample. By \cite{BJ20}*{Theorem 5.9(ii)}, the coherent sheaf
\[
(mL+(n+1)H)\otimes \mathcal{J}(X, m\cdot ||V_\bullet||)
\]
is globally generated. Thus, we can take $B = bH$, where $b\geq n+1$ is a positive integer such that $H^0(X,B\otimes J_{X,D})\neq 0$.
\end{proof}

\subsection{Compatible filtrations on auxiliary linear series}\label{subsec:compatible_filtrations_on_aux_linear_series}
Let $(X,D)$ be a projective klt pair and $L$ be a big line bundle on $X$. Let $B$ be the very ample line bundle from Lemma~\ref{lem:BJ_cor_5.10}.

Let 
\[
R = \bigoplus_{m\in \bN} R_m = \bigoplus_{m\in \bN} H^0(X, mL)
\]
and
\[
M = \bigoplus_{m\in \bN} M_m = \bigoplus_{m\in \bN} H^0(X, mL+B).
\]
Then $M$ is a graded $R$-module. Let $\cF$ be a filtration on $R$ (see Definition~\ref{defn:mult_filt}).
\begin{defn}[cf. Definition~\ref{defn:mult_filt}]\label{defn:compatible_filtration}
We say that ${\cG}$ is \emph{a filtration on $M$ compatible with ${\cF}$} if the following holds:
\begin{itemize}
\item[(F1)] (Decreasing) $\cG^\lambda M_m \subseteq \cG^{\lambda'} M_m$ for all $\lambda > \lambda'$;
\item[(F2)] (Left-continuous) $\cG^\lambda M_m = \bigcap_{\lambda'<\lambda} \cG^{\lambda'}M_m$;
\item[(F3)] (Bounded) $\cG^0M_m = M_m$ and $\cG^\lambda M_m = 0$ for $\lambda \gg 0$;
\item[(F4')] (Compatible with $\cF$) $\cF^\lambda R_m\cdot \cG^{\lambda'}M_{m'}\subseteq \cG^{\lambda+\lambda'} M_{m+m'}$ for all $\lambda,\lambda'\in \bR$ and $m,m'\in \bN$.
\end{itemize}
\end{defn}

\begin{expl}\label{example:compatible_filtration_induced_by_valuation}
    Suppose $v$ is a valuation on $X$ and $\cF = \cF_v$ is the filtration on $R$ induced by $v$. Then the filtration $\cG_v$ given by
    \[
    \cG_v^\lambda M_m = \{s\in H^0(X,mL+B): v(s)\geq \lambda\}
    \]
    is a filtration on $M$ compatible with $\cF_v$.
\end{expl}

\begin{defn}[cf. \cite{BJ20}*{Section 5.1}]
    Let $\cG$ be a filtration on $M$ compatible with $\cF$. For $m,\ell\in \bN$ and $t\geq 0$, we define
    \[
    W_m^t(\cG):= \cG^{mt}M_m \subseteq H^0(X,mL+B)
    \]
    and
    \[
    W_{m,\ell}^t(\cG) := H^0(X, \ell(mL+B) \otimes \overline{\fb(|W_m^t(\cG)|)^\ell})\subseteq H^0(X, \ell(mL+B)),
    \]
    where $\overline{\fa}$ is the integral closure of an ideal $\fa$ (see \cite{Laz04b}*{Definition 9.6.2} for details). For every $m\in\bN$ and $t\geq 0$, $\{W_{m,\ell}^t(\cG): \ell\in \bN\}$ is a graded linear series of $(mL+B)$, which we denote as $W_{m,\bullet}^t(\cG)$.

    When $\cF = \cF_v$ and $\cG = \cG_v$ for some valuation $v$, we often denote $W^t_m(\cF_v)$ as $W^t_\bullet(v)$ (and similarly for $W_{m,\ell}^t$). When $\cF$ and $\cG$ are clear from context, we will sometimes ignore writing $\cF$ and $\cG$ in notations such as $V_m^t(\cF)$ and $W_{m,\ell}^t(\cG)$.

\end{defn}

\subsection{Uniform Fujita approximation}
Let $(X,D)$ be a projective klt pair with $\dim X = n$. Let $L$ be a big line bundle on $X$. Let $v$ be a valuation on $X$ with $A_{X,D}(v)<\infty$. We will denote $A(v) = A_{X,D}(v)$ throughout this subsection. Let $B$ be the very ample line bundle from Lemma~\ref{lem:BJ_cor_5.10} and $a$ be a positive integer such that $aL\geq B$. Let $R$ and $M$ be as in Section~\ref{subsec:compatible_filtrations_on_aux_linear_series}.

\begin{lem}[cf. \cite{BJ20}*{Proposition 5.12}]\label{lem:BJ_prop_5.12}
Suppose $m\in {\bN}$ and $t\in \bQ_{>0}$ such that $mt\geq A(v)$. Then we have
\[
{\cJ}(X,D, m\cdot ||V_\bullet^t(v)||) \subset {\fb}(|W_m^{t'}(v)|),\]
where $t' = t- \frac{A(v)}{m}$.
\end{lem}
\begin{proof}
By the proof of \cite{BJ20}*{Proposition 5.11}, we have
\[
\mathcal{J}(X,D, m\cdot ||V_\bullet^t(v)||) \subseteq \mathfrak{a}_{mt-A(v)}(v).
\]
This implies that
\begin{align*}
H^0(X,(mL +B) \otimes \mathcal{J}(X,D, m\cdot ||V_\bullet^t(v)||))\subset H^0(X, (mL+B)\otimes \mathfrak{a}_{mt-A(v)}(v)) = W_m^{t'}(v).
\end{align*}
By Lemma~\ref{lem:BJ_cor_5.10}, the left hand side is globally generated. Then taking base ideals of the linear system $|mL+B|$ gives the desired inclusion
\[
\mathcal{J}(X,D, m\cdot ||V_\bullet^t(v)||) \subset \mathfrak{b}(|W_m^{t'}(v)|).
\]
\end{proof}
The following Proposition is the key to prove a uniform Fujita type approximation of $S^\tau(v)$ in Theorem~\ref{thm:uniform_fujita_approx_BJ_thm_5.3}.
\begin{prop}[cf. \cite{BJ20}*{Proposition 5.13}]\label{prop:BJ_prop_5.13}
Let $m,\ell\in \bN$ and $t>0$ such that $mt\geq A(v)$. Then we have
\[
\dim V_{m\ell}^t(v) \leq \dim W_{m,\ell}^{t'}(v)\leq \dim V_{(m+a)\ell}^{t''}(v),\]
where $t' = t - \frac{A(v)}{m}$ and $t'' = \frac{mt-A(v)}{m+a} = \frac{mt'}{m+a}$. As a result, we obtain
\[
m^n\vol(V_\bullet ^t(v)) \leq  \vol(W_{m,\bullet}^{t'}(v)) \leq (m+a)^n \vol(V^{t''}_\bullet(v)).
\]
\end{prop}
\begin{proof}
We may assume that $t\in\bQ$. Let $s\in H^0(X, B\otimes J_{X,D})$ be a nonzero section, where $J_{X,D}$ is the ideal in Lemma~\ref{lem:BJ_Lemma_5.8}. Then multiplication by $s^{\ell}$ gives an injection
\[
V_{m\ell}^t\hookrightarrow H^0(X, \ell(mL+B)\otimes (J_{X,D})^\ell \otimes \mathfrak{b}(|V_{m\ell}^t|)).
\]
By Lemma~\ref{lem:BJ_Lemma_5.8} and Lemma~\ref{lem:BJ_prop_5.12}, we have
\begin{align*}
&\, H^0(X, \ell(mL+B) \otimes (J_{X,D})^{\ell-1}\otimes  \mathfrak{b}(|V_{m\ell}^t|)) \\
\subseteq &\, H^0(X, \ell(mL+B) \otimes (J_{X,D})^{\ell-1}\otimes \mathcal{J}(X,D,m\ell\cdot ||V^t_\bullet||)) \\
\subseteq &\, H^0(X, \ell(mL+B) \otimes \mathcal{J}(X,D,m\cdot ||V^t_\bullet||)^\ell) \\
\subseteq &\, H^0(X, \ell(mL+B) \otimes \mathfrak{b}(|W_m^{t'}|)^\ell) \\
\subseteq &\, W_{m,\ell}^{t'}.
\end{align*}
Thus, we have $\dim V_{m\ell}^t\leq \dim W_{m,\ell}^{t'}$. On the other hand, since ${\fb}(|W_m^t|) \subset {\fa}_{mt}(v)$, we have $\overline{{\fb}(|W_m^t|)^\ell} \subset \overline{\mathfrak{a}_{m t}(v)^\ell} \subset \mathfrak{a}_{\ell mt}(v)$ (the last inclusion by \cite{Laz04b}*{Proposition 9.6.6(d)}). Thus,
\begin{align*}
W_{m,\ell}^{t'} &=  H^0(X, \ell(mL+B) \otimes \overline{\mathfrak{b}(|W^{t'}_m|)^\ell}) \\
&\subset H^0(X, \ell(mL+B)\otimes \mathfrak{a}_{\ell mt'}(v)) \\
&\subset H^0(X, \ell(m+a)L \otimes \mathfrak{a}_{\ell (mt-A(v))}(v)) \\
&= V_{(m+a)\ell}^{t''},
\end{align*}
as desired. 
\end{proof}
Next, we define modified $S_{m,k}$-invariants called the $\tS_m^\tau$-invariants, which provide better uniform approximations to the $S^\tau$-invariants.
\begin{defn}[cf. \cite{BJ20}*{Section 5.1}]\label{defn:S_tilde_m_invariants}
    Let $\cG$ be a filtration on $M$ compatible with $\cF$. For any $\tau\in (0,1)$ and $m\in\bN$, define
    \[
    Q_\tau(\cF) := \inf\{t\geq 0: \vol(V_{\bullet}^t(\cF))\leq \tau \vol(L)\},
    \]
    and
    \[
    Q_{m,\tau}(\cG) := \inf\{t\geq 0: \vol(W_{m,\bullet}^t(\cG))\leq m^n\tau \vol(L)\}.
    \]
    We also set $Q_1(\cF) = Q_{m,1}(\cG) = 0$. Then we define
    \[
    \tS_m^\tau(\cG) := \frac{1}{\tau m^n\vol(L)} \int_{Q_{m,\tau}(\cG)}^\infty \vol(W_{m,\bullet}^t(\cG)) dt.
    \]
\end{defn}

\begin{rem}
If we fix $\cF$, then $Q_\tau(\cF)$ is a continuous function of $\tau$ in the range $\tau\in (0,1)$ by the log-concavity of $\vol(V_\bullet^t(\cF))^{1/n}$. However, it is not continuous at $\tau = 1$, since it is possible that $\vol(V_\bullet^t(\cF)) = \vol(L)$ for some $t > 0$. A priori, one could have chosen $Q_1(\cF) = \sup \{t\leq 0: \vol(V_\bullet^t(\cF)) \geq \vol(L)\}$ to make $Q_\tau(\cF)$ a continuous function at $\tau = 1$, but this will not agree with the definition of the usual $S$-invariant.

Furthermore, unlike the case in \cite{BJ20}, it is not true in general that
    \[
    \lim_{m\to\infty} \tS_m^\tau(\cG) = S^\tau(\cF).
    \]
However, we will show that this equality holds when $\cF = \cF_v$ and $\cG=\cG_v$ for some valuation $v$ (Theorem~\ref{thm:uniform_fujita_approx_BJ_thm_5.3}), or when $\cF$ and $\cG$ are ``limits" of $\cF_{v}$ and $\cG_v$ for some sequences of valuations $v$ (see Lemma~\ref{lem:limit_of_S_tau_m_is_S} for details).
\end{rem}

Now, we can state the uniform Fujita type approximation result for the $S^\tau$-invariants.
\begin{thm}[cf. \cite{BJ20}*{Theorem 5.3}]\label{thm:uniform_fujita_approx_BJ_thm_5.3}
Let $\tau\in (0,1]$. Then there exists $C>0$ independent of $m$ or $v$ such that 
\[
|\tS_m^\tau(v) - S^\tau(v)| < \frac{CA(v)}{m}
\]
for every $m\in \bN$ and valuation $v$ with $A(v)<\infty$.
\end{thm}

To prove Theorem~\ref{thm:uniform_fujita_approx_BJ_thm_5.3}, we need the following lemma, which is a uniform Fujita type approximation result for $Q_\tau(v)$.
\begin{lem}\label{lem:Q_v(tau)_fujita_approx}
Let $\tau\in (0,1]$. Then there exists $C_1>0$ independent of $m$ or $v$ such that
\[
Q_\tau(v) - \frac{A(v)}{m}\leq Q_{m,\tau}(v) \leq Q_\tau(v) + \frac{C_1A(v)}{m}.
\]
\end{lem}
\begin{proof}
We may assume that $\tau\in (0,1)$. For the first inequality, we may assume that $Q_\tau(v) > \frac{A(v)}{m}$. Let $t\in \bR$ such that $\frac{A(v)}{m}\leq t < Q_\tau(v)$. Then by Proposition~\ref{prop:BJ_prop_5.13}, we have
\[
\vol(W_{m,\bullet}^{t-A(v)/m}) \geq m^n \vol(V_\bullet^t) > m^n \tau \vol(L).
\]
This implies that $t-\frac{A(v)}{m}\leq Q_{m,\tau}(v)$. Since we can choose $t$ arbitrarily close to $Q_\tau(v)$, we get
\[
Q_\tau(v) - \frac{A(v)}{m}\leq Q_{m,\tau}(v).
\]

For the second inequality, consider \[
\epsilon= \frac{a\tau^{1/n}}{(m+a)(1-\tau^{1/n})}.
\]
If $(1+\epsilon)Q_\tau(v) \geq T(v)$, then we have
\begin{align*}
    Q_\tau(v) &\geq \frac{T(v)}{1+\epsilon} = T(v) - \frac{\epsilon}{1+\epsilon}T(v) \\
    &= T(v) - \frac{a\tau^{1/n}}{m+a-m\tau^{1/n}}T(v) \\
    &\geq Q_{m,\tau}(v) - \frac{a\tau^{1/n}}{m(1-\tau^{1/n})}\frac{A(v)}{\alpha(L)}.
\end{align*}

If $(1+\epsilon)Q_\tau(v) < T(v)$, then by Lemma~\ref{lem:log_concavity_vol}, we have
\[
\frac{1}{1+\epsilon}\vol(V_\bullet^{(1+\epsilon)Q_\tau(v)})^{1/n}  + \frac{\epsilon}{1+\epsilon} \vol(V_\bullet^0)^{1/n} \leq \vol(V_\bullet^{Q_\tau(v)})^{1/n} \leq (\tau\vol(L))^{1/n}.
\]
After rearranging, we obtain
\[
\vol(V_\bullet^{(1+\epsilon)Q_\tau(v)})\leq  \left(\frac{m}{m+a}\right)^n\tau \vol(L).
\]
Let $t\in\bR$ such that
\[
t > \frac{m+a}{m}(1+\epsilon)Q_\tau(v) = \left(1+\frac{a}{m(1-\tau^{1/n})}\right)Q_\tau(v).
\]
By Proposition~\ref{prop:BJ_prop_5.13}, we have
\[
\vol(W_{m,\bullet}^t) \leq (m+a)^n\vol(V_\bullet^{mt/(m+a)}) < (m+a)^n \vol(V_\bullet^{(1+\epsilon)Q_\tau(v)})\leq  m^n\tau\vol(L),
\]
which implies that $t\geq Q_{m,\tau}(v)$. As a result, 
\begin{align*}
Q_{m,\tau}(v) &\leq \left(1+\frac{a}{m(1-\tau^{1/n})}\right) Q_\tau(v)\\
&\leq Q_\tau(v) + \frac{a}{m(1-\tau^{1/n})}T(v)\\
&\leq Q_\tau(v) + \frac{a}{m(1-\tau^{1/n})}\cdot \frac{A(v)}{\alpha(L)}.
\end{align*}
Thus, we may take $C_1 = \frac{a}{(1-\tau^{1/n})\alpha(L)}$.
\end{proof}

\begin{proof}[Proof of Theorem~\ref{thm:uniform_fujita_approx_BJ_thm_5.3}]
By Proposition~\ref{prop:BJ_prop_5.13}, for $t\geq \frac{A(v)}{m}$, we have
\[
m^n\vol(V_\bullet ^t) \leq \vol(W_{m,\bullet}^{t'}) \leq \left(m+a\right)^n \vol(V^{t''}_\bullet),
\]
where $t' = t-\frac{A(v)}{m}$ and $t'' = \frac{mt'}{m+a}$. Thus, by Lemma~\ref{lem:Q_v(tau)_fujita_approx}, we have
\begin{align*}
\Tilde{S}_m^\tau(v) &= \frac{1}{\tau m^n\vol(L)}\int_{Q_{m,\tau}(v)}^\infty\vol(W_{m,\bullet}^{t'}) dt' \\
& \geq \frac{1}{\tau\vol(L)}\int_{Q_{m,\tau}(v) + A(v)/m}^\infty \vol(V_\bullet^t) dt \\
&= S^\tau(v) - \frac{1}{\tau\vol(L)}\int_{Q_\tau(v)}^{Q_{m,\tau}(v) + A(v)/m} \vol(V_\bullet^t) dt \\
&\geq S^\tau(v) - \frac{1}{\tau}\left(Q_{m,\tau}(v) + \frac{A(v)}{m} - Q_\tau(v)\right) \\
&\geq S^\tau(v) - \frac{1+C_1}{\tau}\cdot \frac{A(v)}{m},
\end{align*}
where $C_1$ is the constant in Lemma~\ref{lem:Q_v(tau)_fujita_approx}.
On the other hand, when $m\geq 2(n+1)a$, we note that 
\[
\left(\frac{m+a}{m}\right)^{n+1} \leq 1 + \frac{2(n+1)a}{m}.
\]
Then by Lemma~\ref{lem:Q_v(tau)_fujita_approx}, we have
\begin{align*}
\Tilde{S}^\tau_m(v) &= \frac{1}{\tau m^n\vol(L)}\int_{Q_{m,\tau}(v)}^\infty\vol(W_{m,\bullet}^{t'}) dt' \\
&\leq  \frac{1}{\tau\vol(L)}\int_{\frac{m}{m+a}Q_{m,\tau}(v)}^\infty\left(\frac{m+a}{m}\right)^{n+1} \vol(V_\bullet^{t''}) dt'' \\
&= \left(\frac{m+a}{m}\right)^{n+1} \left(S^\tau(v) + \frac{1}{\tau \vol(L)}\int_{\frac{m}{m+a}Q_{m,\tau}(v)}^{Q_\tau(v)} \vol(V_\bullet^{t''}) dt''\right)\\
&\leq \left(\frac{m+a}{m}\right)^{n+1} \left(S^\tau(v) + \frac{1}{\tau}\left(Q_\tau(v) - \frac{m}{m+a}Q_{m,\tau}(v)\right) \right)\\
&\leq \left(\frac{m+a}{m}\right)^{n+1}\left(S^\tau(v) +\frac{1}{\tau}\left(Q_\tau(v) - \frac{m}{m+a}\left(Q_\tau(v)-\frac{A(v)}{m}\right)\right) \right)\\
&= \left(\frac{m+a}{m}\right)^{n+1}\left(S^\tau(v) +\frac{aQ_\tau(v)+ A(v)}{\tau(m+a)}\right)\\
&\leq \left(1+\frac{2(n+1)a}{m}\right)\left(S^\tau(v) + \frac{aQ_\tau(v)+A(v)}{\tau m}\right) \\
&\leq S^\tau(v) + \frac{2(n+1)a}{m}S^\tau(v) + \frac{2(aQ_\tau(v)+A(v))}{\tau m}  \\
&\leq S^\tau(v) + \frac{2(n+1)a}{m}\frac{A(v)}{\alpha(L)} + \frac{2(\frac{aA(v)}{\alpha(L)}+A(v))}{\tau m}\\
&=  S^\tau(v) + \frac{A(v)}{m}\left(\frac{2(n+1)a}{\alpha(L)} + \frac{2a}{\tau \alpha(L)} + \frac{2}{\tau}\right),
\end{align*}
as desired.
\end{proof}

\begin{thm}[cf. \cite{BJ20}*{Theorem 5.1}]\label{thm:uniform_fujita_approx_T}
For every valuation $v$ with $A(v)<\infty$ and every sufficiently large $m\in\bN$, we have
\[
0\leq T(v) -T_m(v) \leq \frac{2A(v)}{m}.
\]
\end{thm}
\begin{proof}
    The inequality $T_m(v)\leq T(v)$ follows from the fact that $T(v) = \sup_{m} T_m(v)$. We may assume that $m>a + \frac{A(v)}{T(v)}$. Pick $t\in (0,T(v))$ such that $m > a + \frac{A(v)}{t}$. Since $V_\bullet^t$ is nontrivial, so is $\cJ(X,D, (m-a)\cdot ||V_\bullet^t||)$. Then by Proposition~\ref{lem:BJ_prop_5.12}, we have
    \[
    0\neq \cJ(X, D, (m-a)\cdot ||V_\bullet^t||) \subseteq \fb(|W_{m-a}^{t'}|),
    \]
    where $t' = t - \frac{A(v)}{m-a} > 0$. This implies that $W_{m-a}^{t'}\neq 0$ and hence $V_{m}^{t'}\neq 0$. In particular,
    \[
    T_m(v) \geq t' = t-\frac{A(v)}{m-a}.
    \]
    Picking $t$ arbitrarily close to $T(v)$ gives
    \[
    T_m(v) \geq T(v) - \frac{A(v)}{m-a}\geq T(v) - \frac{2A(v)}{m}
    \]
    for $m\geq 2a$, as desired.
\end{proof}

\subsection{Parametrizing filtrations}\label{subsec:parametrize_filtrations}
Let $(X,D)$ be an $n$-dimensional projective klt pair and let $L$ be a big line bundle on $X$. Let $B$ be the very ample line bundle from Lemma~\ref{lem:BJ_cor_5.10} and $a$ be a positive integer such that $aL\geq B$. Let $R$ and $M$ be as in Section~\ref{subsec:compatible_filtrations_on_aux_linear_series}.

Following \cite{BJ20}*{Section 6}, we construct a space parameterizing a filtration $\cF$ on $R$ as well as a filtration $\cG$ on $M$ compatible with $\cF$. We first consider pairs of $\bN$-filtrations $({\cF},{\cG})$ on $R$ and $M$ (see Definition~\ref{defn:N_filtration}) such that for every $\lambda > m$,
\[
{\cF}^\lambda R_m = {\cG}^{\lambda + a}M_m = 0.
\]
Note that this condition is satisfied when $v$ is a valuation on $X$ with $T(v)\leq 1$, ${\cF}$ is the $\bN$-filtration induced by $\cF_v$, and $\cG$ is the $\bN$-filtration induced by ${\cG}_{v}$.

Such a pair of filtrations $(\mathcal{F},\mathcal{G})$ is determined by the flags
\begin{align}\label{eqn:flag_Rm}
\mathcal{F}^m R_m \subset \mathcal{F}^{m-1}R_m \subset \cdots \subset \mathcal{F}^0 R_m = R_m
\end{align}
and 
\begin{align}\label{eqn:flag_Mm}
    \mathcal{G}^{m+a} M_m \subset \mathcal{G}^{m+a-1}M_m \subset \cdots \subset \mathcal{G}^0 M_m = M_m,
\end{align}
such that 
\begin{align}\label{eqn:multiplicativity_Rm}
\mathcal{F}^{p_1}R_{m_1} \cdot \mathcal{F}^{p_2}R_{m_2} \subset \mathcal{F}^{p_1+p_2}R_{m_1+m_2}
\end{align}
and
\begin{align}\label{eqn:multiplicativity_Mm}
\mathcal{F}^{p_1}R_{m_1} \cdot \mathcal{G}^{p_2}M_{m_2} \subset \mathcal{G}^{p_1+p_2}M_{m_1+m_2}
\end{align}
hold for all $p_i, m_i\in \bN.$

Let $\Fl(R_m)$ and $\Fl(M_m)$ denote the flag varieties parameterizing flags of the form (\ref{eqn:flag_Rm}) and (\ref{eqn:flag_Mm}). Let 
\[
H_d := \prod_{m=0}^d \Fl(R_m)\times \Fl(M_m) 
\]
and let $\pi_{c,d}: H_c\to H_d$ denote the natural projection when $c\geq d$. Consider the subset $J_d\subset H_d$ parametrizing points $z$ such that the corresponding pair of filtrations $(\mathcal{F}_z,\mathcal{G}_z)$ satisfies (\ref{eqn:multiplicativity_Rm}) for all $0\leq p_i\leq m_i\leq d$, and (\ref{eqn:multiplicativity_Mm}) for all $0\leq p_1\leq m_1\leq d$ and $0\leq p_2\leq m_2+a\leq d+a$.

By the same proof of \cite{BJ20}*{Lemma 6.5}, we have the following lemma.
\begin{lem}[cf. \cite{BJ20}*{Lemma 6.5}]\label{lem:BJ_lem_6.5}
    The subset $J_d\subset H_d$ is closed.
\end{lem}
Note that $\pi_{c,d}:H_c\to H_d$ restricts to $\pi_{c,d}: J_c\to J_d$ when $c\geq d$. Let $J:= \underset{\longleftarrow}{\lim} J_d(\bC)$, where the inverse limit is taken with respect to $\pi_{c,d}$'s. Let $\pi_d$ denote the projection $J\to J_d(\bC)$. Then every element of $J$ corresponds to a pair of $\bN$-filtrations $(\mathcal{F},\mathcal{G})$ of $R$ and $M$ satisfying \eqref{eqn:flag_Rm}, \eqref{eqn:flag_Mm}, \eqref{eqn:multiplicativity_Rm}, and \eqref{eqn:multiplicativity_Mm}.

\subsection{Proof of Theorem~\ref{thm:qm_val_computing_delta_tau}}\label{subsec:proof_of_theorem_6.1}
This entire section is devoted to the proof of Theorem~\ref{thm:qm_val_computing_delta_tau}. We will follow the set-up from the previous section (Section~\ref{subsec:parametrize_filtrations}). We will denote $A(v) = A_{X,D}(v)$ for any valuation $v$ on $X$ and $\lct(\fb) = \lct(X,D;\fb)$ for any ideal or any graded ideal sequence $\fb$ throughout this subsection.

We first assume that $\tau\in (0,1]$. Let $\{v_i: i\geq 1\}$ be a sequence of divisorial valuations such that $\delta^\tau(v_i)\searrow \delta^\tau(L)$. We may normalize $v_i$ such that $T(v_i) = 1$ for all $i$. By Lemma~\ref{lem:ineq_S_and_T_invariants}, after passing to a subsequence, we may assume that both $\lim_{i\to\infty} S^\tau(v_i) = S$ and $\lim_{i\to\infty} A(v_i) = A$ exist and are finite. Let $z_i\in J$ be the point corresponding to the compatible pair of $\bN$-filtrations $(\mathcal{F}_i,\mathcal{G}_i):=(\mathcal{F}_{v_i,\bN},\mathcal{G}_{v_i,\bN})$. Then $T(\mathcal{F}_{i}) = 1$ by Lemma~\ref{lem:BJ_prop_2.11}. \\

\noindent\textbf{Step 1.}(cf. \cite{BJ20}*{Proposition 6.7, Claim 1}) We can inductively choose an infinite sequence $\bN\supset I_0\supset I_1 \supset \cdots$ such that for each $m\in \bN$, $I_m$ is an finite set and 
\[
Z_m := \overline{\{\pi_m(z_i): i\in I_m\}} \subseteq J_m
\]
satisfies the following property: for any closed subset $V\subsetneq Z_m$, there are only finitely many $i\in I_m$ such that $\pi_m(z_i)\in V$. Note that $Z_m$ is irreducible and $\pi_{c,d}(Z_c)\subseteq Z_d$ for $c\geq d$. \\

\noindent\textbf{Step 2.}(cf. \cite{BJ20}*{Proposition 6.7, Claim 2}) For each $m\in \bN$, there exists a nonempty open subset $U_m$ of $Z_m$ such that for any $z\in U_m$, the corresponding pair of $\bN$-filtrations $(\cF_z, \cG_z)$ satisfies the following properties:
\begin{itemize}
    \item[(a)] for any $1\leq p\leq m$, $\dim \cF_z^pR_m$ is a constant for $z\in U_m$;
    \item[(b)] for any $1\leq p\leq m+a$, $\dim \cG_z^pM_m$ is a constant for $z\in U_m$;
    \item[(c)] for any $t\geq 0$, $\vol(W_{m,\bullet}^t(\cG_z))$ is a constant $w_{m}^t$ for $z\in U_m$;
    \item[(d)] $Q_{m,\tau}(\cG_z)$ is a constant for $z\in U_m$;
    \item[(e)] $\tS_m^{\tau}(\cG_z)$ is a constant for $z\in U_m$;
    \item[(f)] for any $1\leq p\leq m$, the value
    \[
    p\cdot \lct(\fb_{p,p}(\cF_z) + \fb_{p,p+1}(\cF_z) + \cdots + \fb_{p,m}(\cF_z)) 
    \]
    is a constant $a_{p,m}$ a constant for $z\in U_m$, where $\fb_{p,m}(\cF_z) = \fb(|\cF_z^pR_m|)$;
    \item[(g)] for any $t\geq 0$, we have
    \[
    \liminf_{\ell\to\infty} \inf_{z\in U_m} \frac{\dim W_{m,\ell}^t(\cG_z)}{\ell^n/n!} = \limsup_{\ell\to\infty} \sup_{z\in U_m} \frac{\dim W_{m,\ell}^t(\cG_z)}{\ell^n/n!} = w_m^t
    \]
    \item[(h)] $\pi_{m,m-1}(U_{m})\subseteq U_{m-1}.$
\end{itemize}
Indeed, (a) and (b) follow from the fact that $Z_m$ is irreducible for all $m$. (c) follows from \cite{BJ20}*{Proposition 6.3} and the fact that $W_{m,\bullet}^t(\cG_z)$ is constant on $t\in (\frac{k}{m},\frac{k+1}{m}]$ for every $k\in \bN$ and $\bN$-filtration $\cG_z$. (d) follows from (c) and (e) follows from (c) and (d). (f) follows from the fact that the log canonical threshold of a family of ideals is constant on a nonempty open subset (see \cite{BJ20}*{Proposition 6.1} for details). We will deal with property (g) in Lemma~\ref{lem:finiteness_hilb_functions}. Finally, suppose $U\subseteq Z_m$ is an open subset satisfying properties (a)-(g), then we can take $U_m = U\cap \pi_{m,m-1}^{-1}(U_{m-1})$ to get (h).

\begin{lem}\label{lem:finiteness_hilb_functions}
There exists an open subset $U_m\subseteq Z_m$ such that property (g) holds.
\end{lem}
\begin{proof}
We first note that there exists an open subset $U\subseteq Z_m$ and an ideal $\cB\subseteq \cO_{X\times U}$ flat over $U$ such that
\[
\cB \cdot\cO_{X\times \{z\}} = \fb(|W_m^t(\cG_z)|)
\]
for every closed point $z\in U$. Consider a resolution $\mu: Y\to X\times U$ of $\cB$ such that $\cB\cdot \cO_Y = \cO_Y(-F)$ for some effective Cartier divisor $F$ on $Y$. After shrinking $U$ to a smaller open subset, we may assume that $(Y,F)$ is a log smooth family over $U$.

Let $p_1, p_2$ denote the projections from $X\times U$ to $X$ and $U$, respectively. Let
\[
\cW_{m,\ell}^t = p_{2*}(p_1^*\ell(mL+B)\otimes_{\cO_{X\times U}} \overline{\cB^{\ell}}) \subseteq H^0(X,\ell(mL+B))\otimes_{\bC} \cO_{U},
\]
so that for every closed point $z\in U$, we have
\[
\cW_{m,\ell}^t \otimes \kappa(z) \cong W_{m,\ell}^t(\cG_z).
\]
 Let $\pi_i = p_i\circ \mu$ for $i=1,2$. Define
\[
H := \pi_1^*(mL+B) \otimes_{\cO_Y} \cO_Y(-F),
\]
which is big and base-point-free by our construction. For a closed point $z\in U$, let $\mu_z: (Y_z, F_z)\to X$ be the base change of $\mu$ and $H_z = H|_{Y_z}$. Then
\[
\cW_{m,\ell}^t \cong \pi_{2*}\cO_{Y}(\ell H) \quad \text{and} \quad W_{m,\ell}^t(\cG_z) \cong H^0(Y_z, \ell H_z).
\]
Applying \cite{BL22}*{Theorem 5.6} on a relative ample model of $kH$ over $U$, where $k$ is a positive integer such that $R^i\pi_{2*}\cO_Y(mkH) = 0$ for every $i>0$ and $m>0$, we obtain that the set of Hilbert polynomials of $kH_z$ with $z\in U_m$ is bounded. Furthermore, by property (c), all the leading coefficients of these Hilbert polynomials are equal to $w_m^tk^n/n!$. From this, it is clear that
\[
\liminf_{\ell\to\infty} \inf_{z\in U_m} \frac{\dim H^0(Y_z, \ell H_z)}{\ell^n/n!} = \limsup_{\ell\to\infty} \sup_{z\in U_m} \frac{\dim H^0(Y_z, \ell H_z)}{\ell^n/n!} = w^t_m,
\]
which implies (g).
\end{proof}

Let $I_m^\circ = \{z\in I_m: \pi_m(z)\in U_m\}$, which is infinite by our choice of $Z_m$ in step 1. By \cite{BJ20}*{Lemma 6.6}, there exists $z\in J$ such that $\pi_m(z)\in U_m$ for all $m$. Let $(\mathcal{F},\mathcal{G})$ be the pair of filtrations on $R$ and $M$ corresponding to $z$. \\

\noindent\textbf{Step 3.} For each $p\in \bN$, define 
\[
\fb_p(\cF) = \sum_{m\geq p} \fb_{p,m}(\cF) := \sum_{m\geq p} \fb (|\cF^p R_m|) \subseteq \cO_X,
\]
which is a finite sum by the Noetherianity of $X$. Since
\[
\fb_{p_1,m_1}(\cF) \cdot \fb_{p_2,m_2}(\cF) \subseteq \fb_{p_1+p_2,m_1+m_2}(\cF),
\]
$\fb_\bullet(\cF)$ is a graded ideal sequence. We define $\fb_\bullet(\cF_i)$ similarly. 
\begin{lem}[cf. \cite{BJ20}*{Lemma 6.9(1)}]\label{lem:BJ_lem_6.9(1)}
    We have $\lct(\fb_\bullet(\cF))\leq A$.
\end{lem}
\begin{proof}
For every $p\in\bN$ and $i$, we have $\fb_p(\cF_i) = \fb_p(\cF_{v_i}) \subseteq \fa_{p}(v_i)$. As a result, we have
\[
\lct(\fb_\bullet (\cF_i)) \leq \lct(\fa_\bullet (v_i)) \leq A(v_i),
\]
where the last inequality follows from \cite{BJ20}*{Lemma 1.1}. Furthermore, we have
\begin{align*}
    a_{p,m} \leq & \limsup_{i\to\infty}  p\cdot \lct(\fb_{p,p}(\cF_i) + \cdots + \fb_{p,m}(\cF_i)) \\
    \leq & \limsup_{i\to\infty}  p\cdot \lct(\fb_p(\cF_i))\leq\limsup_{i\to\infty} \lct(\fb_\bullet(\cF_i)) \\
    \leq &\limsup_{i\to\infty} A(v_i) = A,
\end{align*}
and hence
\begin{align*}
    \lct(\fb_\bullet(\cF)) &= \sup_{p} p \cdot \lct(\fb_p(\cF))\\ 
    &= \sup_{p}\sup_{m\geq p} p\cdot \lct(\fb_{p,p}(\cF) + \cdots + \fb_{p,m}(\cF)) \\
    &= \sup_p\sup_{m\geq p} a_{p,m} \leq A.
\end{align*}
\end{proof}

\noindent\textbf{Step 4.} We show that $S^\tau(\cF) = S$. 
\begin{lem}\label{lem:limit_of_S_tau_m_is_S}
    We have
    \[
    \lim_{m\to\infty} \tS_m^\tau(\cG) = S^\tau(\cF)= S.
    \]
\end{lem}
\begin{proof}
Similar to Lemma~\ref{lem:BJ_prop_2.11}, for every $t>0$, $m,\ell\in \bN$, and $\tau\in (0,1]$, we have the following relations between $\cG$ and $\cG_\bN$: 
\begin{align}\label{eqn:BJ_prop_2.11_for_dim_W}
    \dim W_{m,\ell}^t(\cG_\bN) \leq \dim W_{m,\ell}^t(\cG)\leq \dim W_{m,\ell}^{t-1/m}(\cG_\bN),
\end{align}
\begin{align}\label{eqn:BJ_prop_2.11_for_Q_cG}
    Q_{m,\tau}(\cG_i)\leq Q_{m,\tau}(v_i)\leq Q_{m,\tau}(\cG_i)+\frac{1}{m}.
\end{align}

After passing to a subsequence, we may assume that $A(v_i)\leq A + 1$ for all $i$. By Proposition~\ref{prop:BJ_prop_5.13}, Lemma~\ref{lem:BJ_prop_2.11}, and \eqref{eqn:BJ_prop_2.11_for_dim_W}, for every $t\geq \frac{A+1}{m}\geq \frac{A(v_i)}{m}$, we have
    \[
    \dim V_{m\ell}^t(\cF_i) \leq \dim V_{m\ell}^t(v_i) \leq \dim W_{m,\ell}^{t-(A+1)/m}(v_i) \leq\dim W_{m,\ell}^{t-(A+2)/m}(\cG_i).
    \]
    Let $i\in I_{m\ell}^\circ$, we get by property (a) that
    \[
    \dim V_{m\ell}^t(\cF) \leq \dim W_{m,\ell}^{t-(A+2)/m} (\cG_i) \leq \sup_{z\in U_m}\dim W_{m,\ell}^{t-(A+2)/m} (\cG_z).
    \]
    Taking $\ell\to\infty$ and using properties (c) and (g) (Lemma~\ref{lem:finiteness_hilb_functions}), we obtain
    \begin{align}\label{eqn:fujita_approx_for_limiting_filtration_1}
    m^n\vol(V_\bullet^t(\cF)) \leq \limsup_{\ell\to\infty} \sup_{z\in U_m}\frac{\dim W_{m,\ell}^{t-(A+2)/m} (\cG_z)}{\ell^n/n!} = w_m^{t-(A+2)/m}= \vol(W_{m,\bullet}^{t-(A+2)/m}(\cG)).
    \end{align}
    Similarly,
    \[
    \dim W_{m,\ell}^t(\cG_i) \leq \dim W_{m,\ell}^t(v_i)\leq \dim V_{(m+a)\ell}^{mt/(m+a)}(v_i)\leq \dim V_{(m+a)\ell}^{mt/(m+a)-1/(m+a)\ell}(\cF_i).
    \]
    Let $i\in I_{(m+a)\ell}$, we get
    \[
    \inf_{z\in U_m} W_{m,\ell}^t(\cG_z) \leq \dim V_{(m+a)\ell}^{mt/(m+a)-1/(m+a)\ell}(\cF).
    \]
    Taking $\ell\to\infty$ and applying properties (c), (g) and the continuity of $\vol(V_\bullet^t(\cF))$ (Lemma~\ref{lem:log_concavity_vol}), we obtain
    \begin{align}\label{eqn:fujita_approx_for_limiting_filtration_2}
    \vol(W_{m,\bullet}^t(\cG)) = \liminf_{\ell\to\infty} \inf_{z\in U_m} \frac{\dim W_{m,\ell}^t(\cG_z)}{\ell^n/n!} \leq (m+a)^n \vol(V_\bullet^{mt/(m+a)}(\cF)).
    \end{align}

    By the same arguments in the proof of Theorem~\ref{thm:uniform_fujita_approx_BJ_thm_5.3} and Lemma~\ref{lem:Q_v(tau)_fujita_approx}, where we replace Proposition~\ref{prop:BJ_prop_5.13} by inequalities \eqref{eqn:fujita_approx_for_limiting_filtration_1} and \eqref{eqn:fujita_approx_for_limiting_filtration_2}, we may obtain
    \[
    |\tS_m^\tau(\cG) - S^\tau(\cF)|< \frac{C_0}{m}
    \]
    for some constant $C_0>0$ not depending on $m$. This implies
    \[
    \lim_{m\to\infty} \tS_m^\tau(\cG) = S^\tau(\cF).
    \]
    For the second equality, we note that \eqref{eqn:BJ_prop_2.11_for_dim_W} and \eqref{eqn:BJ_prop_2.11_for_Q_cG} imply the inequality
    \[
    \tS_m^\tau(\cG_i) \leq \tS_m^\tau(v_i) \leq \tS_m^\tau(\cG_i) + \frac{1}{m}.
    \]
    Then, for $i\in I_m^\circ$, we apply property (e) to get
    \begin{align*}
    |\tS_m^\tau(\cG) - S| &= |\tS_m^\tau(\cG_i) - S| \leq |\tS_m^\tau(v_i) - S| + \frac{1}{m} \\
    &\leq |\tS_m^\tau(v_i) - S^\tau(v_i)| + |S^\tau(v_i) - S| + \frac{1}{m} \\
    &\leq |S^\tau(v_i) - S| + \frac{CA(v_i)+1}{m},
    \end{align*}
    where $C>0$ is the constant from Theorem~\ref{thm:uniform_fujita_approx_BJ_thm_5.3}. Let $i\to\infty$ and $m\to\infty$, we get
    \[
    \limsup_{m\to\infty}|\tS_m^\tau(\cG) - S|\leq \limsup_{m\to\infty}\limsup_{i\to\infty} \left(|S^\tau(v_i) - S| + \frac{CA(v_i)+1}{m}\right) = 0,
    \]
    so $\lim_{m\to\infty} \tS_m^\tau(\cG) = S$, as desired.
\end{proof}

\noindent\textbf{Step 5.}
By \cite{Xu20}*{Theorem 1.1}, there exists a quasi-monomial valuation $v^*$ computing $\lct(\fb_\bullet(\cF))$ ($\fb_\bullet(\cF)$ is defined in step 3). After rescaling, we may assume that $v^*(\fb_\bullet(\cF)) = 1$. Then by Lemma~\ref{lem:BJ_lem_6.9(1)}, we have
\[
A(v^*) = \frac{A(v^*)}{v^*(\fb_\bullet(\cF))} = \lct(\fb_\bullet(\cF)) \leq A.
\]
Furthermore, since $v^*(\fb_p(\cF)) \geq p$, we have $v^*(\fb_{p,m}(\cF)) \geq p$ for all $m\geq p$. Then
\[
\cF^pR_m \subseteq H^0(X, mL\otimes \fb_{m,p}(\cF)) \subseteq H^0(X, mL\otimes \fa_p(v^*)) = \cF^p_{v^*}R_m.
\]
This implies that $S^\tau(\cF)\leq S^\tau(v^*)$, so $S^\tau(v^*) \geq S$ by Lemma~\ref{lem:limit_of_S_tau_m_is_S}. As a result,
\[
\delta^\tau(v^*) = \frac{A(v^*)}{S^\tau(v^*)} \leq \frac{A}{S} = \delta^\tau(L),
\]
so $v^*$ computes $\delta^\tau(L)$. \\

This concludes the proof of Theorem~\ref{thm:qm_val_computing_delta_tau} when $\tau\in (0,1]$. Finally, we deal with the case $\tau = 0$.
\begin{prop}
    There exists a quasi-monomial valuation computing $\delta^0(L) = \alpha(L)$.
\end{prop}
\begin{proof}
    Let $\{v_i:i\geq 1\}$ be a sequence of divisorial valuations on $X$ such that $\alpha(v_i)\searrow \alpha(L)$. We may normalize $v_i$ such that $T(v_i) = 1$ for all $i$. Then $\lim_{i\to\infty}A(v_i) = \alpha(L)$. 
    
    We mimic steps 1-3, except that in step 2 we only require $U_m\subseteq Z_m$ to be a nonempty open subset satisfying (a), (b), (f), and (h). Lemma~\ref{lem:BJ_lem_6.9(1)} still applies in this case, so $\lct(\fb_\bullet(\cF)) \leq \alpha(L)$.  We claim that $T(\cF) = 1$.

    By Lemma~\ref{lem:BJ_prop_2.11}, we have $|T_m(\cF_i) - T_m(v_i)| \leq \frac{1}{m}$. For $i\in I_m^\circ$, we apply property (a) and Theorem~\ref{thm:uniform_fujita_approx_T} to get
\begin{align*}
    |T_m(\cF) - 1| &= |T_m(\cF_i) - 1| \leq |T_m(v_i) - T(v_i)| +\frac{1}{m}  
    \leq \frac{1+CA(v_i)}{m}.
\end{align*}
Taking $i\to\infty$ and $m\to\infty$, we have
\[
|T(\cF) - 1| \leq \limsup_{m\to\infty} |T_m(\cF) - 1| \leq \limsup_{m\to\infty}\limsup_{i\to\infty}\frac{1+CA(v_i)}{m} = 0.
\]
This implies $T(\cF) = 1$.

Finally, by \cite{Xu20}*{Theorem 1.1}, there exists a quasi-monomial valuation $v^*$ computing $\lct(\fb_\bullet(\cF))$. We may assume that $v^*(\fb_\bullet(\cF)) = 1$. Applying the arguments in step 5, we get
\[
A(v^*)\leq \lct(\fb_\bullet(\cF)) \leq \alpha(L) \quad\text{and}\quad T(v^*) \geq T(\cF) = 1.
\]
This implies
\[
\alpha(v^*) = \frac{A(v^*)}{T(v^*)} \leq \alpha(L),
\]
so $v^*$ computes $\alpha(L)$.
\end{proof}
This concludes the proof of Theorem~\ref{thm:qm_val_computing_delta_tau}. \qed

\section{Asymptotics of stability thresholds}\label{sec:asymptotics_of_stability_thresholds}
Let $(X,D)$ be a projective klt pair and let $L$ be a big line bundle on $X$. For each $m\in\bN$, denote $N_m = \dim H^0(X,mL)$. Let $\tau\in [0,1]$ and $\{k_m: 1\leq k_m\leq N_m, m\in \bN\}$ be a sequence of positive integers such that $\lim_{m\to\infty} \frac{k_m}{N_m} = \tau$. 

The goal of this section is to prove the following theorem.
\begin{thm}[cf. \cite{JRT25a}*{Theorem 5.11}]\label{thm:JRT_Theorem_1.7_without_assumptions}
    There exists $C>0$ independent of $m$ such that the following inequalities hold. If $\tau > 0$, we have
    \[
    \left(1-\frac{C}{m}\right)\min\left\{\frac{k_m}{\tau N_m}, 1\right\}\delta^\tau(L) \leq \delta_{m, k_m}(L) \leq \left(1+\frac{C}{m}\right)\max\left\{\frac{k_m}{\tau N_m}, 1\right\}\delta^\tau(L).
    \]
    If $\tau = 0$, we have
    \[
    \left(1-\frac{C}{m}\right)\delta^0(L) \leq \delta_{m, k_m}(L) \leq \left(1+C\max\left\{\frac{k_m}{N_m},\frac{1}{m}\right\}^{1/n}\right)\delta^0(L).
    \]
\end{thm}

\subsection{Newton--Okounkov bodies and approximations of $S$-invariants}\label{subsec:okunkov_bodies_and_S_invariants} 
Let $\mu: (Y, E= E_1+\cdots+E_r)\to X$ be a log resolution and let $Z$ be an irreducible component of $\bigcap_{i=1}^r E_i$. Let $v\in\QM_{\eta}(Y, E)$ with weight vector $\alpha\in \bR_{>0}^r$ and assume $\alpha_1,\ldots,\alpha_r$ are $\bQ$-linearly independent. Let $Z_\bullet: Z = Z_0\supset Z_1 \supset\cdots \supset Z_{n-r} = \{z\}$ be an admissible flag on $Z$. Let $\nu_\bullet$ be the valuation constructed using $v$ and $Z_\bullet$ (see Definition~\ref{defn:nu_bullet}) and let $\Delta$ be the Newton--Okounkov body of $\mu^*L$ with respect to $\nu_\bullet$ (see Definition~\ref{defn:okounkov_bodies_wrt_nu_bullet}). For each $t\in [0, T(v)]$, let $\Delta^t$ be the Newton-Okunkov body of the graded linear series $V_\bullet^t(v)$ with respect to $\nu_\bullet$ (see Section~\ref{subsubsec:okounkov_body_for_graded_linear_series}). Let $G_v: \bR^n\to \bR$ be the linear function $G_v(x) = \sum_{i=1}^r \alpha_i x_i$. Then $\Delta^t = \Delta \cap G_v^{-1}([t, T(v)])$ by Theorem~\ref{mainthm:okounkov_bodies_wrt_qm_val}(a).

Recall $Q_\tau(v) = \inf\{t\geq 0: \vol(V_\bullet^t(v))\leq \tau\vol(L)\}$ for $\tau\in [0,1]$. We have the following formula for $S^\tau(v)$, which is crucial to prove the approximation result for $S^\tau(v)$ in Proposition~\ref{prop:JRT25a_thm_5.1} (although this formula will not be explicitly mentioned in the proof).
\begin{lem}[cf. \cite{JRT25a}*{Remark 4.17}]
    For $\tau \in (0,1]$, we have
    \[
    S^\tau(v) = \frac{n!}{\tau \vol(L)}\int_{\Delta^{Q_\tau(v)}(v)}G_v(x) dx.
    \]
\end{lem}
\begin{proof}
    This follows from the proof of \cite{JRT25a}*{Proposition 4.12}, except that in our case, the concave transform function is $G_v$ and not the projection onto the first coordinate.
\end{proof}

\begin{prop}[cf. \cite{JRT25a}*{Corollary 4.25 and Theorem 5.1}]\label{prop:JRT25a_thm_5.1}
    There exists $C>0$ independent of $m$ such that 
    \[
    S_{m, k_m}(v) \leq \left(1+\frac{C}{m}\right)\max\left\{\frac{\tau N_m}{k_m},1\right\} S^\tau(v)
    \]
    and
    \[
    S_{m, k_m}(v) \geq \begin{cases}
        \left(1 - C\max\left\{\frac{k_m}{N_m},\frac{1}{m}\right\}^{1/n}\right)S^\tau(v), &\text{ if }\tau = 0, \\
        \left(1-\frac{C}{m}\right)\max\left\{\frac{\tau N_m}{k_m},1\right\} S^\tau(v), &\text{ if }\tau > 0.
    \end{cases}
    \]
\end{prop}
\begin{proof}
    The first statement follows from \cite{JRT25a}*{Corollary 4.25} and the second one follows from \cite{JRT25a}*{Theorem 5.1}, except that some small modifications are needed. In \cite{JRT25a}, $v$ is a divisorial valuation $c\cdot \ord_F$ and $F$ is a part of the admissible flag used to construct the Newton--Okounkov body of $L$. Instead, we use the Newton-Okunkov bodies in Theorem~\ref{mainthm:okounkov_bodies_wrt_qm_val}. The concave transform function in \cite{JRT25a} is the projection onto the first coordinate, which will be replaced by $G_v$ (which is still a linear function) in our case. After making these modifications, all the arguments in \cite{JRT25a} will go through.
\end{proof}

\subsection{Proof of Theorem~\ref{thm:JRT_Theorem_1.7_without_assumptions} and Theorem~\ref{mainthm:JRT25a_without_assumptions}}
\begin{proof}[Proof of Theorem~\ref{thm:JRT_Theorem_1.7_without_assumptions}]
By Theorem~\ref{thm:qm_val_computing_delta_tau}, there exists a quasi-monomial valuation $v$ computing $\delta^\tau(L)$. Similar to the setup at the beginning of Section~\ref{subsec:okunkov_bodies_and_S_invariants}), we build a Newton--Okounkov body of $L$ with respect to $v$. Then we apply Proposition~\ref{prop:JRT25a_thm_5.1} for $v$ to obtain Theorem~\ref{thm:JRT_Theorem_1.7_without_assumptions}.  
\end{proof}

\begin{proof}[Proof of Theorem~~\ref{mainthm:JRT25a_without_assumptions}]
Let $\tau = 0$. By Theorem~\ref{thm:qm_val_computing_delta_tau}, there exists a valuation $v$ computing $\alpha(L) = \delta^0(L)$. By Theorem~\ref{thm:uniform_fujita_approx_T}, for any sufficiently large $m\in\bN$, we have
\[
0\leq T(v) - T_m(v) \leq \frac{2A(v)}{m}.
\]
As a result, $\alpha(v) \leq \alpha_m(v)$ and
\begin{align*}
    \alpha_m(v) -\alpha(v) &= \frac{A(v)}{T_m(v)} - \frac{A(v)}{T(v)} = \frac{A(v)(T(v)-T_m(v))}{T(v)T_m(v)} \\
    &\leq \frac{2A(v)^2}{mT(v)^2} = \frac{2\alpha(L)^2}{m}.
\end{align*}

Let $\tau = 1$ and $k_m = N_m = \dim H^0(X,mL)$ for all $m$. Then by Theorem~\ref{thm:JRT_Theorem_1.7_without_assumptions} and note that $\delta^1(L) = \delta(L)$, there exists $C_1> 0$ independent of $m$ such that
\[
\left(1-\frac{C_1}{m}\right)\delta(L) \leq \delta_{m}(L) \leq \left(1+\frac{C_1}{m}\right)\delta(L).
\]
This gives
\[
|\delta_m(L) - \delta(L)| \leq \frac{C_1\delta(L)}{m}.
\]
Thus, we may take $C = \max\{2\alpha(L)^2,C_1\delta(L)\}$ and the proof is complete.
\end{proof}

\appendix

\section{Variation of Newton--Okounkov bodies}\label{sec:variation_okounkov_bodies}
Let $(X, D = D_1+\ldots + D_r)$ be an $n$-dimensional log smooth pair. Let $Z$ be an irreducible component of $\bigcap_{i=1}^r D_i$ with generic point $\eta$. Let $Z = Z_0 \supset Z_1 \supset \cdots \supset Z_{n-r} = \{z\}$ be an admissible flag on $Z$. Let $L$ be a big line bundle on $X$. Let $v_\alpha\in {\QM}_{\eta}(X,D)$ be the quasi-monomial valuation with weight vector $\alpha\in \bR_{>0}^r$. When $\alpha_1,\ldots,\alpha_r$ are $\bQ$-linearly independent, let $\nu_\bullet^{\alpha}$ be the valuation in Definition~\ref{defn:nu_bullet}, defined using the valuation $v_\alpha$ and the admissible flag $Z_\bullet$. Denote $\Gamma_m(\alpha) = \Gamma_m(\nu_\bullet^{\alpha}, L)$, $\Gamma(\alpha) = \Gamma(\nu_\bullet^{\alpha}, L)$ and $\Delta(\alpha) = \Delta(\nu_\bullet^{\alpha}, L)$ (see Definition~\ref{defn:okounkov_bodies_wrt_nu_bullet}).

The goal of this section is to study how these Newton--Okounkov bodies $\Delta(\alpha)$ vary with respect to their weight vectors $\alpha$. As we will see in Section~\ref{subsec:explicit_variations_of_okunkov_bodies_on_P2}, this question is related to finding special $\bG_m$-equivariant degenerations of Fano varieties.

Note that $\Delta(\lambda\alpha) = \Delta(\alpha)$ for any constant $\lambda > 0$, so we may assume that $|\alpha| = 1$. Denote
\[
P_r = \{\alpha\in \bR_{>0}^r: |\alpha| = 1, \alpha_1,\ldots,\alpha_r \text{ are $\bQ$-linearly independent}\}.
\]
We prove the following lemmas, which imply that locally near any $\alpha\in P_r$, these $\Delta(\alpha)$ vary in a continuous fashion.

\begin{lem}\label{lem:Gamma_m_locally_constant}
Let $\alpha\in {P_r}$ and $m\in \bN$. Then there exists $\epsilon > 0$ such that for any $\alpha'\in {P_r}$ with $|\alpha'-\alpha| < \epsilon$, we have
\[
\Gamma_m(\alpha) = \Gamma_m(\alpha')
\]
\end{lem}
\begin{proof}
Let $s_1,\ldots, s_{N_m}$ be a basis of $H^0(X,mL)$ such that $v_\alpha(s_i)\neq v_\alpha(s_j)$ for any pairs $i\neq j$. Suppose locally $s_i$ is given by a function $f_i\in \cO_{X,\eta}$. Then, there exists an open neighborhood $U\subseteq P_r$ of $\alpha$ such that the lowest-order term of $f_i$ with respect to $v_\alpha$ and with respect to $v_{\alpha'}$ are the same for every $i$ and every $\alpha'\in U$. This guarantees that $\Gamma_m(\alpha) =\Gamma_m(\alpha')$ for every $\alpha'\in U$.
\end{proof}

\begin{lem}
Let $\alpha\in {P_r}$ and $\epsilon > 0$. Then there exists $\epsilon' > 0$ such that for any $\alpha'\in {P_r}$ with $|\alpha'-\alpha| < \epsilon'$, we have
\[
\vol_{\bR^n}(\Delta(\alpha) \cap \Delta(\alpha')) \geq \vol_{\bR^n}(\Delta(\alpha)) - \epsilon.
\]
If we further assume that $\Gamma(\alpha)$ is finitely generated, then $\Delta(\alpha') = \Delta(\alpha)$.
\end{lem}
\begin{proof}
Let $m$ be a positive integer such that the volume of the closed convex hull of $\frac{1}{m}\Gamma_m(\alpha)$ is at least $\vol(\Delta(\alpha)) - \epsilon$. Such an $m$ always exists because $\Delta(\alpha)$ is the closed convex hull of $\bigcup_{m\geq 0} \frac{1}{m}\Gamma_m(\alpha)$. By Lemma~\ref{lem:Gamma_m_locally_constant}, for any $\alpha'\in P_r$ sufficiently close to $\alpha$, we have $\Gamma_m(\alpha') = \Gamma_m(\alpha)$. Then
\[
\vol_{\bR^n}(\Delta(\alpha) \cap \Delta(\alpha')) \geq \vol_{\bR^n}\left(\text{the closed convex hull of $\frac{1}{m}\Gamma_m(\alpha)$}\right) \geq \vol_{\bR^n}(\Delta(\alpha)) - \epsilon.
\]

If $\Gamma(\alpha)$ is finitely generated, then $\Delta(\alpha)$ is equal to the closed convex hull of $\frac{1}{m}\Gamma_m(\alpha)$ for some $m\in \bN$ and hence $\Delta(\alpha') = \Delta(\alpha)$.
\end{proof}

\subsection{Limits of convex bodies}
In this subsection, we consider two types of limits of a sequence of convex bodies.
\begin{defn}\label{defn:limits_of_convex_bodies}
Let $\{\Delta_i\subseteq \bR^n: i\geq 1\}$ be an infinite sequence of closed convex sets. The \emph{pointwise limit of $\{\Delta_i:i\geq 1\}$} is the set
\[
\Delta^p := \{p\in \bR^n: \text{there exists }p_i\in \Delta_i \text{ such that }\lim_{i\to\infty}p_i = p\}.
\] 
The \emph{cofinite limit of $\{\Delta_i:i\geq 1\}$} is the set
\[
\Delta^c:= \overline{\{p\in \bR^n: p\in \Delta_i \text{ for all but finitely many $i$}\}} = \overline{\bigcup_{j\geq 0}\bigcap_{i\geq j} \Delta_i}.
\]
\end{defn}
One can show that both $\Delta^c$ and $\Delta^p$ are closed convex sets, and $\Delta^c\subseteq \Delta^p$. The following lemma shows that they must have the same volume.

\begin{lem}\label{lem:pointwise_limit_is_cofinite_limit}
Let $\Delta^p$ be the pointwise limit and $\Delta^c$ be the cofinite limit of a sequence of closed convex sets $\{\Delta_i : i\geq 1\}$. Then
\[
\vol_{\bR^n}(\Delta^c) = \vol_{\bR^n}(\Delta^p). 
\]
In particular, if $\vol_{\bR^n}(\Delta^c) > 0$, then $\Delta^c = \Delta^p$.
\end{lem}
\begin{proof}
Assume that $\vol_{\bR^n}(\Delta^c) < \vol_{\bR^n}(\Delta^p)$. Since both are closed convex sets, there exists a point $q\in \bR^n$ and a closed ball $B_r(q)$ of radius $r>0$ centered at $q$, such that $B_r(q)\subseteq \Delta^p$ and $B_r(q)\cap \Delta^c = \varnothing$. We may assume $q = 0\in \bR^n$.

For $1\leq i\leq n$, let $q_i = r e_i$ and $q_{i+n} = -r e_i$, where $e_i$ is the $i$th standard basis vector. Since $q_i\in B_r(q)$, by the definition of $\Delta^p$, there exist $q_{i,j}\in \Delta_j$ such that $\lim_{j\to\infty} q_{i,j} = q_i$. Thus, there exists $N\in\bN$ such that $|q_{i,j} - q_i| < \frac{r}{2\sqrt{n}}$ for all $j\geq N$ and $1\leq i\leq 2n$. However, this means that for $j\geq N$, the convex hull of $\{q_{i,j}:1\leq i\leq 2n\}$ contains a fixed ball $B_{r'}(0)$ where $r' = \frac{r}{2\sqrt{n}}$\footnote{One can show this by arguing that the distance from origin to the hyperplane passing through $\{q_{i,j}:1\leq i\leq n\}$ is at least $\frac{r}{2\sqrt{n}}$.}. In particular,
\[
B_{r'}(0)\subseteq \bigcap_{j\geq N} \Delta_j,
\]
which is a contradiction.
\end{proof}
\begin{rem}
    The equality $\Delta^c = \Delta^p$ may not hold in general. For example, if $\Delta_i = [\frac{1}{i},\frac{2}{i}]\subseteq \bR^n$ for all $i\geq 1$, then $\Delta^c = \varnothing$ but $\Delta^p = \{0\}.$
\end{rem}

\subsection{Limit of Newton--Okounkov bodies with respect to weight vectors}
We follow the set up at the beginning of Appendix~\ref{sec:variation_okounkov_bodies}. Let $\{\alpha^{(i)}\in {P_r}: i\geq 1\}$ be a sequence of weight vectors such that the limit $\alpha = \lim_{i\to\infty}\alpha^{(i)}$ exists in $\bR^n$. When $\alpha\in P_r$, we show that $\Delta(\alpha)$ is the limit of $\{\Delta(\alpha^{(i)}):i\geq 1\}$ in the sense of Definition~\ref{defn:limits_of_convex_bodies}.
\begin{lem}\label{lem:limit_of_okounkov_body_easy_case}
    Assume that $\alpha\in P_r$. Then $\Delta(\alpha)$ is the pointwise limit and the cofinite limit of $\{\Delta(\alpha^{(i)}): i\geq 1\}$.
\end{lem}
\begin{proof}
By Corollary~\ref{cor:volume_equality_of_okounkov_bodies}, $\vol_{\bR^n}(\Delta(\alpha)) = \frac{1}{n!}\vol(L) > 0$. Thus, by Lemma~\ref{lem:pointwise_limit_is_cofinite_limit}, it suffices to prove that $\Delta(\alpha)$ is the cofinite limit. 

Let $s\in H^0(X,mL)$. By Lemma~\ref{lem:Gamma_m_locally_constant}, we have \[\frac{1}{m}\nu_\bullet^{\alpha}(s)\in \frac{1}{m}\Gamma_m(\alpha^{(i)})\subseteq \Delta(\alpha^{(i)})\] for all sufficiently large $i$. This implies that 
\[
\Delta(\alpha) \subseteq \Delta^c:= \overline{\bigcup_{i\geq 0}\bigcap_{j\geq i}\Delta(\alpha^{(j)})}.
\]
On the other hand, for all $i$, we have
\[
\vol_{\bR^n} \left(\Delta(\alpha^{(i)})\right) = \frac{1}{n!}\vol(L) = \vol_{\bR^n}(\Delta(\alpha))
\]
by Corollary~\ref{cor:volume_equality_of_okounkov_bodies}. Thus, we have
\[
\vol_{\bR^n}(\Delta^c) = \vol_{\bR^n} \left( \overline{\bigcup_{i\geq 0}\bigcap_{j\geq i}\Delta(\alpha^{(j)})}\right)= \lim_{i\to\infty}\vol_{\bR^n}\left(\bigcap_{j\geq i} \Delta(\alpha^{(j)})\right)  \leq\frac{1}{n!}\vol(L) = \vol_{\bR^n}(\Delta(\alpha)).
\]
This combined with the containment $\Delta(\alpha)\subseteq \Delta^c$ proves that $\Delta(\alpha) = \Delta^c$.
\end{proof}

Next, we assume that $\alpha$ is no longer in $P_r$, but it satisfies the following condition:
\begin{equation}\label{assumption:last_entry_of_alpha_is_0}
    \text{$\alpha_r = 0$ and  $\alpha_1,\ldots,\alpha_{r-1}$ are $\bQ$-linearly independent.}
\end{equation}
Let $Z_{-1}$ be the unique irreducible component of $\bigcap_{i=1}^{r-1} D_i$ containing $Z$. Let $\nu_\bullet^\alpha$ be the $\bZ^n$-valuation defined using the quasi-monomial valuation $\alpha\in {\QM}_{Z_{-1}}(X, D_1 + \ldots + D_{r-1})$ and the admissible flag $Z_{-1} \supset Z = Z_0 \supset Z_1\supset \cdots\supset Z_{n-r}$. Let $\Delta(\alpha)$ be the Newton--Okounkov body $\Delta(\nu_\bullet, L)$.

\begin{prop}\label{prop:limit_of_okounkov_body_alpha_is_0_in_last_entry}
Under the assumption~\eqref{assumption:last_entry_of_alpha_is_0}, $\Delta(\alpha)$ is the pointwise limit and the cofinite limit of $\{\Delta(\alpha^{(i)}): i\geq 1\}$.
\end{prop}

\begin{proof}
By Corollary~\ref{cor:volume_equality_of_okounkov_bodies}, $\vol_{\bR^n}(\Delta(\alpha)) = n!\vol(L) > 0$. Thus, by Lemma~\ref{lem:pointwise_limit_is_cofinite_limit}, it suffices to prove that $\Delta(\alpha)$ is the cofinite limit. 

Suppose for each $1\leq i\leq r$, locally around $z$, $D_i$ is cut out by a function $x_i\in \cO_{X,z}$. Let $s\in H^0(X, mL)$ be a section locally given by a function $f\in \cO_{X,z}$. At the generic point of $Z_{-1}$, we can write 
\begin{align}\label{eqn:power_series_expansion_in_prop_5.7}
f = cx_1^{\beta_1}\cdots x_{r-1}^{\beta_{r-1}} + \text{(higher-order terms of $f$)} \in K(Z_{-1})[[x_1,\ldots,x_{r-1}]],
\end{align}
where $cx_1^{\beta_1}\cdots x_{r-1}^{\beta_{r-1}}$ is the lowest-order term of $f$ with respect to $v_\alpha$. By Lemma~\ref{lem:c_beta_has_no_poles}, $c\in \cO_{Z_{-1}, z}$. Note that $\cO_{Z_{-1},\eta}$ (where $\eta$ is the generic point of $Z$) is a discrete valuation ring with uniformizer $x_r$, we can write $c = u x_r^k$ for some unit $u\in \cO_{Z_{-1},\eta}$. In fact, $u\in \cO_{Z_{-1},z}$ because $x_r^k\cO_{Z_{-1}, \eta}\cap \cO_{Z_{-1},z} = x_r^k\cO_{Z_{-1},z}$ and $c\in \cO_{Z_{-1},z}$.

Let $\tilde{u}\in \cO_{X,z}$ be any lifting of $u$, and $\overline{u}\in \cO_{Z,z}$ be the image of $u$ under the quotient $\cO_{Z_{-1},z}\to \cO_{Z,z}$. We can write
\[
f = \tilde{u}x^\beta + f_1, 
\]
for some $f_1\in \cO_{X,z}$ and $\beta = (\beta_1,\ldots,\beta_{r-1},k)$. Since the image of $\tilde{u}$ in $\hat{\cO}_{X,Z_{-1}} \cong K(Z_{-1})[[x_1,\ldots,x_{r-1}]]$ lies in $u + (x_1,\ldots,x_{r-1})$, we have
\begin{align}\label{eqn:power_series_expansion_in_prop_5.7_2}
f \in cx_1^{\beta_1}\cdots x_{r-1}^{\beta_{r-1}} + (x_1,\ldots, x_{r-1})x_1^{\beta_1}\cdots x_{r-1}^{\beta_{r-1}} + f_1 \text{ in }K(Z_{-1})[[x_1,\ldots,x_{r-1}]].
\end{align}
Comparing \eqref{eqn:power_series_expansion_in_prop_5.7} and \eqref{eqn:power_series_expansion_in_prop_5.7_2} gives that $v_\alpha(f_1) > v_\alpha(f)$. Similarly, the image of $\tilde{u}$ in $\hat{\cO}_{X,\eta}\cong K(Z)[[x_1,\ldots,x_r]]$ lies in $\overline{u} + (x_1,\ldots,x_r)$, which implies that
\[
f \in \overline{u}x^\beta + (x_1,\ldots,x_r)x^\beta + f_1 \text{ in } K(Z)[[x_1,\ldots,x_r]].
\]
Since $v_\alpha(f_1) > v_\alpha(f)$, every monomial $x^\gamma$ in the power series expansion of $f_1$ satisfies $\alpha\cdot \beta < \alpha\cdot \gamma$. Hence, for sufficiently large $i$, the lowest-order term of $f$ with respect to $v_{\alpha^{(i)}}$ is $\overline{u}x^\beta$. This gives
\[
\nu_\bullet^{\alpha^{(i)}}(s) = (\beta, \nu_{Z_\bullet}(\overline{u})).
\]
We also have 
\[
\nu_\bullet^{\alpha}(s) = (\beta_1,\ldots,\beta_{r-1}, \nu_{Z_{\bullet\geq -1}}(c)) = (\beta_1,\ldots,\beta_{r-1}, k, \nu_{Z_{\bullet}}(\overline{u})) = (\beta, \nu_{Z_{\bullet}}(\overline{u})).
\]
Therefore, $\frac{1}{m}\nu_\bullet^{\alpha}(s)\in \Delta(\alpha^{(i)})$ for all sufficiently large $i$. This implies that
\[
\Delta(\alpha) \subseteq \Delta^c:= \overline{\bigcup_{i\geq 0}\bigcap_{j\geq i}\Delta(\alpha^{(j)})}.
\]
Since $\vol(\Delta^c)\leq \frac{1}{n!}\vol(L) = \vol(\Delta(\alpha))$ (details can be found in the proof of Lemma~\ref{lem:limit_of_okounkov_body_easy_case}), we conclude that $\Delta(\alpha) = \Delta^c$.
\end{proof}

In the more general case where $\alpha$ satisfies $\dim_{\bQ} \operatorname{span}_{\bQ}(\alpha_1,\ldots,\alpha_r) = r-1$, we still have an analog of Proposition~\ref{prop:limit_of_okounkov_body_alpha_is_0_in_last_entry}. By Lemma~\ref{lem:lin_transform_okounkov_body_under_further_blow_ups}, there exists a log resolution $\mu: (Y, E_1+\ldots + E_r)\to X$ satisfying the following properties:
\begin{itemize}
\item there is an irreducible component $Z_Y$ of $\bigcap_{i=1}^r E_i$ mapping isomorphically to $Z$ via $\mu$;
\item after passing to a subsequence, $v_{\alpha^{(i)}}\in {\QM}_{Z_Y}(Y, E_1 + \ldots + E_{r})$ for all $i$;
\item $v_\alpha\in {\QM}_{Z_Y}(Y, E_1 + \ldots + E_{r})$ and has weight vector $\alpha'$ with $\alpha_r' = 0$.
\end{itemize}
Let $\Delta_Y(\alpha^{(i)})$ denote the Newton--Okounkov body of $L$ constructed using $v_i$ and the admissible flag $Z_{Y,\bullet}$, where $Z_{Y,\bullet}$ are the preimages of $Z_\bullet$ under the isomorphism $Z_Y\cong Z$. Let $\Delta_Y(\alpha)$ denote the Newton--Okounkov body of $L$ constructed using $v$ and the admissible flag $Z_{Y, -1}\supset Z_{Y,0} = Z_{Y}\supset \cdots\supset Z_{Y,n-r}$, where $Z_{Y,-1}$ is the unique irreducible component of $\bigcap_{i=1}^{r-1}E_i$ containing $Z_Y$. Then by Proposition~\ref{prop:limit_of_okounkov_body_alpha_is_0_in_last_entry} and Lemma~\ref{lem:lin_transform_okounkov_body_under_further_blow_ups}(d), there exists an integral linear transformation $M\in \GL_n(\bZ)$ such that
\[
\Delta_Y(\alpha) = \lim_{i\to\infty} \Delta_Y(\alpha^{(i)}) = \lim_{i\to\infty} M\Delta(\alpha^{(i)}) = M\lim_{i\to\infty} \Delta(\alpha^{(i)}),
\]
where the limit is taken as the cofinite limit (or equivalently, the pointwise limit) in Definition~\ref{defn:limits_of_convex_bodies}.

\subsection{An explicit example: Newton--Okounkov bodies on $\bP^2$}\label{subsec:explicit_variations_of_okunkov_bodies_on_P2}
Let $C$ be an irreducible nodal cubic on $\bP^2$ with a node $p$. Let $L = \cO_{\bP^2}(1)$. For each $\alpha\in \bR_{>0}^2$, let $v_\alpha$ be the quasi-monomial valuation in $\text{QM}_{p}(X, C)$ with weight vector $\alpha$ (see Remark~\ref{remark:qm_val_wrt_nc_divisors}). Assume $\alpha_1, \alpha_2$ are $\bQ$-linearly independent. Denote $\Delta(\alpha)$ the Newton--Okounkov body of $L$ with respect to $v_\alpha$ (see Section~\ref{subsubsec:okunkov_body_wrt_nc_divisors_with_zero_dim_strata}). Then $\Delta(\alpha)$ can be explicitly computed as follows.

\begin{prop}[cf. \cite{LXZ22}*{Theorem 6.1}]\label{prop:explicit_okunkov_body_on_P2}
Let $\{d_n:n\in \bZ\}$ be the sequence of integers such that $d_0 = d_1 = 1$ and $d_{n+1} = 3d_n -d_{n-1}$. Then $\Delta(\alpha)\subseteq \bR^2$ is a cone with the following vertices:
\begin{itemize}
\item If $\frac{\alpha_1}{\alpha_2}\in \left(\frac{7-3\sqrt{5}}{2}, \frac{7+3\sqrt{5}}{2}\right)$, then $\frac{\alpha_1}{\alpha_2} \in \left(\frac{d_{n+1}}{d_{n-1}}, \frac{d_{n+2}}{d_{n}}\right)$ for some integer $n$ and the vertices are $\left(0, \frac{d_{n+1}}{d_n}\right)$ and $\left(\frac{d_n}{d_{n+1}},0\right)$.
\item If $\frac{\alpha_1}{\alpha_2}\in (0, \frac{7-3\sqrt{5}}{2}] \cup [\frac{7+3\sqrt{5}}{2}, \infty)$, then the vertices are $\left(\frac{1}{3}, \frac{1}{3}\right), \left(\frac{3\alpha_2}{\alpha_1+\alpha_2}, 0\right)$, and $ \left(0,\frac{3\alpha_1}{\alpha_1+\alpha_2}\right)$.
\end{itemize}
\end{prop}
In particular, $\Delta(\alpha)$ varies discretely in the region $\frac{\alpha_1}{\alpha_2}\in \left(\frac{7-3\sqrt{5}}{2}, \frac{7+3\sqrt{5}}{2}\right)$ and continuously outside this region.

We need the following lemma from \cite{LXZ22} to prove Proposition~\ref{prop:explicit_okunkov_body_on_P2}.
\begin{lem}\cite{LXZ22}*{Lemma 6.5}\label{lem:LXZ_lemma_6.5}
Suppose the two branches of $C$ are locally analytically defined by $x_1, x_2\in \hat{\cO}_{\bP^2,p}$ in a neighborhood of $p$. Then for any integer $n$, there exists an irreducible curve $D_n$ on $X$ such that
\begin{itemize}
\item $D_n\sim \cO_{\bP^2}(d_n)$,
\item $D_n$ is a $(d_{n-1}, d_{n+1})$-unicuspidal curve in the $(x_1,x_2)$ coordinate, i.e., it is locally cut out by a function $f\in \bC[[x_1,x_2]]$ whose Newton polygon is given by \[\left\{(a,b)\in \bN: \frac{a}{d_{n-1}} + \frac{b}{d_{n+1}} \geq 1\right\}.\]
\end{itemize}
\end{lem}

\begin{lem}\label{lem:okounkov_body_on_P2_irrational_polytope_case}
Suppose $\frac{\alpha_1}{\alpha_2} > \frac{7+3\sqrt{5}}{2}$. Then we have
\[
\left(\frac{3\alpha_2}{\alpha_1+\alpha_2}, 0\right), \left(0,\frac{3\alpha_1}{\alpha_1+\alpha_2}\right)\in\Delta(\alpha).
\]
\end{lem}
\begin{proof}
Let $d_1, d_2\in \bZ_{>0}$ such that 
\[
\frac{7+3\sqrt{5}}{2} < \frac{d_1}{d_2} < \frac{\alpha_1}{\alpha_2}.
\]
Let $C_1, C_2$ be irreducible curves on $\bP^2$ such that locally at $p$, $C_i$ is cut out by a function $x_i'\in \hat{\cO}_{X,p}$ of the form
\[
x_i' \in x_i + \hat{\mathfrak{m}}^N,
\]
where $N = \lceil \frac{\alpha_1}{\alpha_2}\rceil$ and $\hat{\fm}$ is the maximal ideal of $\hat{\cO}_{X,p}$.
Let $\mu: Y\to \bP^2$ be the weighted blow up of $\bP^2$ at $p$ with respect to $C_1, C_2$ and weights $(d_1,d_2)$. Let $E$ denote the exceptional divisor of $\mu$ and $C_Y$ the strict transform of $C$ on $Y$. Then $C_Y\sim \mu^*\cO_{\bP^2}(3) - (d_1+d_2)E$. Consider the effective $\bQ$-divisor
\[
B = C_Y + \left(d_1 + d_2 - \frac{9d_1d_2}{d_1+d_2}\right) E \sim_{\bQ} \mu^*\cO_{\bP^2}(3) - \frac{9d_1d_2}{d_1+d_2} E.
\]
One may compute that $B\cdot C_Y = 0$, $B\cdot E>0$ and $B^2>0$. Thus, $B$ is big and nef. This also implies that $B + \epsilon E$ is ample for any $\epsilon \in \left(0, \frac{9d_1d_2}{d_1+d_2}\right)$. Let $\epsilon = \frac{d_1d_2}{m(d_1+d_2)}$ for some positive integer $m$ and $M = m(d_1+d_2)(B+\epsilon E)$. Then $M$ is an ample line bundle on $Y$. Let $k>0$ be a sufficiently large integer such that 
\[
H^1(Y, \cO_Y(kM - E)) = 0.
\]
From the long exact sequence of cohomology, the map
\begin{align}\label{eqn:lifting_sections}
H^0(Y, \cO_Y(kM)) \to H^0(E, \cO_E(kM))
\end{align}
is surjective. Since $E\cong \bP^1$ and $\cO_E(d_1d_2E)\cong \cO_{\bP^1}(-1)$, we have $\cO_E(kM)\cong \cO_{\bP^1}(k(9m-1))$. For $i = 1,2$, let $p_i = C_{i,Y}\cap E$. By the surjectivity of \eqref{eqn:lifting_sections}, there exists a divisor $F_Y\in |\cO_Y(kM)|$ lifting the divisor $k(9m-1)p_1 \in |\cO_E(kM)|$. Let $F = \mu_*F_Y \sim \cO_{\bP^2}(3km(d_1+d_2))$. Then
\[
\mu^*F = F_Y + k(9m-1)d_1d_2 E.
\]

Suppose $F = \mu_*F_Y$ is locally defined by an equation
\[
f = \sum_{\beta\in \bN^2} c_\beta x'^\beta\in \bC[[x_1',x_2']].
\]
We can write $f$ as $f = f_0 + f_1$, where
\[
f_0 = \sum_{\beta: d_1\beta_1 + d_2\beta_2 = \ord_E(F)} c_\beta x'^\beta.
\]
We have
\begin{align*}
    k(9m-1)p_1 &= F_Y|_E = (\mu^*F - \ord_E(F)E) |_{E}\\
    & = (\prindiv \mu^*(f_0) - \ord_E(F)E)|_E \\
    &= \prindiv \left(\sum_{\beta: d_1\beta_1 + d_2\beta_2 = \ord_E(F)} c_\beta u^{\beta}\right), 
\end{align*}
where $u_1,u_2$ are coordinates on $E$ such that $\prindiv(u_1^{d_2}) = p_1$ and $\prindiv(u_2^{d_1}) = p_2$. As a result, $f_0$ is a monomial of the form $cx_1'^{\beta_1}$, where $\beta_1 = \frac{\ord_E(F)}{d_1} = k(9m-1)d_2$. As a result,
\[
f = cx_1'^{\beta_1} + f_1(x_1',x_2') = cx_1^{\beta_1} + g_1(x_1,x_2)  \in \bC[[x_1,x_2]]
\] 
and by our choice of $N$, every monomial $x_1^{e_1}x_2^{e_2}$ in $g_1$ satisfies $d_1e_1 + d_2e_2 > d_1\beta_1$. This implies that $x_1^{\beta_1}$ is the unique lowest-order term of $f$ with respect to $v_{(d_1,d_2)}$. Therefore, $\nu_\bullet^{\alpha'}(F) = (\beta_1, 0)$ for any $\alpha'\in P_2$ sufficiently close to $(d_1,d_2)$. In particular,
\[
\left(\frac{\beta_1}{3km(d_1+d_2)}, 0\right) = \left(\frac{(9m-1)d_2}{3m(d_1+d_2)}, 0\right) \in \Delta(\alpha').
\]
Finally, we can take $(d_1,d_2)$ such that $\frac{d_1}{d_2}$ is arbitrarily close to $\frac{\alpha_1}{\alpha_2}$ and $m$ is arbitrarily large. By Lemma~\ref{lem:limit_of_okounkov_body_easy_case}, we have
\[
\left(\frac{3\alpha_2}{\alpha_1+\alpha_2}, 0\right) \in \Delta(\alpha),
\]
as desired. One can similarly prove that $\left(0,\frac{3\alpha_1}{\alpha_1+\alpha_2}\right)\in \Delta(\alpha)$.

\end{proof}

\begin{proof}[Proof of Proposition~\ref{prop:explicit_okunkov_body_on_P2}]
Suppose $\frac{\alpha_1}{\alpha_2}\in \left(\frac{7-3\sqrt{5}}{2}, \frac{7+3\sqrt{5}}{2}\right)$. Then $\frac{\alpha_1}{\alpha_2} \in \left(\frac{d_{n+1}}{d_{n-1}}, \frac{d_{n+2}}{d_{n}}\right)$ for some integer $n$, since $\{\frac{d_{n+1}}{d_{n-1}}: n\in\bZ\}$ is a strictly increasing sequence with limits
\[
\lim_{n\to \pm \infty} \frac{d_{n+1}}{d_{n-1}} = \frac{7\pm 3\sqrt{5}}{2}.
\]
By Lemma~\ref{lem:LXZ_lemma_6.5}, there exist divisors $D_n\sim \cO_{\bP^2}(d_n)$ and $D_{n+1}\sim \cO_{\bP^2}(d_{n+1})$ such that
\[
\nu_\bullet^{\alpha}(D_n) = (0,d_{n+1}) \quad\text{and}\quad \nu_\bullet^{\alpha}(D_{n+1}) = (d_n, 0).
\]
This gives $\left(0, \frac{d_{n+1}}{d_n}\right), \left(\frac{d_n}{d_{n+1}},0\right) \in \Delta(\alpha)$. Moreover, $(0,0)\in \Delta(\alpha)$ since $L$ is ample. Then $\Delta(\alpha)$ contains the closed convex hull of $\left(0, \frac{d_{n+1}}{d_n}\right), \left(\frac{d_n}{d_{n+1}},0\right)$, and $(0,0)$, which has area $\frac{1}{2}$. However, by Corollary~\ref{cor:volume_equality_of_okounkov_bodies}, $\vol_{\bR^2}(\Delta(\alpha)) = \frac{1}{2}\vol(L) = \frac{1}{2}$. This implies that $\Delta(\alpha)$ must be the closed convex hull of $\left(0, \frac{d_{n+1}}{d_n}\right), \left(\frac{d_n}{d_{n+1}},0\right)$, and $(0,0)$, as desired.

Now suppose $\frac{\alpha_1}{\alpha_2} > \frac{7+3\sqrt{5}}{2}$. By Lemma~\ref{lem:okounkov_body_on_P2_irrational_polytope_case} and the fact that $\nu_\bullet^\alpha(C) = (1,1)$, we have
\[
(0,0), \left(\frac{1}{3},\frac{1}{3}\right), \left(0,\frac{3\alpha_1}{\alpha_1+\alpha_2}\right), \left(\frac{3\alpha_2}{\alpha_1+\alpha_2}, 0\right)\in \Delta(\alpha).
\]
However, the closed convex hull of these four points already has area $\frac{1}{2}$, so it must be equal to $\Delta(\alpha)$ by the volume comparison. 

The case $\frac{\alpha_1}{\alpha_2} < \frac{7-3\sqrt{5}}{2}$ can be worked out similarly, and the proof concludes.
\end{proof}

\begin{rem}
We may explicitly compute Newton--Okounkov bodies $\Delta(\alpha)$ on other del Pezzo surfaces $X$ with nodal anti-canonical divisors $C$ as follows. We can use \cite{Pen25}*{Proposition 2.22} to find unicuspidal divisors on $X$, whose linear combinations will correspond to vertices of $\Delta(\alpha)$. Then, we can apply the volume comparison to conclude that $\Delta(\alpha)$ must be the closed convex hull of these vertices.
\end{rem}

\begin{rem}
Using the techniques of special valuations in \cite{LXZ22}, Proposition~\ref{prop:explicit_okunkov_body_on_P2} implies that $\bP^2$ admits a special $\bG_m$-equivariant degeneration to $\bP(1, d_n^2, d_{n+1}^2)$ for every $n\in\bZ$. In fact, \cite{ABBDILW23}*{Theorem A.3} classifies all such degenerations of $\bP^2$: they are either weighted projective planes $\bP(1, d_n^2, d_{n+1}^2)$ or weighted hypersurfaces of the form $\{x_0x_3 = x_1^{d_{n+1}} + x_2^{d_{n-1}}\}\subseteq \bP(1, d_{n-1},d_{n+1},d_n^2)$.
\end{rem}

\section{Weighted projective stacks and weighted blow-ups}\label{sec:weighted_proj_stacks}
In this section, we define the notions of weighted projective stacks and weighted blow-ups. The main references for this section are \cite{AH11} and \cite{QR22}.

\subsection{Weighted projective stacks}
Let $B$ be a finite type scheme over $\bC$ (or any field of characteristic zero). Let $V_1, \ldots, V_r$ be locally free sheaves on $B$ and $d_1,\ldots, d_r$ be positive integers. 

\begin{defn}
Set
\[
V := \bigoplus_{i=1}^r V_it^{d_i} \]
be a graded $\cO_B$-module (which is graded by $t$). Let
\[
\Sym_B^\bullet V := \bigotimes_{i=1}^r \left(\bigoplus_{k\geq 0} V_i^{\otimes k}t^{kd_i}\right)
\]
be the weighted symmetric algebra of $V$. Let
\[
\bA(V) := \Spec_B (\Sym_B^\bullet V)
\]
be the associated vector bundle over $B$, with zero section $0_B\subseteq \bA(V)$. Consider the $\bG_m$-action on $\bA(V)$ where $\bG_m$ acts on $t$ with weight $-1$. Then \emph{the weighted projective stack associated to $V$} is the quotient stack over $B$ defined by
\[
\sP_B(V) := [(\bA(V) - 0_B)/\bG_m].
\]
\end{defn}

Next, we describe the local charts of $\sP_B(V)$. Since locally each $V_i$ is trivial, we may assume (after shrinking $B$ and splitting $V_i$ into line bundles) that each $V_i$ is a trivial line bundle. Let $R = \Sym_B^\bullet V$. Then
\[
R \cong \cO_B[s_1t^{d_1}, \ldots, s_rt^{d_r}],
\]
where $s_i$ is a trivializing section of $V_i$ for each $i$. Let $R_+$ be the irrelevant ideal of $R$, which consists of elements of $R$ with positive $t$-degrees. Let $V(R_+)\subseteq \Spec R$ be the vanishing locus of the ideal $R_+$. Then
\[
\sP_B(V) = [(\Spec R - V(R_+))/ \bG_m] \cong [(\bA^r_B - 0_B) / \bG_m].
\] Note that 
\[\Spec R - V(R_+)  = \bigcup_{i=1}^r \Spec R[(s_it^{d_i})^{-1}].\]
Therefore, $\sP_B(V)$ is covered by the local charts
\[
\sD_+(s_it^{d_i}) := [\Spec R[(s_it^{d_i})^{-1}]/\bG_m] \cong \left[\Spec \frac{R[(s_it^{d_i})^{-1}]}{s_it^{d_i}-1} / \mu_{d_i} \right] \cong [\bA_B^{r-1}/\mu_{d_i}],
\]
where the second-to-last isomorphism follows from \cite{QR22}*{Lemma 1.3.1}.

\begin{defn}[\cite{AH11}*{Definition 2.3.1 and 2.3.3}]
A flat, separated, finite type algebraic stack $p:\sX\to B$ is said to be a \emph{cyclotomic stack} if the stabilizer of any geometric point of $\sX$ is isomorphic to the group scheme $\mu_n$ for some $n$.

We say that a stack $p:\sX\to B$ \emph{has index $N$} if $N$ is the minimal positive integer satisfying the following condition: for each object $\xi\in \sX(T)$ over a scheme $T$ and each automorphism $a\in \Aut(\xi)$, we have $a^N = \id$.
\end{defn}

The following lemma follows directly from the description of the local charts of $\sP_B(V)$:
\begin{lem}[\cite{AH11}*{Exercise 2.3.2}]\label{lem:weighted_proj_stack_is_cyclotomic}
The weighted projective stack $p:\sP_B(V)\to B$ is a cyclotomic stack of index $N = \lcm(d_1,\ldots,d_r)$.
\end{lem}

By Lemma~\ref{lem:weighted_proj_stack_is_cyclotomic}, the weighted projective stack $\sP_B(V)\to B$ is a separated Deligne-Mumford stack with finite diagonal. Hence, by the Keel-Mori theorem, it has a coarse moduli space $\bP_B(V)$ which is in fact a projective scheme over $B$.
\begin{defn}
We denote the coarse moduli space morphism as
\[
\pi: \sP_B(V)\to \bP_B(V).
\]
We call $\bP_B(V)$ \emph{the weighted projective bundle associated to $V$}.
\end{defn}

On $\sP_B(V)$, there is an invertible sheaf $\cO_{\sP_B(V)}(d)$ which corresponds to the graded $R$-module $R^{[d]} := t^{-d}R$ (i.e., shifting the grading of $R$ by $d$), for every $d\in \bZ$. Using the description of local charts, one can prove the following fundamental properties of weighted projective bundles.
\begin{lem} 
Let $p: \sP_B(V)\to B$ and $N = \lcm(d_1,\ldots,d_r)$. 
\begin{itemize}
\item[(a)] For every $d\in \bZ$, $\cO_{\sP_B(V)}(d)$ descends to a coherent sheaf $\cO_{\bP_B(V)}(d)$ on $\bP_B(V)$.
\item[(b)] $\cO_{\bP_B(V)}(d)$ is an invertible sheaf if and only if $d$ is divisible by $N$.
\item[(c)] We have 
\[
p_*\cO_{\sP_B(V)}(d) \cong (\Sym^\bullet_B V)_d := \bigoplus_{m_1d_1 + \cdots + m_rd_r = d} \left(\bigotimes_{k=1}^r V_k^{\otimes m_k}\right).
\]
\item[(d)] $\pi$ induces an isomorphism
\[
\pi^*: \Pic(\bP_B(V)) \otimes \bZ[N^{-1}] \overset{\sim}\to \Pic(\sP_B(V))\otimes \bZ[N^{-1}].
\]
\end{itemize} 
\end{lem}
\begin{proof}
(a) For each $i$, denote 
\[
R_i = R[(s_it^{d_i})^{-1}] = \bigoplus_{m\in \bZ} R_{i,m}t^{m}.
\]
Then $\sD_+(s_it^{d_i}) = [\Spec R_i/\bG_m]$, whose coarse moduli space is the spectrum of the ring $R_i^{\bG_m} \cong R_{i,0}$. On $\sD_+(s_it^{d_i})$, $\cO_{\sP_B(V)}(d)$ corresponds to the $\bZ$-graded $R_i$-module 
\[
M_i := t^{-d}R_i = \bigoplus_{m\in \bZ} R_{i,m}t^{m-d}.
\]
Let $\pi_i: \sD_+(s_it^{d_i})\to \Spec R_{i,0}$ be the restriction of the coarse moduli space morphism $\pi$. Let $N_i = \pi_{i*}M_i$. Then $N_i$ can be identified as the $R_{i,0}$-module $R_{i,d}$. We have
\begin{align*}
    \pi_i^*N_i &= R_{i,d}\otimes_{R_{i,0}} R_i = R_{i,d}\otimes_{R_{i,0}} \left(\bigoplus_{m\in \bZ} R_{i,m}t^m \right) \\
    &= \bigoplus_{m\in \bZ} \left(R_{i,d}\otimes_{R_{i,0}} R_{i,m}\right) t^m \\
    &\cong \bigoplus_{m\in \bZ} R_{i,d+m} t^m = M_i.
\end{align*}
Here, we use the fact $R_{i,m}\otimes_{R_{i,0}} R_{i,m'} \cong R_{i,m+m'}$ for any $m,m'\in \bZ$.\footnote{This fact can be proved using the following observation. Assume $\gcd(d_1,\ldots,d_r) = 1$. If $m+m' = \sum_{j=1}^r a_jd_j$ for some integers $a_j$ such that $a_j \geq 0$ for all $j\neq i$, then there exists integers $b_j$ such that $b_j\geq 0$ for all $j\neq i$ and $m = \sum_{j=1}^r b_jd_j$.} Thus, $M_i\cong \pi_i^*\pi_{i*}M_i$ for all $i$. This implies that $\cO_{\sP_B(V)}(d)\cong \pi^*\pi_*\cO_{\sP_B(V)}(d)$ and it descends to the coherent sheaf $\cO_{\bP_B(V)}(d) = \pi_*\cO_{\sP_B(V)}(d)$. \\

\noindent (b) For each $i$, denote 
\[
\overline{R_i} = \frac{R[(s_it^{d_i})^{-1}]}{s_it^{d_i}-1}\cong \cO_B[\{s_jt^{d_j}:j\neq i\}].
\]
Then $\sD_+(s_it^{d_i}) = [\Spec \overline{R_i}/\mu_{d_i}]$, on which $\cO_{\sP_B(V)}(d)$ corresponds to the $\bZ/d_i$-graded $\overline{R_i}$-module $\overline{M_i}:= t^{-d}\overline{R_i}$. Let $0_B\subseteq \Spec \overline{R_i} \cong \bA_B^{r-1}$ be the zero section, then $\overline{M_i}|_{0_B} = t^{-d} \cO_B$, i.e., the group $\mu_{d_i}$ acts on it with weight $d$. By \cite{Alp25}*{Proposition 4.4.27}, $\overline{M}_i$ descends to a locally free sheaf on the coarse moduli space of $\sD_+(s_it^{d_i})$ if and only if $d$ is divisible by $d_i$. We obtain the desired statement after combining all the local charts of $\sP_B(V)$. \\

\noindent (c) Locally on $B$, $p_*\cO_{\sP_B(V)}(d)$ corresponds to the degree-zero piece of $t^{-d}R$, which is the degree-$d$ piece of $\Sym_B^\bullet V$.\\

\noindent (d) This follows from Lemma~\ref{lem:weighted_proj_stack_is_cyclotomic} and \cite{AH11}*{Lemma 2.3.7}.
\end{proof}

In particular, (d) implies that $\bQ$-line bundles on $\sP_B(V)$ and $\bP_B(V)$ can be identified and many properties (such as ampleness) can be directly transferred from one side to the other.

\begin{lem}\label{lem:ample_vec_bundle_implies_ample_O(1)}
Suppose $B$ is quasi-projective (over a field of characteristic zero).
\begin{itemize}
\item[(a)] Let $L$ be an invertible sheaf on $B$ and $d$ be a positive integer. Let
\[
V' = \bigoplus_{i=1}^r (V_i^{\otimes d}\otimes L^{\otimes d_i} )t^{d_i}
\]
Then there is a finite morphism
\[
\phi: \sP_B(V) \to \sP_B(V')
\]
such that $\phi^*\cO_{\sP_B(V')}(1) \cong \cO_{\sP_B(V)}(d) \otimes p^*L$.
\item[(b)] Assume $V_i$ is an ample vector bundle on $B$ for all $i$. Then $\cO_{\bP_B(V)}(1)$ is an ample $\bQ$-line bundle.
\end{itemize}
\end{lem}
\begin{proof}
    (a) follows from the universal property of the stack-theoretic Proj (see \cite{QR22}*{Proposition 1.5.1 and Section 1.7.1}). (b) follows from the cohomological criterion for ample vector bundles \cite{Laz04b}*{Theorem 6.1.10} and the Leray spectral sequence.
\end{proof}

We also need the following lemma.
\begin{lem}\label{lem:pic_group_weighted_proj_stack}
We have
\[
\Pic(\sP_B(V)) \cong \bZ [\cO_{\sP_B(V)}(1)] \oplus \Pic(B).
\]
\end{lem}
\begin{proof}
    This statement follows from the base change property of the coarse moduli space of weighted projective stacks and the semi-continuity theorem of cohomology. The details can be found in an unpublished note of Noohi \cite{Noo}*{Section 6 and 7}.
\end{proof}

\begin{lem}\label{lem:toric_hyperplane_in_weighted_proj_stack}
Suppose $V_i$ is an invertible sheaf for some $i$. Consider the closed substack
\[
\mathscr{H}_i := \sP_B\left(\bigoplus_{j\neq i} V_j t^{d_j}\right) \hookrightarrow \sP_B(V).
\]
Then $\mathscr{H}_i$ is a Cartier divisor on $\sP_B(V)$ and there is an isomorphism
\[
\cO_{\sP_B(V)} (\mathscr{H}_i) \cong \cO_{\sP_B(V)}(d_i) \otimes p^*V_i^{-1}
\]
\end{lem}
\begin{proof}
By Lemma~\ref{lem:pic_group_weighted_proj_stack}, we know $\cO_{\sP_B(V)} (\mathscr{H}_i) \cong \cO_{\sP_B(V)}(d_i) \otimes p^*L$ for some line bundle $L$ on $B$.

Let $\hat{V}_i = \bigoplus_{j\neq i} V_j t^{d_j}$. Applying $p_*$ to the short exact sequence
    \[
    0\to \cO_{\sP_B(V)} (-\mathscr{H}_i)(d_i) \to \cO_{\sP_B(V)}(d_i)\to \cO_{\sP_B(\hat{V}_i)}(d_i) \to 0,
    \]
    we obtain
    \[
    0\to p_*\cO_{\sP_B(V)} (-\mathscr{H}_i)(d_i) \to (\Sym_B V)_{d_i} \to (\Sym_B \hat{V}_i)_{d_i}.
    \]
    This implies that
    \[
    L^{-1}\cong p_*\cO_{\sP_B(V)} (-\mathscr{H}_i)(d_i)\cong V_i,
    \]
    as desired.
\end{proof}

\subsection{Weighted blow-ups} 
Let $X$ be a finite type scheme (over a field of characteristic zero).
\begin{defn}
    \emph{A Rees algebra on $X$} is a quasi-coherent, finitely generated, graded $\cO_X$-algebra $R = \bigoplus_{n\geq 0} I_n t^n$, where $\{I_n: n\in \bN\}$ is a sequence of ideal sheaves satisfying the following:
\begin{itemize}
\item $I_0 = \cO_X$;
\item $I_n \supseteq I_{n+1}$ for all $n\in \bN$;
\item $I_n \cdot I_m \subseteq I_{m+n}$ for all $m,n\in \bN$.
\end{itemize}
\end{defn}

\begin{defn}
Let $R = \bigoplus_{n\geq 0} I_nt^n$ be a Rees algebra on $X$. Let $R_+ = \bigoplus_{n\geq 1} I_nt^n$ be the irrelavant ideal of $R$. \emph{The stack-theoretic Proj of $R$} is the quotient stack over $X$ given by
\[
\sX' = \scrProj_X R := [(\Spec_X (R) - V(R_+))/\bG_m],
\]
where $\bG_m$ acts on $t$ with weight $-1$. We say that $\sX'$ is \emph{the stack-theoretic weighted blow-up of $X$ along $R$}, denoted as $\scrBl_R X$ or $\scrBl_{I_\bullet}X$.

The natural inclusion $I_{\bullet+1} \subseteq I_{\bullet}$ corresponds to the inclusion $\cO_{\sX'}(1)\subseteq \cO_{\sX'}$ of invertible sheaves on $\sX'$. Thus, it defines an effective Cartier divisor $\sE$ on $\sX'$ such that $\cO_{\sX'}(-\sE) = \cO_{\sX'}(1)$. ${\sE}$ is called \emph{the exceptional divisor of} $\sX' \to X$.
\end{defn}

Next, we describe the local charts of $\sX' = \scrBl_{R} X$. Suppose $f_1,\ldots,f_r\in R_+$ are homogeneous elements of degree $d_1,\ldots,d_r\in \bZ_{>0}$ such that $R_+\subseteq \sqrt{(f_i: 1\leq i\leq r)}$. Then
\[
 \Spec_X(R) - V(R_+) = \bigcup_{i=1}^r \Spec_X(R[f_i^{-1}]).\]
Thus, we have an open covering of $\sX'$ by 
\begin{align}\label{eqn:local_charts_of_weighted_blow_ups}
\sD_+(f_i) = [\Spec_X(R[f_i^{-1}]) / \bG_m] \cong \left[\Spec_X\left(\frac{R[f_i^{-1}]}{f_i-1}\right) / \mu_{d_i}\right].
\end{align}
where the last isomorphism follows from \cite{QR22}*{Lemma 1.3.1}. The description of local charts implies the following lemma.

\begin{lem}
    $\scrBl_{R} X \to X$ is a cyclotomic stack of finite index.
\end{lem}

Hence, by the Keel-Mori theorem, we can make the following definition.
\begin{defn}
We denote \[\pi: \scrBl_{R} X \to \Bl_{R} X\]  
as the coarse moduli space morphism. We say that $\Bl_{R} X\to X$ is \emph{the weighted blow-up of $X$ along $R$}.
\end{defn}

\begin{lem}\label{lem:O(1)_is_rel_ample}
Suppose $X$ is quasi-projective, Let $R$ be a Rees algebra on $X$ and $X' = \Bl_{R}X$. Then the sheaf $\cO_{{X}'}(N)$ is invertible and relatively ample over $X$ for every positive integer $N$ which is sufficiently divisible. In particular, $X'$ is a projective scheme over $X$.
\end{lem}
\begin{proof}
    The proof is essentially given by \cite{Har77}*{II, Proposition 7.10}, except that one needs to replace $\cO_{X'}(1)$ by the line bundle $\cO_{X'}(N)$, where $N$ is any positive integer divisible by the index of the cyclotomic stack $\scrBl_R X$.
\end{proof}

\begin{comment}
    One can also obtain Lemma~\ref{lem:O(1)_is_rel_ample} from the following lemma in \cite{QR22}.
\begin{lem}\cite{QR22}*{Proposition 1.6.3}
Let $R = \bigoplus_{m} I_mt^m$ be a Rees algebra on $X$. Then $\Bl_{R}X$ is isomorphic to the ordinary blow-up of $X$ along the ideal $I_N$ for any positive integer $N$ such that $R^{(N)}:= \bigoplus_{m}I_{Nm}t^{Nm}$ is generated in degree 1. 
\end{lem}
\end{comment}

\subsection{Weighted blow-ups along strata of snc divisors}\label{subsec:weighted_blow_up_along_snc_div}
Let $X$ be a smooth variety and $D_1,\ldots, D_r$ be simple normal crossing divisors on $X$. Assume $\bigcap_{i=1}^r D_i\neq \varnothing$. Let $d_1,\ldots,d_r$ be positive integers.

Given an ideal $I\subseteq \cO_X$ and $d\geq 1$, let $(I,d)$ be the smallest Rees algebra containing $I t^d$. In other words,
\[
(I,d) = \bigoplus_{m\in \bN} I^{\lceil m/d\rceil} t^m.
\]
For $1\leq k\leq r$, let $I_{D_k}$ be the ideal sheaf of $D_k$. Let $R = \bigoplus_{m\in\bN}I_mt^m$ be the smallest Rees algebra that contains $(I_{D_k}, d_k)$ for all $1\leq k\leq r$. More explicitly, we have
\[
I_{m} = \sum_{m_1 + \ldots + m_r = m} I_{D_1}^{\lceil m_1/d_1\rceil} I_{D_2}^{\lceil m_2/d_2\rceil} \cdots I_{D_r}^{\lceil m_r/d_r\rceil}
\]
for every $m\in \bN$.

\begin{defn}\label{defn:weighted_blow_up_wrt_snc_divisors}
\emph{The stack-theoretic weighted blow-up of $X$ with respect to $\{D_i: 1\leq i\leq r\}$ and weights $\{d_i:1\leq i\leq r\}$} is
\[
\sX'  := \scrBl_{R} X = [(\Spec_X R - V(R_+))/\bG_m].
\]
Its coarse moduli space $X' = \Bl_{R} X$ is called \emph{the weighted blow-up of $X$ with respect to $\{D_i: 1\leq i\leq r\}$ and weights $\{d_i:1\leq i\leq r\}$}.
\end{defn}

Now we describe the local charts of $\sX' = \scrBl_{R}X$. Suppose for $1\leq i\leq r$, locally $D_i$ is cut out by a function $f_i$. Let $Z=\bigcap_{i=1}^r D_i$. Note that
\[
I_m = \bigoplus_{e_1d_1 + \cdots + e_rd_r \geq m} \cO_X\cdot f_1^{e_1} \cdots f_r^{e_r}
\]
and hence $R_+ \subseteq \sqrt{(f_it^{d_i}: 1\leq i\leq r)}$. By \eqref{eqn:local_charts_of_weighted_blow_ups}, the local charts of $\sX'$ are
\[
\sD_+(f_it^{d_i}) = [\Spec_X R[(f_it^{d_i})^{-1}]/\bG_m].
\]
On $\sD_+(f_it^{d_i})$, the exceptional divisor $\sE$ is given by $t^{-1} = 0$, where $t^{-1} =\frac{f_it^{d_i-1}}{f_it^{d_i}}\in R[(f_it^{d_i})^{-1}]$. As a result, the local charts of $\sE$ are given by
\[
\sE\cap \sD_+(f_it^{d_i}) = \left[\Spec \frac{R[(f_it^{d_i})^{-1}]}{t^{-1}}/\bG_m\right]\cong [\Spec \cO_Z[f_1t^{d_1},\ldots, f_rt^{d_r}, (f_it^{d_i})^{-1}]/\bG_m],
\]
where the isomorphism \[\frac{R[(f_it^{d_i})^{-1}]}{t^{-1}}\cong \cO_Z[f_1t^{d_1},\ldots, f_rt^{d_r}, (f_it^{d_i})^{-1}]\] can be deduced from the fact that $f_1^{e_1}\cdots f_r^{e_r}t^m = 0$ in $R[(f_it^{d_i})^{-1}]/t^{-1}$ unless $\sum_{i=1}^r e_id_i = m$. These are precisely the local charts of $\sP_Z(V_Z)$, where $V_Z = \bigoplus_{i=1}^r \cO_Z(-D_i)t^{d_i}$. Thus, we have proved the following lemma.

\begin{lem}\label{lem:exc_div_is_weighted_proj_stack}
Let $\sX'\to X$ be the stack-theoretic weighted blow-up of $X$ with respect to $\{D_i: 1\leq i\leq r\}$ and weights $\{d_i:1\leq i\leq r\}$.
Let $Z = \bigcap_{i=1}^r D_i$ and $V_Z = \bigoplus_{i=1}^r \cO_Z(-D_i)t^{d_i}$. Let ${\sE}$ be the exceptional divisor of $\sX'\to X$. Then we have an isomorphism
\[
{\sE} \cong \sP_Z(V_Z).
\]
It descends to an isomorphism 
\[
E\cong \bP_Z(V_Z)
\]
where $E$ is the exceptional divisor of $X'\to X$. Furthermore, we have
\[
\cO_{\sX'}(1) |_{\sE} \cong \cO_{\sE}(1),
\]
where $\cO_{\sE}(1)$ is the tautological line bundle of the weighted projective stack $\sE\to Z$.
\end{lem}

Finally, since the the weighted blow-up $X'\to X$ is an isomorphism away from $Z$, we can make the following definition.

\begin{defn}\label{defn:weighted_blow_up_along_one_irred_component}
Let $X$ be a normal variety and $D_1,\ldots, D_r$ be prime divisors on $X$. Let $Z$ be an irreducible component of $\bigcap_{i=1}^r D_i$. Assume $(X,D)$ is log smooth in a neighborhood $U$ of $Z$ and $U$ does not intersect any other irreducible component of $\bigcap_{i=1}^r D_i$. Let $d_1,\ldots,d_r$ be a sequence of positive integers. We define \emph{the weighted blow-up of $X$ along $Z$ with respect to $\{D_i: 1\leq i\leq r\}$ and weights $\{d_i:1\leq i\leq r\}$} to be the following scheme $X'$:
\begin{itemize}
\item Over $U$, $X'$ is the weighted blow-up of $X$ with respect to $\{D_i: 1\leq i\leq r\}$ and weights $\{d_i:1\leq i\leq r\}$ as in Definition~\ref{defn:weighted_blow_up_wrt_snc_divisors}.
\item Away from $Z$, $X'$ is isomorphic to $X$. In fact, if $E$ is the exceptional divisor of $X'\to X$, then $(X'- E)\cong (X- Z)$.
\end{itemize}
\end{defn}

\bibliographystyle{alpha}
\bibliography{main}

\end{document}